\PassOptionsToPackage{dvipsnames}{xcolor}
\documentclass[leqno]{siamart251104}
\usepackage{mydef}
\usepackage{amssymb}
\usepackage{graphicx}
\usepackage{xspace}
\usepackage{derivative}
\usepackage{bm,bbm}
\usepackage[numbers,sort,compress]{natbib}
\usepackage{dsfont}
\usepackage{enumitem}
\usepackage{subcaption}
\usepackage{booktabs}
\usepackage{algorithm}
\usepackage{algpseudocode}
\usepackage{comment}
\usepackage[normalem]{ulem}
\usepackage{siunitx}
\newcommand{\myhide}[1]{{}}

\usepackage{graphics}
\usepackage{wrapfig}
\usepackage{float}
\usepackage[percent]{overpic}
\usepackage{varwidth}
\usepackage{tikz}
\usepackage{pgfplots}
\usepackage{adjustbox}
\usetikzlibrary{arrows,matrix,positioning,fit}
\usetikzlibrary{shapes,positioning}
\usetikzlibrary{backgrounds}
\usepackage{tikz-layers}
\usepgfplotslibrary{fillbetween}
\usetikzlibrary{intersections}
\pgfplotsset{compat=1.14}
\usepgfplotslibrary{colorbrewer}
\usepgfplotslibrary{patchplots}
\usepgfplotslibrary[colorbrewer]
\usetikzlibrary{pgfplots.colorbrewer}
\usetikzlibrary[pgfplots.colorbrewer]
\usepgfplotslibrary{units}
\usetikzlibrary{spy}
\usepackage{pgfplotstable}
\usepackage{arrayjobx}
\usepackage{subcaption}

\usetikzlibrary{spy}
\usetikzlibrary{calc}
\usepackage{soul}

\DeclareMathOperator*{\conv}{conv}

\colorlet{PlotColor1}{Spectral-A}
\colorlet{PlotColor3}{Spectral-L}
\colorlet{PlotColor2}{Spectral-O}
\colorlet{PlotColor4}{Spectral-D}
\colorlet{PlotColor5}{black}

\begin{document}

\newcommand\footnotemarkfromtitle[1]{%
  \renewcommand{\thefootnote}{\fnsymbol{footnote}}%
  \footnotemark[#1]%
  \renewcommand{\thefootnote}{\arabic{footnote}}}

\newcommand{\TheTitle}{Well-balanced Second-Order Approximation of the Compressible Atmospheric Euler Equations}
\newcommand{\TheAuthors}{C. Farris, M. Maier, E.~J. Tovar}

\headers{Well-balanced IDP scheme for atmospheric Euler}{\TheAuthors}

\title{{\TheTitle}\thanks{\funding{
  The work of CF was supported by the Laboratory Directed Research and
  Development (LDRD) program at Los Alamos National Laboratory under
  project number 20251079ER. Research conducted at Los Alamos National
  Laboratory is done under the auspices of the National Nuclear Security
  Administration of the U.S. Department of Energy under Contract No.
  89233218CNA000001. MM was partially supported by the National Science
  Foundation under grant DMS-2045636 and by the Air Force Office of
  Scientific Research, USAF, under grant/contract number FA9550-23-1-0007.
  The Los Alamos unlimited release number is LA-UR-26-24992.
  }}}

\author{
  Crystal Farris\footnotemark[2]
  \and
  Matthias Maier\footnotemark[2]
  \and
  Eric J. Tovar\footnotemark[3]
}

\maketitle

\renewcommand{\thefootnote}{\fnsymbol{footnote}}
\footnotetext[2]{Department of Mathematics, Texas A\&M University}
\footnotetext[3]{Xcimer Energy Corporation}

\renewcommand{\thefootnote}{\arabic{footnote}}
\begin{abstract}
  We introduce a second-order approximation to the compressible atmospheric
  Euler equations with gravity that is invariant domain preserving and
  well-balanced with respect to stratified rest states, including
  isentropic and isothermal equilibria. The approximation is built
  upon discrete auxiliary states derived from a hydrostatic reconstruction
  of the density and potential temperature. These auxiliary states, together with an affine shift of
  the numerical state, provide local bounds needed for maintaining
  well-balancing and invariant domain preserving properties of the method.
  The numerical method is then verified and validated with analytic
  solutions, well-balancing tests, and typical benchmark problems for
  atmospheric flows.
\end{abstract}



\section{Introduction}\label{sec:introduction}
In this work, we consider the atmospheric Euler equations in the
\emph{potential temperature formulation} with a gravitational source
term. This formulation allows a direct comparison of the temperature
at different altitudes since the potential temperature is invariant under
adiabatic changes in pressure \cite{curry2015saturated}. This system
provides a framework for simulating complex fluid flow in the atmosphere
under the influence of gravity \cite{artiano2025structurepreserving,
chertock2023topography}. The mathematical model is a foundation for numerical weather
prediction and climate models, where accurately representing fluid dynamics
across a wide range of spatial scales is essential
\cite{GIRFOGLIO2025106510}.

Although large-scale atmospheric phenomena are largely hydrostatic in
nature, gravity waves and other important mesoscale processes are
fundamentally time dependent and, thus,
non-hydrostatic~\cite{lou1998mesoscale}. Accurately capturing these effects
therefore requires moving beyond hydrostatic models towards full
time-dependent formulations based on the Navier Stokes or Euler equations
\cite{giraldo2008mesoscale, marras2016review}. Importantly, even within
such fully non-hydrostatic frameworks, it remains essential to preserve the
hydrostatic equilibrium as a \emph{steady state at rest}. This property is
crucial for accurately capturing atmospheric dynamics, as the balance
between the pressure gradient and gravitational forces fundamentally
determines how motion in the atmosphere develops and evolves. The
difficulty of maintaining such a \emph{well-balancing} property is
compounded by the need to maintain important physical invariants of the
fluid throughout the computation, such as positivity of density and
temperature~\cite{zhang2021high}. Consequently, the development of robust,
efficient, and physically consistent numerical schemes remains an active
area of research~\cite{artiano2025structurepreserving, chertock2023topography, navasmontilla2024turbulent}.

The objective of this paper is to introduce a discretization of the
atmospheric Euler equations in potential temperature formulation that is
\emph{invariant domain preserving} (IDP) and \emph{well-balanced with
respect to stratified rest states}. In other words, we aim to produce a scheme that
ensures quantities like density and potential temperature remain in their
physically relevant ranges, and that steady state solutions are preserved.
The discretization is based on continuous finite elements and is formulated
for general unstructured meshes. To accomplish this, we develop an
algebraic framework for constructing discrete auxiliary states that
simultaneously preserve the hydrostatic equilibria and stay within the
admissible set of the PDE. We show that using a hydrostatic reconstruction
of the numerical density, potential temperature, and numerical state (Def.~\ref{def:star_states})
combined with an appropriate affine shift (Eqn.~\ref{eq:shifted_bar_state})
allows us to construct a low-order approximation that is invariant domain
preserving and well-balanced. This approximation is shown to be a
consistent approximation of the underlying PDE (Prop.~\ref{prop:consistency})
to $\calO(h)$ accuracy, where $h$ is the spatial mesh size. We show that
the first-order scheme can be algebraically re-written as a convex
combination of the numerical state and the novel auxiliary states with the
affine shift (Lemma~\ref{lemma:convex_combination}). Through this convex
combination, we extract local bounds (Sec.~\ref{sec:local_bounds}) which
are used as a safeguard for our higher-order approximation without
affecting the well-balancing property. The final methodology is shown to be invariant domain preserving and well-balanced with
respect to stratified rest states (Theorem~\ref{thm:final_theorem}), and we demonstrate numerically the method is $\calO(h^2)$ accurate. We verify the methodology with
convergence tests, well-balancing problems and validate with standard
benchmarks for atmospheric flows.

\subsection{Related works}

Recent numerical methods proposed for approximating Euler equations with a
gravitational source in the atmosphere include discontinuous Galerkin and
spectral element methods
\cite{artiano2025structurepreserving,francis2004spec,
giraldo2008mesoscale, liu2025dg}. Various finite difference and finite volume
discretizations, as well as WENO and TENO reconstructions
\cite{chandrashekar2015finitevolume, marras2016review, navasmontilla2023atmospheric,
navasmontilla2024turbulent} have been proposed. Importantly, some of these
works additionally incorporate well-balancing in their numerical method,
see \cite{chertock2023topography, navasmontilla2023atmospheric,
li2021gravitational, du2024gravity, artiano2025structurepreserving, berthon2025entropy, liu2025dg}. A high-order DG method for the atmospheric Euler
equations in the energy formulation was discussed in~\cite{zhang2021high}
which is positivity-preserving and well-balanced for a general equilibrium.
Somewhat related developments for the shallow water equations provide
further insight into structure-preserving discretizations with an
additional emphasis on maintaining an invariant domain preserving property
\cite{azerad2017shallow, guermond2025erk}. To the best of the authors'
knowledge, a rigorous enforcement of both invariant domain preservation and
well-balancing with respect to stratified rest states, covering both
isentropic and isothermal equilibria, for the atmospheric Euler equations in the potential
temperature formulation has not been previously discussed.

\subsection{Outline of the paper}
The paper is organized as follows. We introduce the model problem, \ie the
atmospheric Euler equations in Section~\ref{sec:model_problem}, as well as
important thermodynamic and mathematical properties. In
Section~\ref{sec:approximation} we give brief details for the finite
element setting considered in this work. Then, in
Section~\ref{sec:star_states}, we define what it means to be well-balanced
in the discrete setting and define the hydrostatic reconstruction. Then, in
Section~\ref{sec:low_order}, we introduce the low-order method and show
that it can be written as a convex combination of the  auxiliary states
based on the hydrostatic states and a novel affine shift in the style
of~\citep{guermond2025erk}. The well-balancing and IDP properties follow
naturally from the definition of the shifted auxiliary states and
hydrostatic reconstruction. This convex combination allows us to extract
local bounds used to correct our higher-order approximation in
Section~\ref{sec:high_order} to ensure it maintains the IDP properties.
Finally, we present numerical illustrations in Section \ref{sec:results}
where we verify and validate the numerical method.  For the sake of
completeness, we also include the full solution to the Riemann problem for
this set of equations in Appendix~\ref{sec:wave_structure}
and~\ref{sec:riemann_problem}.

\section{Preliminaries}
\label{sec:model_problem}
Let $D\subset\polR^d$ be a polygonal domain where $d\geq 1$ is the spatial
dimension. Let $\rho$ denote the fluid density, $\bbm$ the momentum, and
$E$ the total energy of the system. Under the assumption of an adiabatic
atmosphere (\ie no heat exchange in the atmosphere), the Euler equations
supplemented with a linear gravity source term can be written as follows:
\begin{equation}\label{euler}
  \begin{cases}
    \begin{aligned}
      \partial_t\rho + \DIV(\bbm) &= 0,
      \\
      \partial_t\bbm + \DIV\big(\rho^{-1}\bbm \otimes \bbm +
      p\polI_d\big) &= -\rho g \mathbf{\hat{e}}_{d},
      \\
      \partial_t E + \DIV(\bv (E + p))&=-\rho g (\bv\SCAL \mathbf{\hat{e}}_{d}),
    \end{aligned}
  \end{cases}
\end{equation}
where $\polI_d$ is the $d\times d$ dimensional identity matrix and
$\mathbf{\hat{e}}_{d}$ is the unit-normal in the direction of $x_d$. Here,
$x_d$ is the $d$-th component of the cartesian coordinate $\bx\eqq(x_1,
\dots, x_d)$; it holds true that $\GRAD x_d = \mathbf{\hat{e}}_{d}$. Here,
$\bbm=\rho\bv$, and $g$ is the gravitational acceleration. Let $e\eqq 1/\rho\left(E -
\tfrac12\rho\|\bv\|^2\right)$ be the specific internal energy. We assume
that the pressure is given by the ideal gas equation of state: $p=(\gamma -
1) \rho e$. With this assumption, the temperature of the system is given by
$T=\frac{e}{c_v}$. Here, $\gamma\eqq \frac{c_p}{c_v}$ is the ratio of
specific heat capacities.

For atmospheric flow applications, it is common to recast the compressible
Euler equations above in the \textit{potential temperature} formulation.
This formulation is a result of the adiabatic assumption and the first law
of thermodynamics \cite{curry2015saturated}. The implication is that the energy equation
can be swapped out for the evolution of the specific entropy, or
equivalently, the potential temperature \cite{lilly1962derivation}. Using the ideal gas
equation of state and the first law of thermodynamics, the potential
temperature can be defined as follows \cite{ghosh2016benchmarks}:
\begin{equation}\label{eq:p_t}
    \theta \eqq T\bigg(\frac{P_0}{p}\bigg)^\frac{R}{c_p} ,
\end{equation}
where $P_0$ is a given reference pressure value and $R$ is the specific
gas constant for dry air which can be expressed as: $R=c_v(\gamma - 1)$. Now, let
$\bu(\bx,t) \eqq (\rho,\bbm, \rho\theta)\tr$ denote the state vector of
conserved quantities. Then, given some initial data $\bu_0(\bx) \eqq
\big(\rho_0, \bbm_0, (\rho\theta)_0\big)(\bx)$ at initial time $t_0$, we
seek an approximation to the following system:
\begin{equation}\label{model_problem}
  \begin{cases}
    \begin{aligned}
      \partial_t\rho + \DIV(\bbm) &= 0,
      \\
      \partial_t\bbm + \DIV\big(\rho^{-1}\bbm \otimes \bbm +
      p(\bu)\polI_d\big) &= -\rho g \mathbf{\hat{e}}_{d},
      \\
      \partial_t (\rho \theta) + \DIV(\bv \rho \theta)&=0.
    \end{aligned}
  \end{cases}
\end{equation}
In this formulation, the pressure is given by:
\begin{equation}\label{eq:equation_of_state}
  p(\bu) \eqq C_{\text{eos}}\big(\rho\theta\big)^\gamma, \text{ with }
  C_{\text{eos}}:= P_0\Big(\frac{R}{P_0}\Big)^\gamma,
\end{equation}
For notation purposes, we define the short-hand notation for the hyperbolic flux as: $\bm{f}(\bu)\eqq \left(\bbm,\rho^{-1}\bbm\otimes\bbm + p(\bu)\polI_d,\bv\rho\theta\right)\tr$.
We also introduce the pressureless gas flux: $\bm{g}(\bu)\eqq \left(\bbm,\rho^{-1}\bbm\otimes\bbm,\bv\rho\theta\right)\tr$ which will be used in the finite element discretization in \S\ref{sec:FEM}.

\begin{remark}
While the potential temperature and total energy formulations of the compressible Euler equations are equivalent at the continuous level, this equivalence does not carry over to the weak form, as the pointwise algebraic and differential manipulations used to pass between the two are not possible in the distributional sense. As a consequence, the potential temperature form will not conserve total energy in general, and vice versa.
\end{remark}

\subsection{Structural properties}
We briefly summarize relevant structural properties of the atmospheric
Euler equations in potential temperature formulation~\eqref{model_problem}
that will be used in the numerical approach for constructing a robust
approximation technique.
\begin{definition}[Admissible set, invariant set, and invariant domain]
  \label{def:admissible_set}
  The admissible set of the atmospheric Euler equations \eqref{model_problem}
  reads
  \begin{equation}
    \mathcal{A} := \{\mathbf{u} = (\rho,\mathbf{m},\rho\theta)^T
    \in\R^{d+2} : \rho > 0,\,\theta > 0\}.
  \end{equation}
  Furthermore, we call any subset $\mathcal{B}\subset\mathcal{A}$ an
  \emph{invariant set}, provided there holds that the Riemann average
  $\bar\bu(t)$ of the solution to the Riemann problem with (projected) left
  and right states $\bu_L$, $\bu_R\in\mathcal{B}$ lies in $\mathcal{B}$
  again, viz., $\bar\bu(t)\in\mathcal{B}$ \cite{guermond2016invariant}.
  A convex invariant set $\mathcal{B}$ is called an
  \emph{invariant domain} for a numerical scheme $S_h$ if $S_h(\bu_h)\in\mathcal{B}$
  for every discrete state $\bu_h\in\mathcal{B}$~\cite[Defs.~2.1,
  2.2]{guermond2018second}.
\end{definition}
An invariant domain $\mathcal{B}$ is described by a collection of
\emph{quasi-concave} functions \cite{guermond2016invariant}, say
$\Psi^l(\bu)$, $l=1,\ldots,k$ and setting
\begin{equation}
  \label{eq:invariant_domain}
  \calB\eqq\{\bu\in\polR^m : \Psi^l(\bu)\,>\,0,\,l=1,\dots,k\}.
\end{equation}
The concrete local bounds used in the numerical approach are discussed in
Section~\ref{sec:local_bounds}. The introduction of a gravitational source
term in~\eqref{model_problem} requires a careful numerical construction to
preserve so-called rest states. To this end we introduce the following definition.
\begin{definition}[Problem at Stratified Equilibrium]
  \label{def:problem_at_rest}
  Let $\bu(\bx,t)$ be a solution to~\eqref{model_problem}. Then, we say
  $\bu(\bx, t)$ at time $t$ is at \emph{stratified equilibrium} if the
  following condition holds true for some constants $\theta_0>0$ and $N\ge0$:
  \begin{equation}
    \label{general_problem_at_rest}
    \begin{cases}
      \begin{aligned}
        \bu(\bx,t) &\text{ does not depend on
        }x_1,\ldots,x_{d-1},
        \\
        \bv\cdot\mathbf{\hat{e}}_d &= 0,
        \\
        \theta &= \theta_0\exp(N^2x_d/g),
        \\
        \GRAD p(\bu) &= -\rho g\,\mathbf{\hat{e}}_d.
      \end{aligned}
    \end{cases}
  \end{equation}
  Note that under appropriate boundary conditions---such as periodic
  boundary conditions in $x_1,\ldots,x_{d-1}$ directions---this implies
  $\partial_t\bu(\bx,t)=\bzero$. Here, $N$ is the
  Brunt--V\"ais\"al\"a frequency of the
  stratification~\cite{giraldo2008mesoscale}.
\end{definition}
We make the following observation.
\begin{lemma}[Stratified equilibrium condition]
  \label{algebraic_balance}
  Let $\bu(\bx,t)$ be at stratified equilibrium at time $t$. Then,
  using the equation of state~\eqref{eq:equation_of_state}, the last
  condition in \eqref{general_problem_at_rest} can be recast as a matching
  condition on the density and potential temperature:
  \begin{equation}
    \label{eq:problem_at_rest}
    \theta = \theta_0\exp(N^2x_d/g),
    \qquad
    (\rho\theta)^{\gamma-1} + \frac{g}{\tilde{c}}\int_0^{x_d}\frac{1}{\theta}\diff s = \mathrm{const.},
    \qquad
    \tilde{c} := \frac{C_{eos}\gamma}{\gamma-1}.
  \end{equation}
\end{lemma}
A common example of a state at stratified equilibrium with $N=0$ is \cite{GROSHEINTZLAVAL2019324,ghosh2016benchmarks}
\begin{equation}
  \label{adiabatic_eq}
  \begin{cases}
    \begin{aligned}
      \bv\cdot\mathbf{\hat{e}}_d = 0,
      \quad
      \bv\cdot\mathbf{\hat{e}}_1, ...,\bv\cdot\mathbf{\hat{e}}_{d-1} = const.
      \quad \theta = 300,
      \\[0.5em]
      T = \theta - \frac{g x_d}{\gamma c_{\text{v}}},
      \quad
      p = P_0 \Big(\frac{T}{\theta}\Big)^{\gamma/(\gamma-1)},\quad \rho =
      \frac{p}{R T}.
    \end{aligned}
  \end{cases}
\end{equation}

\begin{remark}[Isentropic versus adiabatic equilibrium]
  The rest state with $N=0$, \ie $\theta=\mathrm{const.}$, is commonly referred to as
  ``adiabatic equilibrium'' in the literature \cite{li2021gravitational, navasmontilla2023atmospheric}. Though, it is
  true that an isentropic equilibrium also assumes an adiabatic atmosphere,
  it does not need to hold true that an adiabatic atmosphere is also
  isentropic. For this reason, we prefer to refer to conditions
  \eqref{eq:problem_at_rest} with $N=0$ and \eqref{adiabatic_eq} as \emph{isentropic}
  conditions. In this case, \eqref{eq:problem_at_rest} reduces to
  $\rho^{\gamma-1} + \frac{g}{\tilde{c}\theta^\gamma}x_d = \mathrm{const.}$
\end{remark}

\begin{remark}[Isothermal equilibrium]~\label{rem:isothermal}
  A special case of Definition~\ref{def:problem_at_rest} is the
  \emph{isothermal equilibrium} with constant temperature $T$ and
  $N^2=g^2/(c_pT)$; see,
  e.\,g.,~\cite{navasmontilla2023atmospheric,li2021gravitational}.
  We show in Section~\ref{isothermal_test} that the scheme proposed
  here preserves it exactly.
\end{remark}

\section{Well-balanced, invariant-domain preserving approximation}
\label{sec:approximation}
We now introduce a fully discrete approximation technique for
\eqref{model_problem} based on continuous finite elements. We introduce the
concept of \emph{discrete well-balancing} in Section~\ref{sec:star_states},
\ie the approximation maintains the problem-at-rest property
(Definition~\ref{def:problem_at_rest}) on the discrete level. To this end,
we construct discrete auxiliary states which encode the hydrostatic
equilibrium. These states, combined with an affine shift, allow us to
derive a first-order method in Section~\ref{sec:low_order} which is a
convex combination of said states. The methodology is proved to be
consistent, well-balanced and invariant-domain preserving. This convex
combination further allows us to extract bounds used as a safeguard for our
higher-order method in Section~\ref{sec:high_order}. The final algorithm is
formally second-order accurate in space.

\subsection{Continuous finite element setting}\label{sec:FEM}
The spatial approximation adopted in this paper is based upon the invariant
domain preserving methodology \citep{guermond2016invariant}. The low order
method generalizes the algorithm presented in~\citep[p.\,163]{lax1954}, as
it is discretization independent. Full details of this method for
continuous and discontinuous finite elements, and finite volume
approximations are given in~\citep{guermond2019idp}. The approaches for
the compressible Euler equations are proposed in
~\citep{guermond2018second, clayton2022invariant, clayton2023robust,
clayton2025approximation}. For the sake of completeness, we briefly
summarize the finite element setting that will be used for our
approximation.

Let $D\subset\mathbb{R}^d$ be a domain with polygonal boundary, discretized
with a shape and form regular mesh $\mathcal T_h$, where $h$ is the mesh size, consisting of
simplicial cells (triangles for d=2, tetrahedra for d=3), or quadrilateral (for $d=2$) or hexahedral (for $d=3$) cells $K$, covering
the domain $D$. We work with the finite element space
\begin{equation}
    \mathcal{P}(\mathcal{T}_h) := \{v \in C^0(D;\mathbb{R}) \mid v|_K \circ T_K \in \hat{\mathcal{P}},\ \forall K\in\mathcal{T}_h\},
\end{equation}
where $T_K:\hat{K}\to K$ is the (possibly non-affine) geometric bijection
from the reference element to the current element $K$, and $\hat{\mathcal{P}}=\mathbb{P}_1$ on simplices and $\hat{\mathcal{P}}=\mathbb{Q}_1$ on quadrilaterals and hexahedra. See \cite{azerad2017shallow} for full details. Let $\{\varphi_i\}_{i\in\calV}$ denote the nodal Lagrange
basis of the continuous finite element space constructed with (bi-/tri-)
linear elements on the mesh $\mathcal T_h$. Here $\calV$ is an index
enumerating the spatial degrees of freedom. Let $\bx_i$ denote the
collocation point associated with the $i$th degree of freedom, and
introduce a short notation for the $d$th component, $z_i:=(\bx_i)_d$.

   In order to ensure a well-behaved discrete approximation in the $d$th spatial direction,
and to guarantee well-balancing in the presence of a constant horizontal background
velocity, we impose the following structural assumptions on the mesh $\mathcal{T}_h$, which we will make use of individually.

\begin{assumption}[Balanced vertical stencil]
  \label{ass:balanced_stencil}
  For every interior node $i \in \calV$, there is a bijection
  $j\mapsto j'$ of $\I(i)$ such that
    \begin{gather}
      \label{eq:mesh_assumption}
      \bx_{j'}-\bx_i = -(\bx_j-\bx_i) + \mathcal{O}(h^2),
      \qquad
      \mathbf{c}_{ij'} = -\cij + \mathcal{O}(h^d),
    \end{gather}
  \ie the stencil is point symmetric up to higher-order terms.
\end{assumption}
\begin{assumption}[Horizontal row cancellation]
  \label{ass:row_cancellation}
  The nodes lie on finitely many horizontal levels $z_1 < \cdots < z_L$,
  and for every $i \in \calV$ and every level $\ell\in\{1,\dots, L\}$,
    \begin{equation}
      \sum_{\substack{j \in \I(i) \\ z_j = z_\ell}} \cij
      \quad\text{is parallel to } \hat{e}_d.
    \end{equation}
\end{assumption}
Assumption~\ref{ass:balanced_stencil} is mild and will be essential in establishing the consistency
property later. Assumption~\ref{ass:row_cancellation} is essential for well-balancing with a non-zero horizontal background velocity $\bV^{\mathrm{eq}}$. We provide some examples of meshes that satisfy these assumptions in Figure~\ref{fig:mesh_assumptions}.
We note that Assumption~\ref{ass:row_cancellation} is not needed for when the flow is truly at rest ($\bV\equiv \bm{0})$ and we preserve the at-rest well-balanced property for general meshes.

\begin{figure}
\centering
\begin{subfigure}[t]{0.23\textwidth}
\centering
\begin{tikzpicture}[scale=1.6]
  \foreach \k in {0,1,2,3}{
    \foreach \i in {0,1,2,3}{
      \pgfmathsetmacro{\xbase}{\i*0.4}
      \pgfmathsetmacro{\zbase}{\k*0.37}
      \pgfmathsetmacro{\zshift}{0.12*sin((\i*0.4)*180/1.2)*\k*0.37*(1-\k*0.37/1.1)}
      \coordinate (n\i\k) at ({\xbase},{\zbase+\zshift});
    }
  }
  \foreach \k in {0,1,2,3}{
    \foreach \i in {0,1,2}{
      \pgfmathtruncatemacro{\ii}{\i+1}
      \draw[gray!50] (n\i\k)--(n\ii\k);
    }
  }
  \foreach \i in {0,1,2,3}{
    \foreach \k in {0,1,2}{
      \pgfmathtruncatemacro{\kk}{\k+1}
      \draw[gray!50] (n\i\k)--(n\i\kk);
    }
  }
  \foreach \k in {0,1,2,3}{
    \foreach \i in {0,1,2,3}{
      \filldraw[black] (n\i\k) circle (0.8pt);
    }
  }
  \draw[->,thick] (1.35,1.05) -- (1.35,0.65) node[right,font=\tiny]{$g$};
\end{tikzpicture}
\caption*{\footnotesize Terrain-following\\ $Q_1$ (smooth)}
\end{subfigure}
\begin{subfigure}[t]{0.23\textwidth}
\centering
\begin{tikzpicture}[scale=1.6]
  \foreach \k in {0,...,3}{
    \foreach \i in {0,...,3}{
      \coordinate (w\i-\k) at ({1.2*\i/3+0.08*sin(180*\k/3)},
                                {\k/3+0.08*sin(360*\i/3)});
    }
  }
  \foreach \k in {0,...,3}{
    \foreach \i in {1,...,3}{
      \pgfmathtruncatemacro{\ii}{\i-1}
      \draw[gray!50] (w\ii-\k)--(w\i-\k);
      \draw[gray!50] (w\k-\ii)--(w\k-\i);
    }
  }
  \foreach \k in {0,...,3}{
    \foreach \i in {0,...,3}{
      \filldraw[black] (w\i-\k) circle (0.7pt);
    }
  }
  \draw[->,thick] (1.45,0.95)--(1.45,0.55)
    node[right,font=\tiny]{$g$};
\end{tikzpicture}
\caption*{\footnotesize Smoothly warped $Q_1$\\ ($3\times3$ cells)}
\end{subfigure}
\begin{subfigure}[t]{0.23\textwidth}
\centering
\begin{tikzpicture}[scale=1.6]
  \foreach \k in {0,...,6}{
    \foreach \i in {0,...,6}{
      \coordinate (w\i-\k) at ({1.2*\i/6+0.08*sin(180*\k/6)},
                                {\k/6+0.08*sin(360*\i/6)});
    }
  }
  \foreach \k in {0,...,6}{
    \foreach \i in {1,...,6}{
      \pgfmathtruncatemacro{\ii}{\i-1}
      \draw[gray!50] (w\ii-\k)--(w\i-\k);
      \draw[gray!50] (w\k-\ii)--(w\k-\i);
    }
  }
  \foreach \k in {0,...,6}{
    \foreach \i in {0,...,6}{
      \filldraw[black] (w\i-\k) circle (0.5pt);
    }
  }
  \draw[->,thick] (1.45,0.95)--(1.45,0.55)
    node[right,font=\tiny]{$g$};
\end{tikzpicture}
\caption*{\footnotesize Smoothly warped $Q_1$\\ ($6\times6$ cells)}
\end{subfigure}
\begin{subfigure}[t]{0.23\textwidth}
\centering
\begin{tikzpicture}[scale=1.6]
  \foreach \k in {0,...,12}{
    \foreach \i in {0,...,12}{
      \coordinate (w\i-\k) at ({1.2*\i/12+0.08*sin(180*\k/12)},
                                {\k/12+0.08*sin(360*\i/12)});
    }
  }
  \foreach \k in {0,...,12}{
    \foreach \i in {1,...,12}{
      \pgfmathtruncatemacro{\ii}{\i-1}
      \draw[gray!50] (w\ii-\k)--(w\i-\k);
      \draw[gray!50] (w\k-\ii)--(w\k-\i);
    }
  }
  \foreach \k in {0,...,12}{
    \foreach \i in {0,...,12}{
      \filldraw[black] (w\i-\k) circle (0.3pt);
    }
  }
  \draw[->,thick] (1.45,0.95)--(1.45,0.55)
    node[right,font=\tiny]{$g$};
\end{tikzpicture}
\caption*{\footnotesize Smoothly warped $Q_1$\\ ($12\times12$ cells)}
\end{subfigure}

\vspace{0.4em}
{\footnotesize Assumption~\ref{ass:balanced_stencil} -- Balanced vertical stencil}
\vspace{1.0em}

\begin{subfigure}[t]{0.23\textwidth}
\centering
\begin{tikzpicture}[scale=1.6]
  \def\xa{0}\def\xb{0.45}\def\xc{0.85}\def\xd{1.2}
  \def\za{0}\def\zb{0.35}\def\zc{0.72}\def\zd{1.1}
  \foreach \z in {\za,\zb,\zc,\zd}{
    \draw[gray!50] (\xa,\z) -- (\xd,\z);
  }
  \foreach \x in {\xa,\xb,\xc,\xd}{
    \draw[gray!50] (\x,\za) -- (\x,\zd);
  }
  \foreach \x in {\xa,\xb,\xc,\xd}{
    \foreach \z in {\za,\zb,\zc,\zd}{
      \filldraw[black] (\x,\z) circle (0.8pt);
    }
  }
  \draw[->,thick] (1.35,1.05) -- (1.35,0.65) node[right,font=\tiny]{$g$};
\end{tikzpicture}
\caption*{\footnotesize Rectilinear $Q_1$\\ non-uniform}
\end{subfigure}
\begin{subfigure}[t]{0.23\textwidth}
\centering
\begin{tikzpicture}[scale=1.6]
  \def\za{0}\def\zb{0.37}\def\zc{0.74}\def\zd{1.1}
  \def\skewa{0.0}\def\skewb{0.12}\def\skewc{0.07}\def\skewd{0.0}
  \foreach \x/\sk in {0/\skewa, 0.4/\skewa, 0.8/\skewa, 1.2/\skewa}{
    \draw[gray!50]
      ({\x+\skewa},\za) -- ({\x+\skewb},\zb)
                        -- ({\x+\skewc},\zc)
                        -- ({\x+\skewd},\zd);
  }
  \foreach \z/\sk in {\za/\skewa, \zb/\skewb, \zc/\skewc, \zd/\skewd}{
    \draw[gray!50] ({0+\sk},\z) -- ({1.2+\sk},\z);
    \foreach \x in {0, 0.4, 0.8, 1.2}{
      \filldraw[black] ({\x+\sk},\z) circle (0.8pt);
    }
  }
  \draw[->,thick] (1.5,1.05) -- (1.5,0.65) node[right,font=\tiny]{$g$};
\end{tikzpicture}
\caption*{\footnotesize Flat layers,\\ skewed cols.\ $Q_1$}
\end{subfigure}
\begin{subfigure}[t]{0.23\textwidth}
\centering
\begin{tikzpicture}[scale=1.6]
  \def\za{0}\def\zb{0.37}\def\zc{0.74}\def\zd{1.1}

  \foreach \z in {\za,\zb,\zc,\zd}{
    \draw[gray!50] (0,\z)--(1.2,\z);
  }
  \foreach \xi/\xii in {0/0.4, 0.4/0.8, 0.8/1.2}{
    \foreach \zi/\zii in {\za/\zb, \zb/\zc, \zc/\zd}{
      \draw[gray!50] (\xi,\zi)--(\xii,\zii);
      \draw[gray!50] (\xi,\zi)--(\xi,\zii);
    }
  }
  \foreach \zi/\zii in {\za/\zb, \zb/\zc, \zc/\zd}{
    \draw[gray!50] (1.2,\zi)--(1.2,\zii);
  }

  \foreach \z in {\za,\zb,\zc,\zd}{
    \foreach \x in {0,0.4,0.8,1.2}{
      \filldraw[black] (\x,\z) circle (0.8pt);
    }
  }

  \draw[->,thick] (1.35,1.05) -- (1.35,0.65)
    node[right,font=\tiny]{$g$};
\end{tikzpicture}
\caption*{\footnotesize Structured $P_1$\\ flat layers}
\end{subfigure}
\begin{subfigure}[t]{0.23\textwidth}
\centering
\begin{tikzpicture}[scale=1.6]
  \def\za{0}\def\zb{0.37}\def\zc{0.74}\def\zd{1.1}

  \foreach \x/\xL/\xR in {0.2/0/0.4, 0.6/0.4/0.8, 1.0/0.8/1.2}{
    \draw[gray!50] (\xL,\za) -- (\x,\zb);
    \draw[gray!50] (\xR,\za) -- (\x,\zb);
  }
  \draw[gray!50] (0,\za) -- (1.2,\za);

  \foreach \x/\xL/\xR in {0.2/0/0.4, 0.6/0.4/0.8, 1.0/0.8/1.2}{
    \draw[gray!50] (\xL,\zc) -- (\x,\zb);
    \draw[gray!50] (\xR,\zc) -- (\x,\zb);
  }
  \draw[gray!50] (0,\zc) -- (1.2,\zc);

  \foreach \x/\xL/\xR in {0.2/0/0.4, 0.6/0.4/0.8, 1.0/0.8/1.2}{
    \draw[gray!50] (\xL,\zc) -- (\x,\zd);
    \draw[gray!50] (\xR,\zc) -- (\x,\zd);
  }
  \draw[gray!50] (0.2,\zd) -- (1.0,\zd);
  \draw[gray!50] (0.2,\zb) -- (1.0,\zb);

  \foreach \x in {0, 0.4, 0.8, 1.2}{
    \filldraw[black] (\x,\za) circle (0.8pt);
    \filldraw[black] (\x,\zc) circle (0.8pt);
  }
  \foreach \x in {0.2, 0.6, 1.0}{
    \filldraw[black] (\x,\zb) circle (0.8pt);
    \filldraw[black] (\x,\zd) circle (0.8pt);
  }

  \draw[->,thick] (1.35,1.05) -- (1.35,0.65)
    node[right,font=\tiny]{$g$};
\end{tikzpicture}
\caption*{\footnotesize Flat layers,\\ staggered $P_1$}
\end{subfigure}

\vspace{0.4em}
{\footnotesize Assumption~\ref{ass:row_cancellation} -- Horizontal row cancellation}
\vspace{1.0em}
\caption{Example meshes for each mesh assumption.
  \textit{Top row:} meshes satisfying Assumption~\ref{ass:balanced_stencil} (balanced vertical stencil).
  \textit{Bottom row:} meshes satisfying Assumption~\ref{ass:row_cancellation} (horizontal row cancellation); the horizontal
  contributions of $c_{ij}$ cancel exactly within each row.}
\label{fig:mesh_assumptions}
\end{figure}

At time $t^n$ a discrete approximation $\bu_h(t^n,\bx)$ to the solution of
\eqref{model_problem} is constructed by setting
$\bu_h(\bx)\eqq\sum_{i\in\calV}\bsfU_i\upn\varphi_i(\bx)$, where the
collocated state vector is $\bsfU_i\upn\eqq(\varrho_i\upn, \bsfM_i\upn,
\Theta_i\upn)\tr$. Furthermore, we introduce discrete primitive quantities
as follows: $\bsfV_i\upn\eqq\bsfM_i\upn / \varrho_i\upn$ and
$\vartheta_i\upn\eqq\Theta_i\upn / \varrho_i\upn$. We introduce a lumped
mass matrix $\polM\upL=(m_i)_i$ and a consistent mass matrix
$\polM\upH=(m_{ij})_{ij}$ with entries
\begin{align*}
  m_i\eqq\int_D\varphi_i(\bx)\diff \bx,
  \qquad
  m_{ij}\eqq\int_D\varphi_i(\bx)\varphi_j(\bx)\diff\bx.
\end{align*}
Furthermore, a vector valued matrix
\begin{align*}
  \cij\eqq\int_D\varphi_i(\bx)\GRAD\varphi_j(\bx)\diff\bx\quad\in\,\polR^d
\end{align*}
is introduced for approximating the divergence operator. Let
$\calI(i)\subset\calV$ denote the stencil of degree of freedom $i$, \ie the
set of all degrees of freedom $j$ coupling with $i$:
$\calI(i)\eqq\big\{j\in\calV \;:\;m_{ij}\not=0\big\}$. As a final requisite
we assume that the basis $\{\varphi_i\}_{i\in\calV}$ forms a partition of
unity, \ie $\sum_{i\in\calV}\,\varphi_i(\bx)\,\equiv\,1$. Crucially this
implies
\begin{align}
  \label{eq:sum_properties}
  m_i = \ssum m_{ij},
  \quad
  \sum_{j\in\calI(i)}\cij\;=\;\bzero,
  \quad\text{and}\quad
  \sum_{i\in\calI(j)}\cij\;=\; \bn_j,
\end{align}
which is required for the discrete scheme to be conservative. Here,
\begin{align}
  \label{eq:weighted_normal}
  \bn_j\eqq \int_{\partial\Omega}\varphi_j(\bx)\bn(\bx)\diff o_{\bx},
\end{align}
is a weighted normal associated with the $j$th degree of freedom. The quantity
$\bn(\bx)$ denotes the outward facing unit normal of $\partial\Omega$.

\subsection{Discrete well-balancing and hydrostatic reconstruction}
\label{sec:star_states}
We now introduce an algebraic approach for maintaining the problem-at-rest
property (Definition~\ref{def:problem_at_rest}) at the discrete level. To
this end, we adopt an idea from the shallow water literature for
well-balancing that uses a so-called \emph{hydrostatic reconstruction} of
the density (which is analogous to the water depth) to enforce well-balancing on the discrete
level~\cite{azerad2017shallow, guermond2025erk, audusse_etal_2004}.
Following the general philosophy outlined in~\cite{audusse_etal_2004}, we
introduce a discrete hydrostatic reconstruction of the density and
potential temperature and corresponding hydrostatic star state as follows.
\begin{definition}[Hydrostatic star state]\label{def:star_states}
   Let $i\in\calV$ and let $\bsfU_i= (\rhoi, \bsfM_i, \Theta_i)\tr$ be a
   given state. Then, we define the hydrostatic reconstruction
   $\bsfU_i^{\ast,j}$ of $\bsfU_i$ with respect to $j\in\calI(i)$ by
   setting
  \begin{gather}
    \label{eq:star_states}
    \bsfU_i^{\ast,j}\eqq
    \left({\begin{array}{c}
      \rhostarij
      \\
      \bsfV_i\rhostarij
      \\
      \Theta_i^{*,j}
    \end{array}}\right),
    \qquad\text{with}
    \\
    \notag
    \begin{cases}
      \rhostarij \eqq \dfrac{\Theta_i^{*,j}}{\vartheta_i^{*,j}},
      \quad
      \Theta_i^{*,j} \eqq \Big(\max\big(\epsilon_{ij}, \Theta_i^{\gamma-1} +
      \dfrac{g}{\tilde{c}\thetabar}\min\big(0, z_i -
      z_j\big)\big)\Big)^{1/(\gamma - 1)},
      \\[0.5em]
      \vartheta_i^{*,j} \eqq
      \begin{cases}
        \min\big(\max(\thetai,\thetaj),\,\thetai\Theta_i/\Theta_i^{*,j}\big) & \text{if } z_i<z_j,
        \\
        \thetai & \text{otherwise},
      \end{cases}
      \\[1.2em]
      \thetabar \eqq
      \begin{cases}
        \dfrac{\thetai\thetaj\ln(\thetaj/\thetai)}{\thetaj-\thetai} & \text{if } \thetai\neq\thetaj,
        \\
        \thetai & \text{otherwise},
      \end{cases}
    \end{cases}
  \end{gather}
  where $\epsilon_{ij} \eqq
  \delta\cdot\min(\Theta_i^{\gamma-1}, \Theta_j^{\gamma-1})$, for
  some $\delta\in(0,1)$ such that $\delta\ll 1$. For all numerical tests
  below, we take $\delta=10^{-14}$. Recall that
  $z_i$ denotes the $d$th component of the collocation point $\bx_i$, and
  $\tilde{c}=C_{eos}\gamma/(\gamma-1)$. Note that we suppress the time
  index $n$ when it is clear from context.
\end{definition}
A short remark on the strategy behind the construction of
\eqref{eq:star_states} is in order. The star state~\eqref{eq:star_states}
lifts the lower of the two states $\bsfU_i$, $\bsfU_j$ to the height of the
upper one and leaves the upper one unchanged. The formula for
$\Theta_i^{*,j}$ integrates the equilibrium
condition~\eqref{eq:problem_at_rest} exactly along a profile with
exponential potential temperature, where $\thetabar^{-1}$ is the exact mean
of $\theta^{-1}$ along the edge, and $\vartheta_i^{*,j}$ is set to the
potential temperature of the upper state. At equilibrium both reconstructed
states therefore coincide with the upper state, which gives
$\rhostarij=\rhostarji$; see Definition~\ref{def:discrete_equilibrium} and
Lemma~\ref{lem:hydrostatic_reconstruction} below. The maximum keeps
$\vartheta_i^{*,j}\ge\thetai$, so that $\rhostarij\le\rhoi$. The minimum
limits the change to $\vartheta_i^{*,j}-\thetai=\mathcal{O}(|z_j-z_i|)$,
which we require to establish consistency in
Proposition~\ref{prop:consistency}.
\begin{definition}[Discrete problem at equilibrium and well-balancing]
  \label{def:discrete_equilibrium}
  \begin{itemize}
    \item[(i)]
      We say that a discrete state $\bu_h$ is \emph{at equilibrium} if
      there is a constant $\bV^{\mathrm{eq}}$ with
      $\bV^{\mathrm{eq}}\SCAL\hat{\mathbf{e}}_{d} = 0$ such that for all $i\in\calV$
      \begin{align}
        \label{eq:discrete_problem_at_rest}
        \begin{cases}
          \bV_i=\bV^{\mathrm{eq}},
          \\
          \Theta_j^{\gamma-1}-\Theta_i^{\gamma-1} = -\frac{g}{\tilde{c}\thetabar}(z_j-z_i)\quad\text{for all }j\in\calI(i),
          \\
          \rhostarij = \rhostarji\quad\text{for all }j\in\calI(i).
        \end{cases}
      \end{align}
    \item[(ii)]
      A mapping $S: [\mathcal{P}(\mathcal{T}_h)]^{d+2}\to[\mathcal{P}(\mathcal{T}_h)]^{d+2}$ is
      said to be a \emph{well-balanced scheme} if  $S(\bu_h) =
      \bu_h$ holds for all discrete states $\bu_h$ at equilibrium.
  \end{itemize}
\end{definition}

\begin{lemma}[Hydrostatic reconstruction]
  \label{lem:hydrostatic_reconstruction}
  Let $\bu(\bx,t)$ be a solution at stratified equilibrium
  (Def.~\ref{def:problem_at_rest}) with constant
  horizontal velocity, and assume that $\rho\theta^2$ is nonincreasing
  in $x_d$, \ie $N^2\le g^2\rho/(\gamma p)$. Let
  $\bu_h(\bx)=\mathcal{I}_h\bu(\bx,t_n)$ be the Lagrange interpolant of
  $\bu(\bx,t)$ at time $t_n$. Then, $\bu_h$ is already a discrete state at
  equilibrium satisfying Def.~\ref{def:discrete_equilibrium}.
\end{lemma}
\begin{proof}
  Since $\bu$ is at stratified equilibrium with constant
  horizontal velocity, its interpolant satisfies
  $\bV_i\SCAL\hat{\mathbf{e}}_d = 0$ and $\bV_i = \bV^{\mathrm{eq}}$
  for all $i\in\calV$. It remains to show the second and third
  condition in~\eqref{eq:discrete_problem_at_rest}.
  Owing to Lemma~\ref{algebraic_balance} the continuous solution
  $\bu(\bx,t)$ satisfies \eqref{eq:problem_at_rest} and the identity holds
  true in particular for all Lagrange points. Since
  $\thetaj/\thetai=\exp(N^2(z_j-z_i)/g)$ we have
  $\int_{z_i}^{z_j}\theta^{-1}\diff s=(z_j-z_i)/\thetabar$, implying:
  \begin{equation}
    \Theta_j^{\gamma-1} - \Theta_i^{\gamma-1} =
    -\frac{g}{\tilde{c}\thetabar}\big(z_j - z_i\big) \label{dwb},
    \quad\text{for all }i,j\in\calV.
  \end{equation}
  This is the second condition. Assume first that $z_j< z_i$.
  Then, $\rhostarij = \rhoi$ and
  $(\Theta_j^{*,i})^{\gamma-1} = \max\big(\epsilon_{ij},
  \Theta_j^{\gamma-1} +
    \frac{g}{\tilde{c}\thetabar}\big(z_j -
    z_i\big)\big)$.

    By \eqref{dwb}, the second argument of the
    maximum equals $\Theta_i^{\gamma-1} > \epsilon_{ij}$. Thus,
  $\Theta_j^{*,i} = \Theta_i$, $\vartheta_j^{*,i} =
  \min(\thetai,\thetaj\Theta_j/\Theta_i) = \thetai$ since
  $\thetaj\le\thetai$ for $N\ge0$ and $\thetai\Theta_i\le\thetaj\Theta_j$,
  and $\rhostarji = \Theta_i/\thetai = \rhoi = \rhostarij$. If $z_j \geq z_i$, the same argument with $i$
  and $j$ interchanged yields $\rhostarij = \rhoj = \rhostarji$. Hence,
  $\rhostarji = \rhostarij$ in either case.
\end{proof}

As in~\cite{azerad2017shallow,guermond2025erk}, we would like to utilize
these hydrostatic star states in our numerical flux for ensuring
well-balancing. This leads us to the following lemma which shows that this
is indeed feasible.
\begin{proposition}[Consistency]
  \label{prop:consistency}
  Let $\bu(\bx,t)$ be a smooth solution to the atmospheric Euler equations
  with density and potential temperature bounded
  away from zero, \ie $0<\rho_{\min}\leq\rho\leq\rho_{\max}$ and
  $\theta\geq\theta_{\min}>0$ in $D$, let $\gamma>1$, and let $\bu_h=\mathcal{I}_h\bu(\bx,t_n)$ be the Lagrange interpolant of
  $\bu(\bx,t)$. Then:
  \begin{enumerate}
    \item[(i)]
      The pressure and gravity term admit a first-order approximation as
      follows:
  \end{enumerate}
  \begin{align}
    \label{eq:well_balanced_source_term}
    \int_{D}(\nabla p + \rho g \hat{\mathbf{e}}_{d})\varphi_i \diff x
    \,=\,
    \sum\limits_{j\in\I(i)}
    (p_j^{*,i} - p_i^{*,j}) \cij \,+\, m_i\mathcal{O}(h).
  \end{align}
  \begin{enumerate}
    \item[(ii)]
      Under Assumption~\ref{ass:balanced_stencil}, the
      pressureless flux $\bm{g}(\bm{u})$ can be approximated as follows:
  \end{enumerate}
  \begin{align}
    \label{eq:well_balanced_flux}
    \int_{D}\DIV\bm{g}(\bm{u})\varphi_i \diff x
    \,=\,
    \sum\limits_{j\in\I(i)} \big(\bm{g}(\Ustarij) + \bm{g}(\Ustarji)\big)
    \cij \,+\, m_i\mathcal{O}(h),
  \end{align}

  where $i\in\calV$ is an arbitrary interior degree of freedom.
\end{proposition}

\begin{proof}
    (i) Fix $i\in\calV$. Assume the solution is smooth and assume that
    $\bsfU_i\in\calA$. Suppose first $z_j \leq z_i$. By definition
$|z_j - z_i| \leq h$, and so we have
\begin{equation*}
  \Theta_j^{\gamma-1} + \frac{g}{\tilde{c}\thetabar}
  \min\big(0, z_j - z_i\big)
  \geq \Theta_j^{\gamma-1} -
  \frac{g}{\tilde{c}\thetabar} h
  > \delta\cdot\Theta_j^{\gamma-1}
  \geq \epsilon_{ij},
\end{equation*}
where the second inequality holds provided
$\Theta_j^{\gamma-1} > \frac{g}{(1-\delta)\tilde{c}\thetabar}h$. Using $\thetabar\geq
\theta_{\min}$, $\Theta_j\geq\rho_{\min}\theta_{\min}$, and $\delta<\frac12$, the above condition is
satisfied for all
$h\leq h_0\eqq\frac{\tilde{c}\,\theta_{\min}^\gamma\rho_{\min}^{\gamma-1}}{2g}$,
in which case the left-hand side is in fact bounded from below by
$\frac12(\rho_{\min}\theta_{\min})^{\gamma-1}$. Hence the star-state
$\Theta_j^{*,i}$ reduces to:
\begin{equation*}
  (\Theta_j^{*,i})^{\gamma-1} = \Theta_j^{\gamma-1} +
  \frac{g}{\tilde{c}\thetabar}
  \min\big(0, z_j - z_i\big).
\end{equation*}
Now, using that $\thetaj = \thetai +
    \mathcal{O}(h)$ and $|z_i - z_j| \leq h$, expanding the binomial for $\big(\Theta_j^{\gamma-1} +
  \frac{g}{\tilde{c}\thetabar}\big(z_j - z_i\big)\big)^\frac{\gamma}{\gamma-1}$ we have:
  \begin{equation*}
      \big(\Theta_j^{\gamma-1} +
  \frac{g}{\tilde{c}\thetabar}\big(z_j - z_i\big)\big)^\frac{\gamma}{\gamma-1} = \Theta_j^\gamma + \frac{\gamma}{\gamma-1}\Theta_j\frac{g}{\tilde{c}\thetabar}(z_j-z_i) + O(h^2).
    \end{equation*}
    Since $\Theta_i^{*,j}=\Theta_i$ for $z_j\leq z_i$
    and $p_j^{*,i} = C_{eos}(\Theta_j^{*,i})^\gamma$, multiplying
    by $C_{eos}$ and using
    $C_{eos}\frac{\gamma}{\gamma-1}=\tilde{c}$ gives
\begin{equation*}
    p_j^{*,i} - p_i^{*,j} = p_j - p_i + g\varrho_j\frac{\thetaj}{\thetabar}(z_j-z_i) + O(h^2).
\end{equation*}
Thus, we have:
    \begin{multline*}
      \sum\limits_{j\in\I(i)}\big(p_j^{*,i} - p_i^{*,j}\big)\cij
       \,=\,
      \underbrace{\sum\limits_{j\in\I(i)}
      C_{eos}((\varrho_j\vartheta_j)^\gamma - (\varrho_i\vartheta_i)^\gamma )\cij}_{\text{(i)}} \\ \qquad + \underbrace{\sum\limits_{j\in\I(i)}g\varrho_j\frac{\thetaj}{\thetabar}(z_j-z_i)\cij}_{\text{(ii)}} + \underbrace{\sum\limits_{j\in\I(i)}O(h^2)\cij}_{\text{(iii)}}
      \\ \,=\,
      \underbrace{\int_{D}(\nabla p)\varphi_i \diff
      x\,+\,m_i\mathcal{O}(h)}_{=\,\text{(i)}} +
      \underbrace{\int_{D}(\rho g\,\hat{\mathbf e}_d)\varphi_i \diff
      x\,+\,m_i\mathcal{O}(h)}_{=\,\text{(ii)}}
      \,+\,\underbrace{m_i\mathcal{O}(h)}_{=\,\text{(iii)}}
    \end{multline*}
    For the second term we used the fact that $\rhoj = \rhoi +
    \mathcal{O}(h)$, $\frac{\thetaj}{\thetabar}= 1 +
    \mathcal{O}(h)$, and $|z_j - z_i| \leq h$, as
    well as $\sum_{j\in\I(i)}(z_j-z_i)\cij=m_i\hat{\mathbf e}_d$, which
    holds since the coordinate $z$ is reproduced exactly by the finite
    element space.
    The estimate for $\text{(iii)}$ stems from the fact that
    $\mathcal{O}(h^2)\norm{\mathbf{c}_{ij}}_{\ell^2} = m_i\mathcal{O}(h)$. The case $z_j>z_i$ follows similarly.

    (ii) With the shorthand notation $h_{ij}:=\max(0,z_j-z_i)$, again,
  \begin{equation*}
  \varrho_i^{\gamma-1} + \frac{g}{\tilde{c}\thetabar\thetai^{\gamma-1}}
  \min\big(0, z_i - z_j\big)
  \geq\frac12\rho_{\min}^{\gamma-1} > \epsilon_{ij}/\thetai^{\gamma-1}
  \end{equation*}
for all $h\leq h_0$, by the same argument as in (i). We have
    \begin{align*}
      \bm{g}(\Ustarij)
      \,&=\,
      \frac{\Theta_i^{*,j}}{\Theta_i}\big(\bm{g}(\bsfU_i)-\zeta_{ij}\,\mathbb{E}\,\bm{g}(\bsfU_i)\big),
      \\
      \frac{\Theta_i^{*,j}}{\Theta_i}
      \,&=\,
      \Big(1+g\frac{\gamma-1}{\gamma}\,\frac{\rhoi\thetai}{\thetabar\,p(\bsfU_i)}
      \,\min(0,z_i-z_j)\Big)^{1/(\gamma-1)},
    \end{align*}
    where $\zeta_{ij}\eqq1-\thetai/\vartheta_i^{*,j}$ and
    $\mathbb{E}\eqq\mathrm{diag}(1,\dots,1,0)\in\mathbb{R}^{(d+2)\times(d+2)}$
    removes the last row of $\bm{g}$.
    Using again the lower bound above together with
    $\varrho_i\leq\rho_{\max}$, we have
    \begin{equation*}
      1+g\frac{\gamma-1}{\gamma}\,\frac{\rhoi\thetai}{\thetabar\,p(\bsfU_i)}
      \,\min(0,z_i-z_j)
      \;=\;
      \Big(\frac{\Theta_i^{*,j}}{\Theta_i}\Big)^{\gamma-1}
      \;\geq\;
      \frac12\Big(\frac{\rho_{\min}}{\rho_{\max}}\Big)^{\gamma-1}
      \;>\;0.
    \end{equation*}
    Now, using again $\thetai/\thetabar= 1 + \mathcal{O}(h)$
    and the fact that $(1+h)^\alpha=1 \,+\,
    \alpha\,h+\mathcal{O}(h^2)$ for any $\alpha > 0$,
    this implies
    \begin{align*}
      \bm{g}(\Ustarij) &- \bm{g}(\bsfU_i)
      \\
      &=
      \bm{g}(\bsfU_i)
      \Big(\Big(1-g\frac{\gamma-1}{\gamma}\frac{\rhoi}{p(\bsfU_i)}h_{ij}
      +\mathcal{O}(h_{ij}h)\Big)^{1/(\gamma-1)}-1\Big)
      -\frac{\Theta_i^{*,j}}{\Theta_i}\zeta_{ij}\mathbb{E}\bm{g}(\bsfU_i)
      \\
      &=
      -\bm{g}(\bsfU_i)\big(
      \frac{g}{\gamma}\frac{\rhoi}{p(\bsfU_i)}h_{ij}+
      \mathcal{O}(h_{ij}h)\big)
      -\zeta_{ij}\mathbb{E}\bm{g}(\bsfU_i).
    \end{align*}
    This implies for any
    interior degree of freedom $i$:
    \begin{multline*}
      \Big| \sum\limits_{j\in\I(i)} \big(\bm{g}(\Ustarij) +
      \bm{g}(\Ustarji)\big)\cij
      \,-\, \sum\limits_{j\in\I(i)} \big(\bm{g}(\bsfU_i) +
      \bm{g}(\bsfU_j)\big)\cij \Big|
      \\
      \;=\;
      \Big|\sum\limits_{j\in\I(i)} \big(
      \bm{g}(\bsfU_j)\,\frac{g}{\gamma}\,\frac{\rhoj}{p(\bsfU_j)}h_{ji}
      \,+\,
      \bm{g}(\bsfU_i)\,\frac{g}{\gamma}\,\frac{\rhoi}{p(\bsfU_i)}h_{ij}
      \\
      +\zeta_{ji}\,\mathbb{E}\,\bm{g}(\bsfU_j)+\zeta_{ij}\,\mathbb{E}\,\bm{g}(\bsfU_i)
      \big)\,\cij +\mathcal{O}\big((h_{ij}+h_{ji})h\big)\,\cij \Big|
      =(\mathrm{iii})
    \end{multline*}
    Using $h_{ij}+h_{ji}=|z_j-z_i|$ and Assumption~\ref{ass:balanced_stencil}, we then conclude:
    \begin{multline*}
      (\mathrm{iii})
      \;=\;
      \Big|\sum\limits_{j\in\I(i)} \big(
      \underbrace{\bm{g}(\bsfU_j)\,\frac{g}{\gamma}\,\frac{\rhoj}{p(\bsfU_j)}
      \,-\,
      \bm{g}(\bsfU_i)\,\frac{g}{\gamma}\,\frac{\rhoi}{p(\bsfU_i)}}
      _{\,=\,\mathcal{O}(h)}\big)\,h_{ji}\,\cij
      \\[-0.25em]
      +\;\bm{g}(\bsfU_i)\,\frac{g}{\gamma}\,\frac{\rhoi}{p(\bsfU_i)}
      \underbrace{\sum\limits_{j\in\I(i)}|z_j-z_i|\,\cij}_{\,=\,m_i\mathcal{O}(h)}
      +\;\mathbb{E}\,\bm{g}(\bsfU_i)
      \underbrace{\sum\limits_{j\in\I(i)}(\zeta_{ij}+\zeta_{ji})\,\cij}_{\,=\,m_i\mathcal{O}(h)}
      \\
      +\;\sum\limits_{j\in\I(i)} \mathcal{O}(h^2)\,\cij \Big|
      \;=\;m_i\,\mathcal{O}(h).
    \end{multline*}
    Both underbraced sums follow by pairing $j$ with $j'$ from
    Assumption~\ref{ass:balanced_stencil}, using
    $\zeta_{ij}=\min\big(\max(0,1-\thetai/\thetaj),1-\Theta_i^{*,j}/\Theta_i\big)$
    for $z_i<z_j$ and $\zeta_{ij}=0$ otherwise. The weights,
    $|z_j-z_i|$ and $\zeta_{ij}+\zeta_{ji}$, are of order $h$ and change by
    $\mathcal{O}(h^2)$ under $j\mapsto j'$; since
    $\mathbf{c}_{ij'}=-\cij+\mathcal{O}(h^d)$, each pair contributes
    $\mathcal{O}(h^{d+1})=m_i\mathcal{O}(h)$.

    The statement now follows from the fact that $\sum_{j\in\I(i)}
    \big(\bm{g}(\bsfU_i) + \bm{g}(\bsfU_j)\big)\cij$ is a second-order
    approximation of the pressureless flux: Since
    $\sum_{j\in\I(i)}\cij = \bm{0}$, this sum equals
    $\int_D\DIV(\mathcal{I}_h\bm{g}(\bu))\varphi_i\diff x$. Integrating
    by parts, with $\varphi_i$ vanishing on $\partial D$, yields
    \begin{align*}
      \Big|\int_D\DIV\big(\bm{g}(\bu)-\mathcal{I}_h\bm{g}(\bu)\big)
      \varphi_i\diff x\Big|
      &\leq
      \|\bm{g}(\bu)-\mathcal{I}_h\bm{g}(\bu)\|_{L^\infty(\mathrm{supp}\,\varphi_i)}
      \,\|\GRAD\varphi_i\|_{L^1(D)}
      \\
      &= \mathcal{O}(h^2)\,\mathcal{O}(h^{d-1}) = m_i\mathcal{O}(h).
    \end{align*}
\end{proof}

\subsection{Low order method}
\label{sec:low_order}

We now introduce the low-order method. The goal is to utilize the
hydrostatic reconstruction states above in combination with local discrete
auxiliary states to ensure both well-balancing and invariant-domain
preserving properties.

We recall that an essential quantity for constructing a  robust numerical
method is the so called \emph{bar state}, or \emph{auxiliary} state;
see~\cite{hoff1979finite, Harten_Lax_VanLeer_1983, Nessyahu_Tadmor_1990,
guermond2016invariant}:
\begin{equation}\label{eq:bar_states}
  \overline{\bsfU}_{ij}^n := \frac{1}{2}
  \left(\bsfU_i^{n,*,j} + \bsfU_j^{n,*,i}\right) -
  \frac{1}{2\dij\upLn}\left(\bm{f}(\bsfU_j^{n,*,i})-
  \bm{f}(\bsfU_i^{n,*,j})\right)\cij,
\end{equation}
where, following~\cite{guermond2016invariant}, we introduce a
graph-viscosity as follows:
\begin{align}
  \label{eq:dij}
  \dij\upLn := \max\{\hat{\lambda}_\text{max}(\bsfU_i^{n,\ast,j},
  \bsfU_j^{n,\ast,i})\norm{\cij},
  \hat{\lambda}_\text{max}(\bsfU_j^{n,\ast,i},
  \bsfU_i^{n,\ast,j})\norm{\mathbf{c}_{ji}}\},
\end{align}
with the convention that $d_{ii}\upLn\eqq-\sum\limits_{j\in\calI(i)\setminus\{i\}}\dij\upLn$. Here, $\hat{\lambda}_\text{max}$ is an upper bound on the maximum wave
speed of a local Riemann problem defined by the left state
$\bsfU_i^{n,\ast,j}$ and right state $\bsfU_j^{n,\ast,i}$. We provide
details for computing the maximum wave speed in
Appendix~\ref{sec:riemann_problem}.
The key property that ensures the invariant domain preservation is the fact that the bar
state $\overline{\bsfU}_{ij}^n$ is the \emph{Riemann average} over the
space $[-\frac{1}{2}, \frac{1}{2}]$ at an appropriate time $t^\ast$
depending on $d_{ij}\upLn$; see~\cite[Thm.~4.1]{guermond2016invariant}.

We are now in a position to construct a low-order scheme which is a convex
combination of the state $\bsfU_i\upn$ and discrete auxiliary states both
with an affine shift. We define $p_i^{n,\ast,j} := p(\bsfU_i^{n,\ast,j})$. For
all $i\in\calV$ and $j\in\calI(i)$, we set
\begin{equation}
  \label{eq:convex}
  \bsfU_i\upLnp := (1 + \frac{2\tau d_{ii}\upLn}{m_i})(\bsfU_i^n +
  \frac{\tau}{m_i}\bsfS_i\upLn) + \sum\limits_{j\in\I(i)\setminus \{i\}}
  \frac{2\tau \dij\upLn}{m_i}(\overline{\bsfU}_{ij}^n +
  \frac{\tau}{m_i}\bsfS_i\upLn),
\end{equation}
where the affine shift is defined by:
\begin{align}
  \label{eq:shifted_bar_state}
  &\bsfS_i\upLn :=\sum\limits_{j\in\I(i)}\bsfS_{ij}\upLn,
  \\\notag
  &\bsfS_{ij}\upLn :=-2(\dij\upLn + \bsfV_i^n\SCAL\cij)(\bsfU_i^{\ast,j} - \bsfU_i\upn)\notag.
\end{align}
A straightforward calculation establishes the following lemma.
\begin{lemma}
  \label{lemma:convex_combination}
  The low-order update~\eqref{eq:convex} is algebraically equivalent to:
   \begin{align}
    \label{eq:low_order_scheme}
    m_i\frac{\bsfU_i\upLnp - \bsfU_i^n}{\tau}
    =& \sum\limits_{j\in\mathcal{I}(i)}\mathbf{F}_{ij}\upLn,
    \\\notag
    \mathbf{F}_{ij}\upLn := &
    -(\mathbf{g}(\bsfU_j^{n,\ast,i}) + \mathbf{g}(\bsfU_i^{n,\ast,j}))\cij +
    \dij\upLn(\bsfU_j^{n,\ast,i} - \bsfU_i^{n,\ast,j})
    \\\notag
    \qquad -&
    \left[{\begin{array}{c}
    0
    \\
    {(p_j^{n,*,i} - p_i^{n,*,j} + 2p_i)\cij}
    \\
    0
    \end{array}}\right].
  \end{align}
\end{lemma}
The main result of the section is the following proposition summarizing the
structure preserving properties of the low-order update~\eqref{eq:convex}.
\begin{proposition}
  \label{prop:low_order_method}
  For the low-order update $\bsfU\upn \mapsto \bsfU\upLnp$ given by
  \eqref{eq:convex} the following properties hold true provided that a CFL
  condition holds true,
  \begin{align}
    \label{eq:cfl_condition}
    \tau_n \leq \min\limits_{i\in\calV}\frac{m_i}{2 |d_{ii}^{L,n}|}:
  \end{align}
  \begin{enumerate}[label= (\roman*)]
    \item
      It is first-order consistent for smooth flows.
    \item
      It is well-balanced for $\bV^{\mathrm{eq}}=\bm{0}$
      on general meshes, and for $\bV^{\mathrm{eq}}\neq\bm{0}$ under
      Assumption~\ref{ass:row_cancellation}.
    \item
      Up to boundary fluxes, it conserves density $\rho$ and potential
      temperature $\Theta = \varrho\vartheta$. The momentum $\bbm$, on the
      other hand, is only conserved if $g=0$.
    \item It maintains admissibility, $\bsfU_i\upLnp\in\calA$.
    \item It is invariant-domain preserving in the sense that
      \begin{align*}
        \bsfU_i\upLnp \,\in\,
        \conv\limits_{j\in\I(i)}\big(\overline{\bsfU}_{ij}^n\big) \,+\,
        \frac{\tau}{m_i}\bsfS_i\upLn.
      \end{align*}
  \end{enumerate}
\end{proposition}

\begin{proof}
  Since~\eqref{eq:convex} and~\eqref{eq:low_order_scheme} are equivalent by
  virtue of Lemma~\ref{lemma:convex_combination}, we prove all the results
  for the latter formulation.

  (i) The consistency is a direct consequence of
  \eqref{eq:low_order_scheme} and Proposition~\ref{prop:consistency}.

  (ii)
  Suppose $\bu_h\upn$ is at equilibrium. Then $\Theta_i^{n,*,j}=\Theta_j^{n,*,i}$,
  $\rhostarijn = \rhostarjin$ and $\bV_i\upn=\bV_j\upn=\bV^{\mathrm{eq}}$ for all
  $i\in\mathcal{V}$ and $j\in\I(i)$, so the viscosity terms vanish, the pressure terms reduce to $-2p_i\cij$ and sum to zero, and
  $\sum_{j\in\calI(i)} \bsfF_{ij}\upLn = -\sum_{j\in\calI(i)}(\mathbf{g}(\bsfU_j^{n,\ast,i}) + \mathbf{g}(\bsfU_i^{n,\ast,j}))\cij$.
  If $\bV^{\mathrm{eq}}=\bm{0}$ this sum vanishes trivially; otherwise, since $\bsfU_i^{n,*,j}$ depends only on $\max(z_i,z_j)$, Assumption~\ref{ass:row_cancellation} gives
  $\sum_{j\in\calI(i)} \bsfF_{ij}\upLn = \mathbf{0}$. Hence $\bsfU_i\upLnp = \bsfU_i\upn$ for
  all $i\in\mathcal{V}$, and so the scheme is well-balanced.

  (iii)
In order to show conservation, we need to prove that the right hand side
  of \eqref{eq:low_order_scheme} vanishes when summed up over index
  $i\in\calV$: We have that $\cij = -\mathbf{c}_{ji}$ for all
  $i\in\calV$, $j\in\I(i)$, provided that either $i$, or $j$ is an
  \emph{interior} degree of freedom, \ie $\bx_i$ or $\bx_j$ are not
  collocated on the boundary $\partial\Omega$. Furthermore,
  $\mathbf{g}(\bsfU_j^{n,\ast,i}) + \mathbf{g}(\bsfU_i^{n,\ast,j})$ and
  $\dij\upLn$ are symmetric, and $\bsfU_j^{n,\ast,i} - \bsfU_i^{n,\ast,j}$ is
  skew-symmetric. The remaining contribution acts on the momentum and is
  given by $(p_j^{n,*,i} - p_i^{n,*,j} + 2p_i)\cij$. In case of vanishing gravity, $g=0$,
  $(p_j^{n,*,i} - p_i^{n,*,j} + 2p_i)\cij = (p_j^{n} + p_i^{n})\cij$ and thus the total momentum is conserved as well.

  (iv)
  Admissibility is a direct consequence of \eqref{eq:convex}, the fact that
  $\overline{\bsfU}_{ij}^n\in\calA$ are admissible by construction
  \eqref{eq:dij} of $\dij\upLn$ under a CFL condition, and the observation
  that, since $\rhostarij\leq\varrho_i$, $\Theta_i^{*,j}\leq\Theta_i$, and
  $\dij\upLn\geq|\bsfV_i^n\SCAL\cij|$ by~\eqref{eq:dij},
    \begin{align*}
    \varrho(\bsfS_i\upLn) &= \sum\limits_{j\in\I(i)} -2(\dij\upLn +
    \bsfV_i^n\cdot\cij)(\rhostarij - \rhoi)
    \,\geq\,0\quad\text{and} \quad
    \\
    \varrho\vartheta(\bsfS_i\upLn)  &= \sum\limits_{j\in\I(i)} -2(\dij\upLn
    + \bsfV_i^n\cdot\cij)(\Theta_i^{*,j} - \Theta_i)
    \,\geq\,0.
  \end{align*}
  Thus, $\bsfU_i\upLnp\in\mathcal{A}$ for all $i\in\calV$ by \eqref{eq:convex}.

  (v)
  The convex combination follows directly from \eqref{eq:convex}.
\end{proof}

\begin{remark}[Spatial accuracy]
    When using linear elements, the scheme~\eqref{eq:low_order_scheme} is
    observed to be first-order accurate in space with respect to the mesh
    size. We remind the reader that the importance of the low-order scheme
    is to use it as a safeguard for the high-order method.
\end{remark}

\subsection{Local bounds}
\label{sec:local_bounds}
A crucial ingredient for ensuring fidelity and robustness of the high-order
method is the construction of suitable local bounds. Motivated by the
observation (see: Prop.~\ref{prop:low_order_method}) that the low order
update satisfies
\begin{align*}
  \bsfU_i\upLnp \,\in\,
  \conv\limits_{j\in\I(i)}\big(\overline{\bsfW}_{ij}^n\big),
  \qquad
  \overline{\bsfW}_{ij}^n\,:=\, \overline{\bsfU}_{ij}^n \,+\,
  \frac{\tau}{m_i}\bsfS_i\upLn,
\end{align*}
we construct local bounds $\calB_i^n$ as follows:
\begin{equation}\label{invariant_domain}
  \calB_i^n\eqq\Big\{\bu\in\polR^m \;:\;
    \rhoi^{n,\min}\le\rho\le \rhoi^{n,\max}
    \text{ and }
     \vartheta_i^{n,\min} \le\theta\le \vartheta_i^{n,\max}
  \Big\},
\end{equation}
where
\begin{gather*}
    \rhoi^{n,\min}\eqq
    \min_{j\in\calI(i)}\rho(\overline{\bsfW}_{ij}\upn),\qquad
    \rhoi^{n,\max}\eqq \max_{j\in\calI(i)}\rho(\overline{\bsfW}_{ij}\upn),
    \\
    \vartheta_i^{n,\min}\eqq
    \min_{j\in\calI(i)}\theta(\overline{\bsfW}_{ij}\upn),\qquad
    \vartheta_i^{n,\max}\eqq
    \max_{j\in\calI(i)}\theta(\overline{\bsfW}_{ij}\upn).
\end{gather*}
\begin{lemma}
  It holds true that
  \begin{align*}
    \calB_i^n\,\supset\,
    \conv\limits_{j\in\I(i)}\big(\overline{\bsfW}_{ij}^n\big)
    \,\ni\, \bsfU_i\upLnp.
  \end{align*}
\end{lemma}

\begin{proof}
  We can describe $\calB_i^n$ equivalently with \eqref{eq:invariant_domain}
  and the following functions:
  \begin{align*}
    \begin{aligned}
      \Psi_1(\bu) &:=  \rho(\bu) - \rho_i^{n, \min},
      &\qquad
      \Psi_2(\bu) &:=  \rho_i^{n, \max}-\rho(\bu),
      \\
      \Psi_3(\bu) &:=  (\rho\theta)(\bu) - \rho(\bu)\vartheta_i^{n, \min},
      &\qquad
      \Psi_4(\bu) &:=  \rho(\bu)\vartheta_i^{n, \max}-(\rho\theta)(\bu).
    \end{aligned}
  \end{align*}
  The property now follows from the observation that all $\Psi_l(\bu)$ are
  \emph{quasi-concave} when viewed as a function of the conserved
  quantities; see, e.\,g., \cite[Lem~7.4]{guermond2019idp}.
\end{proof}
We highlight that the property above regarding the local bounds is stronger than the notion of \textit{positivity} preservation (\ie local vs global bounds). In the next section we construct a high-order update $\bsfU_i^{n+1}$ that
ensures $\bsfU_i^{n+1}\in\calB_i^n$.

\subsection{High-order method}
\label{sec:high_order}
Robust high-order accuracy in space can be achieved by combining a
high-order spatial discretization based on entropy
viscosity~\cite{guermond2011entropy} with a suitable limiter
technique~\cite{guermond2018second} that ensures robustness. This is then
combined with a robust high-order time-stepping
strategy~\cite{guermond2025erk}. For the sake of completeness, we outline
briefly our concrete choice of constructing a high-order update with convex
limiting. We first introduce a high-order flux

\begin{multline}
  \mathbf{F}_{ij}\upHn \eqq -
  \big(\bm{f}(\bsfU_j^{n}) + \bm{f}(\bsfU_i^{n})\big)\cij +
  \dij\upHn(\bsfU_j^{n,\ast,i} - \bsfU_i^{n,\ast,j})
  \\
  -
  \left[\begin{array}{c}
  0
  \\
  \tilde{\rho}_{ij}^{n}\,g\,\big(z_j - z_i\big)\cij
  \\
  0
  \end{array}\right],
\end{multline}
where $\tilde{\rho}_{ij}^{n}\eqq \frac{p(\bsfU_j^n)-p(\bsfU_i^n)}{\tilde{c}\overline{\vartheta}_{ij}^n\big((\rhoj^{n}\thetaj^n)^{\gamma -1} -
  (\rhoi^{n}\thetai^n)^{\gamma -1}\big)}$ if $\rhoi\upn\thetai\upn\neq\rhoj\upn\thetaj\upn$, and $\tilde{\rho}_{ij}^n\eqq\rhoi\upn\thetai\upn/\overline{\vartheta}_{ij}^n$ otherwise.
Here, we set $\mathbf{F}_i\upHn \eqq \sum\limits_{j\in\I(i)} \mathbf{F}_{ij}\upHn$ and $\dij\upHn$ is a high-order graph viscosity coefficient that is
constructed with an entropy-viscosity commutator~\cite{guermond2011entropy}:
\begin{align*}
  \dij\upHn = \dij\upLn\frac{\alpha_i^n + \alpha_j^n}{2}
\end{align*}
where $\alpha_i^n\in[0,1]$ is the entropy production
indicator~\cite{guermond2019idp, guermond2018second} defined by
\begin{gather*}
    \alpha_i^n \eqq \min\Big(1,\frac{|N_i|}{D_i^n+\frac{m_i}{|D|}|\eta(\bsfU_i^n)|}\Big),
    \qquad
    N_i^n \eqq \sum\limits_{j\in\I(i)}\big(\mathbf{q}(\bsfU_j^{n}) -
    (\nabla_\bu\eta(\U_i^n))^T\bm{f}(\bsfU_j^{n})\big)
    \mathbf{c}_{ij},
    \\
    D_i^n \eqq
    |\sum\limits_{j\in\I(i)}\mathbf{q}(\bsfU_j^{n})\mathbf{c}_{ij}| +
    |\sum\limits_{j\in\I(i)}(\nabla_\bsfU\eta(\U_i^n))^T\bm{f}(\bsfU_j^{n})\mathbf{c}_{ij}|,
\end{gather*}
where the term $\frac{m_i}{|D|}|\eta(\bsfU_i^n)|$ regularizes the quotient. Here, we used the entropy pair:
\begin{align*}
  \eta(\bu) \eqq \frac{1}{2}\rho\norm{\mathbf{v}}^2 + \frac{1}{\gamma-1}p +
  \rho g x_d,
  \quad\text{and}\quad
  \bq(\bu)\eqq \mathbf{v}(\eta(\bu) + p).
\end{align*}

Since the high-order scheme involves estimating the high-order mass matrix $M^H$,
to mitigate dispersion effects, we follow the approach of~\cite[§3.4]{guermond2014maximum} to
construct an approximation of $(\mathbb{M}^H)^{-1}$.
We introduce the matrix $\mathbb{B} \eqq (\mathbb{M}^L - \mathbb{M}^H)(\mathbb{M}^L)^{-1}$ and approximate
$(\mathbb{M}^H)^{-1}$ by $(\mathbb{M}^L)^{-1}(\mathbb{I} + \mathbb{B})$. As established in~\cite[Prop.~3.1]{guermond2013correction},
this first-order expansion is superconvergent and eliminates the dispersion error
in the approximation with piecewise linear
continuous finite elements on uniform meshes. Denoting by $\{b_{ij}\}_{j \in \mathcal{I}(i)}$,
for all $i \in \mathcal{V}$, the entries of $\mathbb{B}$ are given explicitly by
\begin{equation*}
    b_{ij} = \delta_{ij} - \frac{m_{ij}}{m_j}, \qquad
    b_{ji} = \delta_{ij} - \frac{m_{ji}}{m_i}.
\end{equation*}

Recalling that $\sum_{j \in \mathcal{I}(i)} b_{ji} = 0$, we introduce a high-order update as follows \cite{guermond2018second,
guermond2019idp}:
\begin{gather}
  \label{eq:limited_high}
  \bsfU_i\upnp \eqq
  \sum_{j\in\calI(i)\setminus\{i\}}\lambda_i(\bsfU_i\upLnp+\ell_{ij}
  \upn\bsfP_{ij}\upn),
  \qquad\text{where}
  \\\notag
  \bsfP_{ij}\upn\eqq \frac{\tau}{m_i
  \lambda_i}\left\{\bsfF_{ij}\upHn-\bsfF_{ij}\upLn + b_{ij}\bsfF_{j}\upHn -
  b_{ji}\bsfF_{i}\upHn\right\},
  \qquad
  \lambda_i\eqq\frac{1}{\text{card}(\calI(i))-1},
\end{gather}
and where the limiter coefficients $\ell_{ij}^{n}\in[0,1]$ are chosen symmetric, $\ell_{ij}\upn=\ell_{ji}\upn$, and such that
\begin{align*}
  \Psi_k(\bsfU_i\upLnp + \ell_{ij}\upn\bsfP_{ij}\upn)\geq 0
  \quad\text{for each } k=1,\ldots,4.
\end{align*}
Boundary conditions are enforced by post-processing the nodal states at
boundary degrees of freedom before each stage: for slip boundaries the
normal component of the momentum is removed using the nodal normal, and
Dirichlet data are imposed directly; see~\cite{ryujin-2021-1, ryujin-2021-3}.
We have the following result.
\begin{theorem}\label{thm:final_theorem}
  The scheme $\bsfU\upn\mapsto\bsfU\upnp$ given
  by~\eqref{eq:limited_high} supplemented with the limiter coefficients
  $\ell_{ij}\upn$  such that $\Psi_k(\bsfU_i\upLnp +
  \ell_{ij}\upn\bsfP_{ij}\upn)\geq 0$ for all $k=1,\ldots,4$, is
  conservative (up to source contributions), well-balanced with respect to stratified rest states, and
  invariant-domain preserving under the CFL condition, i.\,e.,
  \begin{align*}
    \bsfU_i\upnp\,\in\, \calB_i^n.
  \end{align*}
\end{theorem}

\begin{proof}
  (i) Suppose that the gravitational acceleration is zero. Since
  $\ell_{ij}\upn$ is symmetric and $m_i\lambda_i\bsfP_{ij}\upn$ is skew-symmetric, it
  follows that
  \begin{align*}
    \sum\limits_{i\in\mathcal{V}}m_i\bsfU_i^{n+1} =
    \sum\limits_{i\in\mathcal{V}}m_i\bsfU_i^{n}
  \end{align*}
  and thus the scheme is conservative without source terms.

  (ii) Suppose $\bu_h\upn$ is at equilibrium (Definition~\ref{def:discrete_equilibrium}) with $\bV^{\mathrm{eq}}=\bm{0}$. Then, for each $j\in\calI(i)$, we have that
  $\bsfU_i^{n,\ast,j}=\bsfU_j^{n,\ast,i}$ and, by the definition of
  $\tilde{\rho}_{ij}^n$, we see $\tilde{\rho}_{ij}^n\,g\,(z_j-z_i)=-(p_j\upn-p_i\upn)$.
  Hence $\mathbf{F}_{ij}\upHn=\mathbf{F}_{ij}\upLn$ and
  $\bsfP_{ij}\upn=\bm{0}$ which yields $\bsfU_i\upnp=\bsfU_i\upn$.

  (iii) To show invariant domain preservation suppose the time step
  satisfies the CFL condition \eqref{eq:cfl_condition}. Then,
  $\bsfU_i\upLnp$ is a convex combination and the local bounds in
  \eqref{invariant_domain} hold. Then, using the construction of the
  limiter we have that:
  \begin{align*}
    \Psi_k(\bsfU_i\upnp) \,=\,
    \Psi_k\Big(\sum_{j\in\calI(i)\setminus\{i\}}\lambda_i\big(\bsfU_i\upLnp+\ell_{ij}
    \upn\bsfP_{ij}\upn\big)\Big)\,\ge\,0,
  \end{align*}
  because $\Psi_k\Big(\bsfU_i\upLnp+\ell_{ij}\upn\bsfP_{ij}\upn\Big)\geq
  0$ for all $j\in\I(i)\setminus\{i\}$ and $\Psi_k$ is a quasi-concave function.
\end{proof}

\begin{remark}[Well-balancing with respect to non-rest states]
    For equilibria with $\bV^{\mathrm{eq}}\neq\bm{0}$,
    we observe numerically that the high-order scheme is
    well-balanced under
    Assumption~\ref{ass:row_cancellation} and a smallness condition on the background velocity
    $\bV^{\mathrm{eq}}$. The
    corrections $\bsfP_{ij}\upn$ then no longer vanish individually, but
    they cancel in the sum, $\sum_{j}\bsfP_{ij}\upn=\bm{0}$, so that
    $\bsfU_i\upnp=\bsfU_i\upn$ as long as the limiter is inactive.
    Moreover, all states entering the limiter are of the form
    $(\rho,\rho\bV^{\mathrm{eq}},\rho\vartheta)\tr$, so the local bounds
    concern only the density and the potential temperature. Ignoring boundary degrees of freedom, these
    bounds are met with $\ell_{ij}\upn=1$ provided
    \vspace{-1em}
    \begin{equation*}
      |\bV^{\mathrm{eq}}|\;\lesssim\;\frac{C\,a}{\text{card}(\calI(i))},
    \end{equation*}
    where $a$ is the sound speed and $C$ depends only on the mesh and the
    CFL number. We do not pursue a rigorous analysis of this case
    here.
\end{remark}

\section{Numerical Illustrations}\label{sec:results}

This section illustrates the method in the following ways: (i) verification using convergence tests, (ii) well-balancing results, and (iii) the validity of the method through various benchmark problems found in the literature.
The numerical tests are performed using \verb|ryujin| \cite{ryujin-2021-1, ryujin-2021-3}, which is a high performance code built upon the \verb|deal.II| finite element library \cite{dealII95}. All tests use continuous, linear finite elements for the spatial approximation (Section~\ref{sec:warped_mesh} additionally uses discontinuous ones) and a third order explicit Runge-Kutta method technique presented in time~\citep[Tab.~1(b)]{guermond2025erk}. The time step size at each stage is computed with
\begin{equation*}
    \tau_n :=\text{CFL} \min\limits_{i\in\nu}\frac{m_i}{2 |d_{ii}^{L,n}|}
\end{equation*}
where $\text{CFL}$ is a user-defined parameter. For all verification tests, we set CFL$=0.5$ unless stated otherwise, and for larger benchmarks we set CFL = $0.9$.

\subsection{Verification}\label{sec:verification}
To determine convergence properties of the method, we use the following consolidated error indicator for the tests:
\begin{equation}
    \delta^q(t) = \sum\limits_{k=1}^m\frac{\norm{\bm{u}_{k,h}(t) - \bm{u}_k(t)}_q}{\norm{\bm{u}_k(t)}_q},
\end{equation}
where $\bm{u}_k(t)$ is the exact solution and $\bm{u}_{k,h}(t)$ is the spatial approximation for each component $k$ of the system; components with $\norm{\bm{u}_k(t)}_q=0$ are not normalized.

\subsubsection{Smooth Wave Solution}
We first test the numerical method with a 1D smooth solution for the model~\eqref{model_problem} without sources, \ie $g=0$. We do so by adapting the smooth traveling wave introduced in ~\cite[Sec.~5.2]{guermond2018second}.
Since the application of this benchmark is new to the atmospheric model, we provide the details for the solution in ~\citep{guermond2018second}.

Let $v(x,t) = v_0$, $p(x,t) = p_c$ and
the density profile is given by:
\begin{equation*}
\rho(x,t) =
    \begin{cases}
        \rho_0 +  2^6(x_1 - x_0)^{-6}(x - v_0t - x_0)^3(x_1 - x + v_0t)^3 & \text{if } x_0 \leq x - v_0t \leq x_1\\
        \rho_0 & \text{otherwise},
    \end{cases}
\end{equation*}
where $x_0 = 0.1$, $x_1 = 0.3$, $c_v = 719$, and $\gamma = \frac{7}{5}$. This gives us the following relation for potential temperature:
\begin{equation*}
    \theta(x,t) = (\frac{p_c}{p_0})^\frac{1}{\gamma}\frac{p_0}{c_v(\gamma - 1)\rho(x,t)}.
\end{equation*}
The computational domain is set to be $D = [0, \SI{1}{m}]$ with Dirichlet
boundary conditions, and $v_0= p_c = 1$. The reference pressure $p_0 =1$
Pa. The final time is set to $t_f=0.1$ s. The consolidated errors and
convergence rates for the high order method are listed in Table
\ref{tab:smooth_wave}. We observe optimal convergence with rates consistent with the
literature, cf.~\cite{guermond2013correction}.
\begin{table}
\begin{center}
\begin{tabular}{rrlrlrl}
  \toprule
   $\bm{I}$ &   $\bm{\delta^1(T)}$ &      &   $\bm{\delta^2(T)}$ &      & $\bm{\delta^\infty(T)}$ & \\[0.25em]
  101 & 1.32e-03 &  & 2.85e-03 &  & 8.41e-03 &  \\
  201 & 1.45e-04 & 3.19 & 3.44e-04 & 3.05 & 1.16e-03 & 2.86 \\
  401 & 1.33e-05 & 3.45 & 4.44e-05 & 2.95 & 2.23e-04 & 2.38 \\
  801 & 1.51e-06 & 3.14 & 6.52e-06 & 2.77 & 3.91e-05 & 2.51 \\
  1601 & 1.69e-07 & 3.16 & 9.52e-07 & 2.78 & 7.65e-06 & 2.35 \\
  3201 & 1.86e-08 & 3.18 & 1.38e-07 & 2.78 & 1.37e-06 & 2.48 \\
  \bottomrule
\end{tabular}
\end{center}
\caption{$\delta^1(T)$, $\delta^2(T)$, and $\delta^\infty(T)$ errors and convergence rates for the one-dimensional smooth wave problem at the final time $t_f=0.1$ s with $g=0$ and CFL $0.5$.} \label{tab:smooth_wave}
\end{table}

\subsubsection{Exact Solution with Gravitational Source}\label{gravity_test}
We now verify the convergence of the method when considering the gravitational source.
We derive the exact solution as follows.
Let $\rho(x,t) = \rho_0$,
$v(x,t) = v_0$, and pressure be defined by the equation of state $p(x,t) = C_{eos}(\rho\theta)^\gamma$. Our set of conservation equations reduce to
\begin{align*}
    \partial_xp &= -\rho g \\
    \partial_t\theta + v_0\partial_x\theta &= 0.
\end{align*}
Then, potential temperature equation is found to be
\begin{equation*}
    \theta(x,t) = \left[-\frac{g}{C_{eos}\rho_0^{\gamma-1}}(x-v_0t)\right]^\frac{1}{\gamma}.
\end{equation*}
The computational domain is set to be $D = [-11, \SI{-1}{m}]$ with Dirichlet boundary conditions, and $v_0= \rho_0 = 1$. Here, we have set the integration constant to zero; on $D$ this ensures $\theta^\gamma>0$ for all $t\geq0$. In 2D we set $D=[0,10]\times[-11,-1]$, $\bm{v} = (0, v_0)^T$ and
\begin{equation*}
    \theta(\bm{x},t) = \left[-\frac{g}{C_{eos}\rho_0^{\gamma-1}}(y-v_0t)\right]^\frac{1}{\gamma}.
\end{equation*}
Let $g=1$ and $\gamma=1.4$. We compare the consolidated errors and rates for first-order in two-dimensions in Table \ref{tab:low-order-results} at a final time of $t_f=1$ s, where we sum over $\rho$, $m_2$, and $\rho\theta$. Note that we still observe first-order convergence for a non-uniform mesh distorted by 10\%: a weaker assumption on the mesh than in Proposition \ref{prop:consistency}, which assumes some regularity on the mesh in the direction of gravity. The consolidated errors and convergence rates for the high order method in 1D are listed in Table \ref{tab:exact_source} at a final time of $t_f=5$ s.
We observe optimal convergence. For this test we do not limit the high-order update at Dirichlet boundary nodes: their states are overwritten by the time-dependent boundary data, which the local bounds constructed from interior states cannot accommodate.
\begin{table}[t]
\centering
  \begin{subtable}[t]{0.48\linewidth}
    \centering
    \resizebox{\linewidth}{!}{%
    \begin{tabular}{rrlrl}
    \toprule
    $\bm{I}$ & $\bm{\delta^1(T)}$ & & $\bm{\delta^\infty(T)}$ & \\
    121 & 4.00e-02 &  & 1.21e-01 &  \\
    441 & 2.68e-02 & 0.58 & 8.93e-02 & 0.44 \\
    1681 & 1.58e-02 & 0.76 & 5.94e-02 & 0.59 \\
    6561 & 8.73e-03 & 0.86 & 3.66e-02 & 0.70 \\
    25921 & 4.64e-03 & 0.91 & 2.12e-02 & 0.79 \\
    \bottomrule
    \end{tabular}}%
    \caption{Uniform mesh}
  \end{subtable}
  \hfill
  \begin{subtable}[t]{0.48\linewidth}
    \centering
    \resizebox{\linewidth}{!}{%
    \begin{tabular}{rrlrl}
    \toprule
    $\bm{I}$ & $\bm{\delta^1(T)}$ & & $\bm{\delta^\infty(T)}$ & \\
    121 & 3.68e-02 &  & 1.72e-01 &  \\
    441 & 2.73e-02 & 0.43 & 1.35e-01 & 0.35 \\
    1681 & 1.61e-02 & 0.76 & 8.12e-02 & 0.74 \\
    6561 & 8.79e-03 & 0.87 & 4.91e-02 & 0.73 \\
    25921 & 4.68e-03 & 0.91 & 2.77e-02 & 0.83 \\
    \bottomrule
    \end{tabular}}%
    \caption{10\% mesh distortion}
  \end{subtable}
  \vspace{-1em}
  \caption{$\delta^1(T)$ and $\delta^\infty(T)$ errors and convergence
    rates  for the low-order method in 2D for the exact solution with
    nonzero gravity. We observe first-order convergence for both the (a)
    uniform mesh and (b) 10\% distorted mesh at a final time of $t_f = 1$ s
    and CFL=0.5.}
  \label{tab:low-order-results}
\end{table}
\begin{table}
\begin{center}
\begin{tabular}{rrlrlrl}
  \toprule
   $\bm{I}$ &   $\bm{\delta^1(T)}$ &      &   $\bm{\delta^2(T)}$ &      &   $\bm{\delta^\infty(T)}$ &      \\[0.25em]
  101 & 1.41e-04 &  & 1.63e-04 &  & 4.78e-04 &  \\
  201 & 3.60e-05 & 1.96 & 4.27e-05 & 1.94 & 1.42e-04 & 1.75 \\
  401 & 9.10e-06 & 1.99 & 1.09e-05 & 1.96 & 4.18e-05 & 1.77 \\
  801 & 2.28e-06 & 1.99 & 2.77e-06 & 1.98 & 1.21e-05 & 1.79 \\
  1601 & 5.72e-07 & 2.00 & 7.00e-07 & 1.99 & 3.36e-06 & 1.84 \\
  3201 & 1.43e-07 & 2.00 & 1.76e-07 & 1.99 & 9.14e-07 & 1.88 \\
  \bottomrule
\end{tabular}
\end{center}
\caption{$\delta^1(T)$, $\delta^2(T)$, and $\delta^\infty(T)$ errors and convergence rates for the 1D exact solution with a nonzero gravitational source at a final time of $t_f=5$ s and CFL 0.5.} \label{tab:exact_source}
\end{table}
\subsubsection{Well-Balancing Test on a Warped Mesh}\label{sec:warped_mesh}
\begingroup
This test assesses the well-balancing properties of the proposed numerical
fluxes on a curvilinear mesh, following~\cite{artiano2025structurepreserving}. We test this on our isentropic equilibrium state.

The mesh is constructed via a smooth curvilinear mapping of the reference
coordinates $(\xi, \eta)$ to physical space $(x, y)$:
\begin{equation*}
    x = 500\left(1 + \xi + 0.1\sin(\pi\xi)\sin(\pi\eta)\right), \quad
    y = 500\left(1 + \eta + 0.1\sin(\pi\xi)\sin(\pi\eta)\right),
\end{equation*}
resulting in the warped mesh shown in Figure~\ref{fig:warped_mesh}.
We set $\mathbf{v}_0(\bx) = 0$. For the adiabatic atmosphere, set the
initial potential temperature to be $\theta_0(\bx) = 300$~K,
with the initial density and pressure again satisfying hydrostatic equilibrium,
with thermodynamic constants $c_p = 1004$, $c_v = 717$, $R = c_p - c_v$,
and $g = 9.81$.
The mesh consists of only $16\times16$ cells ($289$ cG nodes), and the time
step is fixed to $\tau = 0.01$ s, i.e., $500\,000$ time steps up to
$T = 5000$ s. We compare (a) the low-order method of
Section~\ref{sec:low_order}, (b) the high-order
method~\eqref{eq:limited_high} without mass matrix correction
($b_{ij} = b_{ji} = 0$), (c) the full high-order method, and (d) the full
high-order method on a uniform $16\times16$ mesh, each with continuous and
with discontinuous $Q_1$ elements. Figure~\ref{fig:warped_error}
shows the $L^\infty$ error of the velocity, $\|\mathbf{v}_h(t)\|_{L^\infty(D)}$,
in double precision.

All variants are well balanced in exact arithmetic
(Theorem~\ref{thm:final_theorem}). Since the underlying PDE has no
smoothing property, round-off perturbations are not damped. The low-order
method (a) stays at $10^{-13}$. Variants (b) and (d) show an algebraic
growth of the round-off error, whereas the mass matrix correction on the
warped mesh, variant (c), amplifies it considerably for continuous elements.
For discontinuous elements, where the consistent mass matrix is inverted
exactly, all high-order variants remain below $10^{-7}$. The small fixed time step amplifies the round-off error: running the
continuous variants at CFL $0.5$ instead ($\tau \approx 0.05$ s) yields
final errors of about $10^{-8}$ for (b) and (d), and of
$2\times10^{-2}$ for (c). A similar growth was
reported in~\cite{artiano2025structurepreserving} for this test with
$16\times16$ elements of degree two; there, the equilibrium is maintained to
machine precision only with double-double arithmetic.

\begin{figure}
    \centering
    \renewcommand{\thesubfigure}{(\roman{subfigure})}%
    \captionsetup[subfigure]{labelformat=simple, labelsep=space}%
    \subfloat[Warped mesh\label{fig:warped_mesh}]{%
    \begin{tikzpicture}[x=0.0034cm, y=0.0034cm]
        \foreach \k in {0,...,16}{
            \pgfmathsetmacro{\c}{-1+\k/8}
            \draw[gray, line width=0.5pt] plot[domain=-1:1, samples=33, smooth, variable=\t]
                ({500*(1+\c+0.1*sin(180*\c)*sin(180*\t))},
                 {500*(1+\t+0.1*sin(180*\c)*sin(180*\t))});
            \draw[gray, line width=0.5pt] plot[domain=-1:1, samples=33, smooth, variable=\t]
                ({500*(1+\t+0.1*sin(180*\t)*sin(180*\c))},
                 {500*(1+\c+0.1*sin(180*\t)*sin(180*\c))});
        }
    \end{tikzpicture}}%
    \hfill
    \subfloat[Velocity error\label{fig:warped_error}]{%
    \begin{tikzpicture}
        \begin{semilogyaxis}[
            width=8.4cm,
            height=4.8cm,
            xlabel={$t$ [s]},
            ylabel={$\|\mathbf{v}_h(t)\|_{L^\infty(D)}$},
            xmin=0, xmax=5000,
            xtick={0,1000,...,5000},
            ytick={1e-15,1e-10,1e-5,1e0},
            scaled x ticks=false,
            xticklabel style={/pgf/number format/1000 sep={}},
            tick label style={font=\footnotesize},
            label style={font=\footnotesize},
            unbounded coords=discard,
            legend columns=6,
            legend cell align=left,
            legend style={at={(0.5,1.03)}, anchor=south, font=\footnotesize,
                draw=none, /tikz/every even column/.append style={column sep=4pt}},
            grid=both,
            grid style={dashed, gray!30},
        ]
        \addlegendimage{blue, thick}\addlegendentry{(a)}
        \addlegendimage{orange, thick}\addlegendentry{(b)}
        \addlegendimage{teal, thick}\addlegendentry{(c)}
        \addlegendimage{violet, thick}\addlegendentry{(d)}
        \addlegendimage{black, thick}\addlegendentry{cG}
        \addlegendimage{black, thick, dashed}\addlegendentry{dG}
        \addplot[blue, thick, forget plot] table[x index=0, y index=1, comment chars={\#}]
            {well-balanced-adiabatic-low_order-double.txt};
        \addplot[orange, thick, forget plot] table[x index=0, y index=1, comment chars={\#}]
            {well-balanced-adiabatic-high_order-no_neumann-double.txt};
        \addplot[teal, thick, forget plot] table[x index=0, y index=1, comment chars={\#}]
            {well-balanced-adiabatic-high_order-double.txt};
        \addplot[violet, thick, forget plot] table[x index=0, y index=1, comment chars={\#}]
            {well-balanced-adiabatic-high_order-double-uniform.txt};
        \addplot[blue, thick, dashed, forget plot] table[x index=0, y index=1, comment chars={\#}]
            {well-balanced-adiabatic-dg-low_order.txt};
        \addplot[orange, thick, dashed, forget plot] table[x index=0, y index=1, comment chars={\#}]
            {well-balanced-adiabatic-dg-high_order-no_neumann.txt};
        \addplot[teal, thick, dashed, forget plot] table[x index=0, y index=1, comment chars={\#}]
            {well-balanced-adiabatic-dg-high_order.txt};
        \addplot[violet, thick, dashed, forget plot] table[x index=0, y index=1, comment chars={\#}]
            {well-balanced-adiabatic-dg-high_order-uniform.txt};
        \end{semilogyaxis}
    \end{tikzpicture}}
    \caption{Well-balancing test: (i) warped mesh obtained
    from the curvilinear mapping of~\cite{artiano2025structurepreserving}
    ($16\times16$ cells); (ii) $L^\infty$ error of the velocity over time
    in double precision for the variants (a)--(d) with cG $Q_1$ (solid)
    and dG $Q_1$ (dashed) elements.}
    \label{fig:warped}
\end{figure}
\endgroup

\subsubsection{Pressure Perturbation over 1D Isothermal Steady State}\label{isothermal_test}
We now address the Remark~\ref{rem:isothermal}
and test our method with a benchmark that has isothermal equilibrium background state.
We follow the set up discussed in~\cite{navasmontilla2023atmospheric, li2021gravitational} which is composed of isothermal background state and  small perturbation to the pressure.
We recall that a typical isothermal problem set up is given by:
\begin{equation}
    \begin{split}
    \label{isothermal_eq}
    \bv\cdot\mathbf{\hat{e}}_d = 0,\quad
    \bv\cdot\mathbf{\hat{e}}_1, ...,\bv\cdot\mathbf{\hat{e}}_{d-1} = const., \quad T = const.
    \\
    \rho = \rho_0\exp\Big(-\frac{x_d}{RT_0}\Big),
    \quad
    p = \rho RT_0 = p_0\exp\Big(-\frac{x_d}{RT_0}\Big),
    \quad
    g=1.
    \end{split}
\end{equation}
Let $\hat{x} = x - 2$. The initial condition is given by:
\begin{subequations}
\begin{align}
    \rho(x,0) &= \rho^e(\hat{x}) \\
    p(x,0) &= p^e(\hat{x}) + \eta\exp{(-100(\hat{x}-0.5)^2)} \\
    v(x,0) &= 0,
\end{align}
\end{subequations}
with constants
\begin{align}
    \rho_0 = 1,  p_0 = 1,  T_0 = 1, R = 1, c_v = 2.5, \gamma =1.4, g=1.
\end{align}
Since we use the potential temperature formulation, we solve for the variable $\theta$ in this work. The initial potential temperature is given by:
\begin{equation}
    \theta(x,0) = \frac{p_0}{R\rho^e(\hat{x})}\cdot\left[\frac{p^e(\hat{x}) + \eta\exp{(-100(\hat{x}-0.5)^2)}}{p_0}
    \right]^\frac{1}{\gamma}.
\end{equation}
Let the computational domain be $D=[0,5]$. We compare the solution for $\eta=0.01$ at a final time of $t=0.8$ with CFL $0.05$ in Figure~\ref{fig:pressure_perturbation} at different refinement levels. It is compared against a reference solution computed with 3200 cells. Our scheme is able to accurately capture the propagation of the waves produced due to the initial perturbation of the pressure. Our results also agree well with the literature~\cite{navasmontilla2023atmospheric, li2021gravitational, ghosh2016benchmarks}.
\begin{figure}
\centering
\begin{tikzpicture}
\begin{axis}[
    name=mainplot,
    xmin=1.0, xmax=4.0,
    xlabel={x (m)},
    ylabel={$\Delta p$},
    grid=minor,
    width=12.5cm,
    height=8.0cm,
    legend pos=north east]
\addplot[very thick, black, dashed]
    table [col sep=comma, x="Points:0", y="delta_p"]
    {ref_pressure_perturb.txt};
\addlegendentry{reference}
\addplot[thin, orange]
    table [col sep=comma, x="Points:0", y="delta_p"]
    {N200_pressure_perturb.txt};
\addlegendentry{N = 200}
\addplot[thin, cyan]
    table [col sep=comma, x="Points:0", y="delta_p"]
    {N400_pressure_perturb.txt};
\addlegendentry{N = 400}
\addplot[thin, red]
    table [col sep=comma, x="Points:0", y="delta_p"]
    {N800_pressure_perturb.txt};
\addlegendentry{N = 800}

\draw[black, thick]
  (axis cs:1.5,0.0065) rectangle (axis cs:1.6,0.00755);

\coordinate (zoomNE) at (axis cs:1.6,0.00755);
\coordinate (zoomSE) at (axis cs:1.6,0.0065);
\end{axis}

\begin{axis}[
    at={(4.5cm,3cm)},
    anchor=south west,
    width=4.0cm,
    height=4.0cm,
    xmin=1.5, xmax=1.6,
    ymin=0.0065, ymax=0.00755,
    grid=minor,
    axis background/.style={fill=white},
    xticklabel style={font=\scriptsize},
    yticklabel style={font=\scriptsize}
]

\addplot[very thick, black, dashed]
    table [col sep=comma, x="Points:0", y="delta_p"]
    {ref_pressure_perturb.txt};

\addplot[thin, orange]
    table [col sep=comma, x="Points:0", y="delta_p"]
    {N200_pressure_perturb.txt};

\addplot[thin, cyan]
    table [col sep=comma, x="Points:0", y="delta_p"]
    {N400_pressure_perturb.txt};

\addplot[thin, red]
    table [col sep=comma, x="Points:0", y="delta_p"]
    {N800_pressure_perturb.txt};

\coordinate (insetNW) at (rel axis cs:0,1);
\coordinate (insetSW) at (rel axis cs:0,0);
\end{axis}

\draw[black, thin] (zoomNE) -- (insetNW);
\draw[black, thin] (zoomSE) -- (insetSW);

\end{tikzpicture}
\caption{Pressure perturbation at $t=0.8$ with $\eta=0.01$ for different refinement levels and a reference solution of 3200 cells. The scheme accurately captures the propagating waves due to the initial perturbation of the pressure. A zoom in of the left peak is provided for a better view of convergence.}
\label{fig:pressure_perturbation}
\end{figure}

As stated in Remark~\ref{rem:isothermal}, our scheme preserves isothermal equilibrium states exactly, \ie up to round-off errors. We illustrate this fact in Table~\ref{iso_wb}.
\begin{table}
\begin{center}
\begin{tabular}{rrrr}
\toprule
   $\bm{I}$ &   $\bm{\delta^1(T)}$ &   $\bm{\delta^2(T)}$ &   $\bm{\delta^\infty(T)}$ \\[0.25em]
   101 & 2.04e-14 & 2.56e-14 & 7.25e-14 \\
   201 & 3.16e-14 & 4.60e-14 & 2.46e-13 \\
   401 & 4.80e-14 & 5.72e-14 & 2.57e-13 \\
   801 & 7.42e-14 & 9.06e-14 & 4.05e-13 \\
  1601 & 1.15e-13 & 1.32e-13 & 5.46e-13 \\
  3201 & 1.36e-13 & 1.45e-13 & 6.54e-13 \\
\bottomrule
\end{tabular}
\end{center}\caption{Errors at $T=0.8$ for the well-balancing test with the isothermal background state.}
\label{iso_wb}
\end{table}

\subsubsection{Sod Shock Problem}
 This problem is similar to the Riemann problem, but it includes the source term. We will use this to test how well our method accurately captures shocks and contact discontinuities. The domain is $D=[0,1]$ with slip boundary conditions. Let $w_Z = (\rho_Z , v_Z , \theta_Z )^T$ be the primitive data for our problem. We use the following set of Riemann data for the problem set up:
\begin{equation}
    (\rho,v,\theta) =
    \begin{cases}
        (1, 0, 1) & x< 0.5 \\
        (0.125, 0, 1.54) & x\geq 0.5,
    \end{cases}
\end{equation}

with thermodynamic constants
\begin{align*}
    p_0 = 1, c_v = 2.5, \gamma =1.4, g=1.
\end{align*}
The solution is presented in Figure~\ref{fig:sod_shock}. The solution is evaluated at time $t_f=0.2$ with 200 cells, and a reference solution is computed with a fine mesh of 3200 cells. We can see that the high-order scheme accurately captures the contact and shock waves that arise from the discontinuous initial data.

We also verify that our conserved quantities stay conserved up to machine precision. To this end, we use the warped mesh of Section~\ref{sec:warped_mesh} with $64\times64$ cells, periodic boundary conditions in the horizontal direction, and slip boundary conditions at the top and bottom. The isentropic equilibrium with $\theta=300$~K and $g=9.81$ is perturbed by a warm bubble, $\theta=300+2(1-r/200)$ for $r=|\bx-(500,300)|\le200$, where the density is evaluated as in Section~\ref{bubble}, and we run up to $t=200$~s. Table~\ref{tab:conservation_check} shows that the total mass and the total $\rho\theta$ are conserved to all printed digits. The total momentum changes due to gravity, in agreement with Proposition~\ref{prop:low_order_method}\,(iii).

\begin{table}[t]
\centering\small
\begin{tabular}{ccccc}
\hline
$\bm{t}$ & $\int_D\rho\diff x$ & $\int_D m_1\diff x$ & $\int_D m_2\diff x$ & $\int_D\rho\theta\diff x$ \\[0.5em]
$0$   & $1.11464602519850\times10^{6}$ & $0$ & $0$ & $3.34488525245695\times10^{8}$ \\
$200$ & $1.11464602519850\times10^{6}$ & $-1.97\times10^{0}$ & $4.08\times10^{2}$ & $3.34488525245695\times10^{8}$ \\
\hline
\end{tabular}\caption{Total mass, momentum, and $\rho\theta$ at the initial and final time for the perturbed isentropic equilibrium on the warped mesh with gravity.} \label{tab:conservation_check}
\end{table}

\begin{figure}
    \centering
    \begin{subfigure}[b]{0.32\linewidth}
        \centering
        \hspace{-0.9cm}
        \begin{tikzpicture}
            \begin{axis}[
                title={Density},
                xmin=0, xmax=1,
                xlabel={x (m)},
                ylabel={},
                grid=major,
                width=5.1cm,
                height=5cm,
                legend pos=north east
            ]
            \addplot[very thick, black, dashed] table [col sep=comma, x="Points:0", y="rho"] {density_ref_sod_shock.txt};
            \addlegendentry{reference}
            \addplot[red, solid] table [col sep=comma, x="Points:0", y="rho"] {density_sod_shock.txt};
            \addlegendentry{density}
            \end{axis}
        \end{tikzpicture}
        \caption{Density at time $t_f$}
    \end{subfigure}
    \hfill
    \begin{subfigure}[b]{0.32\linewidth}
        \centering
        \hspace{-0.9cm}
        \begin{tikzpicture}
            \begin{axis}[
                title={Velocity},
                xmin=0, xmax=1,
                ymin=-0.5, ymax=1.5,
                xlabel={x (m)},
                ylabel={},
                grid=major,
                width=5.1cm,
                height=5cm,
                legend pos=north east
            ]
            \addplot[very thick, black, dashed] table [col sep=comma, x="Points:0", y="v"] {velocity_ref_sod_shock.txt};
            \addlegendentry{reference}
            \addplot[red, solid] table [col sep=comma, x="Points:0", y="v"] {velocity_sod_shock.txt};
            \addlegendentry{velocity}
            \end{axis}
        \end{tikzpicture}
        \caption{Velocity at time $t_f$}
    \end{subfigure}
    \hfill
    \begin{subfigure}[b]{0.32\linewidth}
        \centering
        \hspace{-0.9cm}
        \begin{tikzpicture}
            \begin{axis}[
                title={Pressure},
                xmin=0, xmax=1,
                xlabel={x (m)},
                ylabel={},
                grid=major,
                width=5.1cm,
                height=5cm,
                legend pos=north east
            ]
            \addplot[very thick, black, dashed] table [col sep=comma, x="Points:0", y="p"] {pressure_ref_sod_shock.txt};
            \addlegendentry{reference}
            \addplot[red, solid] table [col sep=comma, x="Points:0", y="p"] {pressure_sod_shock.txt};
            \addlegendentry{pressure}
            \end{axis}
        \end{tikzpicture}
        \caption{Pressure at time $t_f$}
    \end{subfigure}
    \caption{Density, velocity, and pressure for the sod shock problem at $t_f=0.2$ for $N=200$ cells. The data is plotted against a reference solution from a fine mesh of $N=3200$ cells.}
    \label{fig:sod_shock}
\end{figure}

\subsection{Benchmarks}\label{sec:validation}

We now illustrate the validity of our method using benchmark problems in the atmospheric flow literature. The CFL is set to 0.9 for all of these problems.

\subsubsection{2D Rising Thermal Bubble}\label{bubble}
We consider the rising thermal bubble in 2D proposed by \cite{robert1993bubble}. This test is a standard benchmark problem in the numerical weather prediction community, see \cite{ahmad2007benchmarks, duarte2014moistflows}. The problem introduces a warm bubble by perturbing the potential temperature. This produces vertical momentum causing the bubble to rise and later deform into a mushroom-like shape.
The given computational domain is $D=[5000,15000]\times [0,10000]m^2$ and free-slip boundary conditions are imposed on all sides of the domain. Let $R=2000m$. We set $\mathbf{v}_0(\bx) = 0$, and let $\rho_0(\bx)$ be given by the isentropic profile~\eqref{adiabatic_eq}, evaluated pointwise with the local potential temperature $\theta_0(\bx)$. Hence, the initial state is in hydrostatic equilibrium outside the bubble but deliberately not inside it; this imbalance drives the motion. Set the potential temperature to be
\begin{equation}
    \theta_0(\bx) =
    \begin{cases}

        300 + 2[1 - \frac{r_0}{2000}] & \text{if } (x-10000)^2+(y-2000)^2 \leq R^2 \\
        300 & \text{if } (x-10000)^2+(y-2000)^2>R^2,
    \end{cases}
\end{equation}
with $r_0 = \sqrt{(x-10000)^2+(y-2000)^2}$ and thermodynamic constants
\begin{align*}
    p_0 = 100000, c_v = 715, \gamma =1.4, g=9.8.
\end{align*}
We ran to a final time of $t_f=1020s$.
\begin{figure}
    \centering
        \subfloat{\adjustbox{width=0.6\linewidth, valign=b}{\includegraphics[trim={200bp 100bp 250bp 100bp}]{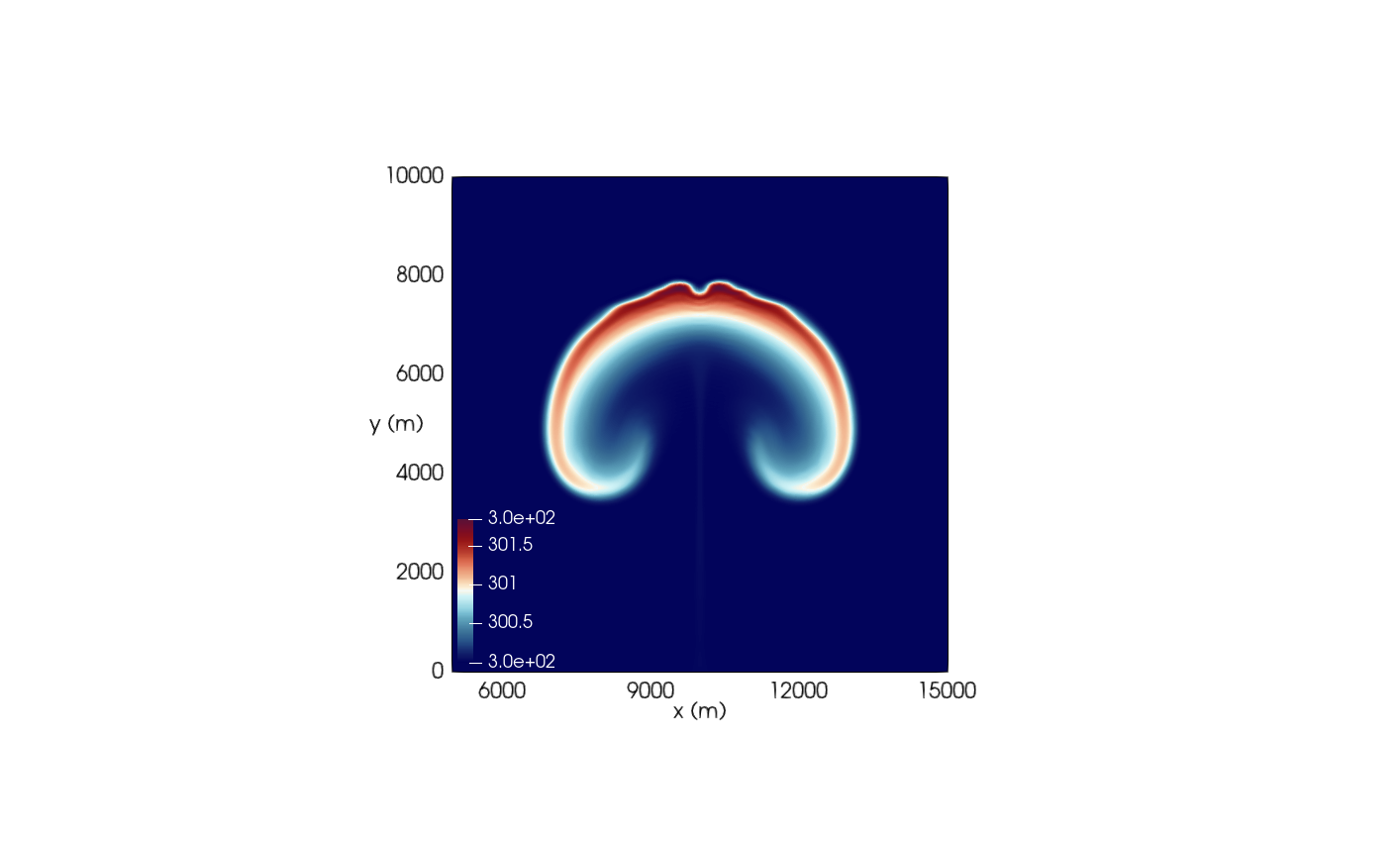}}}
        \hspace{-2.9cm}
        \subfloat{\adjustbox{width=0.6\linewidth, valign=b}{\includegraphics[trim={200bp 100bp 250bp 100bp}]{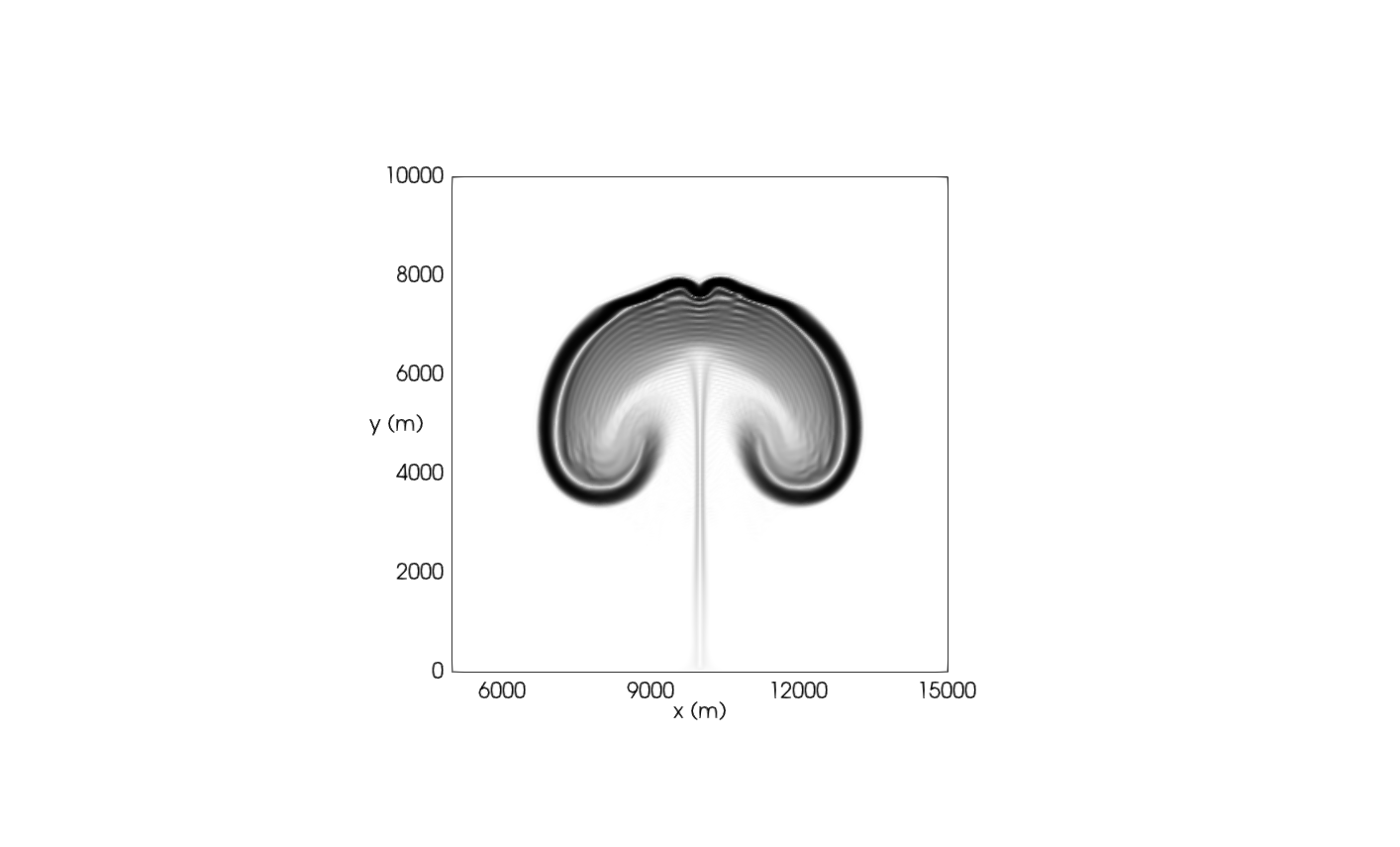}}}
    \caption{2D rising thermal bubble with a mesh size of $h=50$ m. A plot of the potential temperature perturbation (left) and the corresponding schlieren plot (right) are given at the final time $t_f=1020$ s.}
    \label{fig:rising_bubble}
\end{figure}

\subsubsection{2D Density Current}
We next test our method by modeling the 2-dimensional density current proposed by \cite{carpenter1990ppm}. This problem is another common benchmark in the numerical weather prediction community, and it models a cold air bubble descending and then propagating across the bottom boundary. This is done similarly to that in Section~\ref{bubble} by slightly perturbing the potential temperature inside the bubble and making it colder. We set the computational domain up as in \cite{MARRAS201577}, with $D=[0,25600]\times[0,6400]$. Let the (dimensionless) radius be $r=1$. We set $\mathbf{v}_0(\bx) = 0$, and let $\rho_0(\bx)$ be given as in Section~\ref{bubble}, so that the initial state is again deliberately out of hydrostatic equilibrium inside the bubble. Set the potential temperature to be
\begin{equation*}
    \theta_0(\bx) =
    \begin{cases}
        300 - \frac{15}{2}[1+\cos{\pi r_0}] & \text{if } (\frac{x}{4000})^2+(\frac{y-3000}{2000})^2 \leq r^2 \\
        300 & \text{if } (\frac{x}{4000})^2+(\frac{y-3000}{2000})^2>r^2,
    \end{cases}
\end{equation*}
with $r_0 = \sqrt{(\frac{x}{4000})^2+(\frac{y-3000}{2000})^2 }$ and thermodynamic constants
\begin{align*}
    p_0 = 100000, c_v = 717.5, \gamma =1.4, g=9.806.
\end{align*}
The result at time $t_f=900$ s is seen in Figure \ref{fig:density_comparison} (a) for a mesh size of $h=50$ in both the $x$ and $y$ directions. The front location of the current agrees with that of \cite{MARRAS201577,GIRFOGLIO2025106510, ahmad2007benchmarks}, as well at the location of the rotors. The density current for various mesh refinements is provided in Figure \ref{fig:density_comparison} over time from $t_0=0$ s through $t_f=900$ s.

\begin{figure}[htbp]
\centering

\begin{subfigure}{\textwidth}
\centering
\includegraphics[width=0.49\linewidth,trim={125bp 200bp 50bp 250bp}]{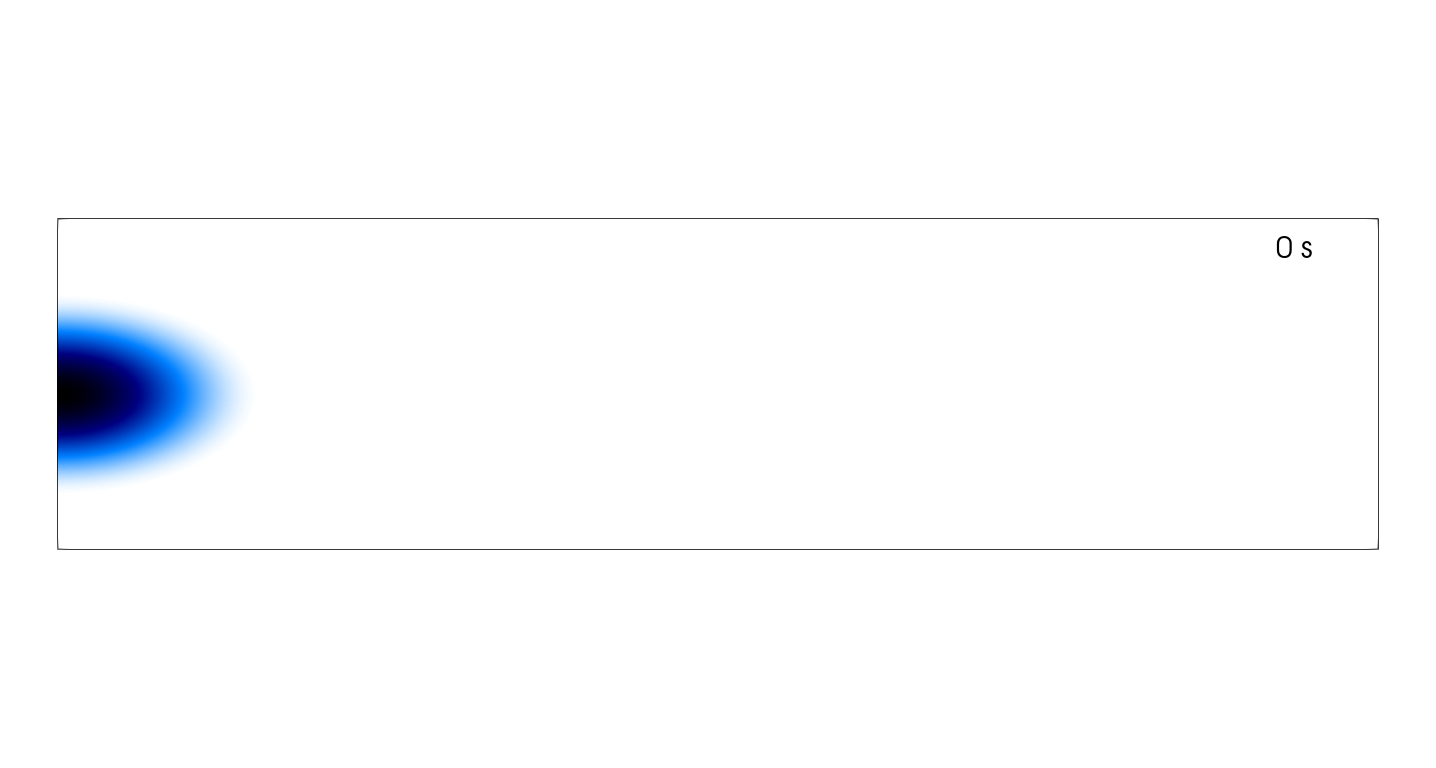}
\hfill
\includegraphics[width=0.49\linewidth,trim={125bp 200bp 50bp 250bp}]{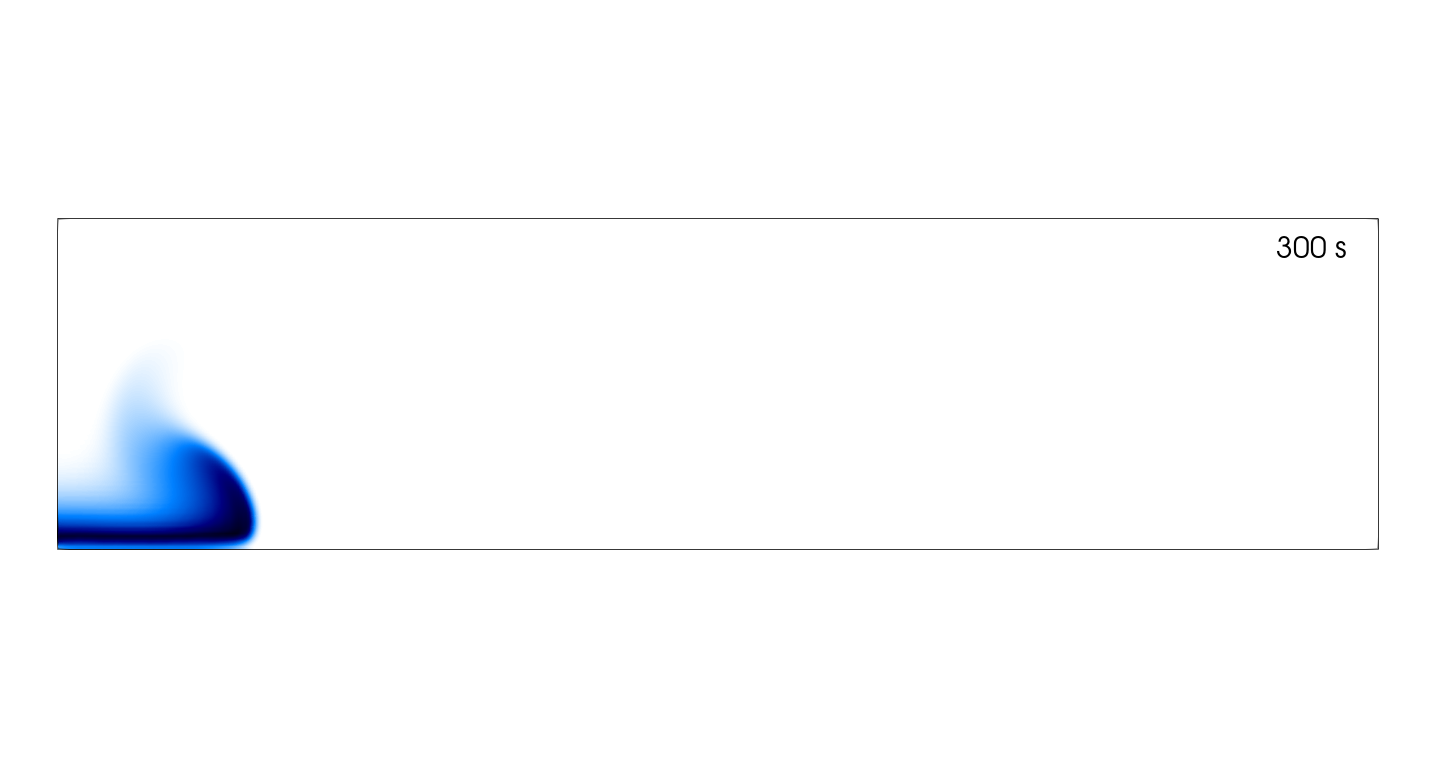}
\newline
\includegraphics[width=0.49\linewidth,trim={125bp 200bp 50bp 250bp}]{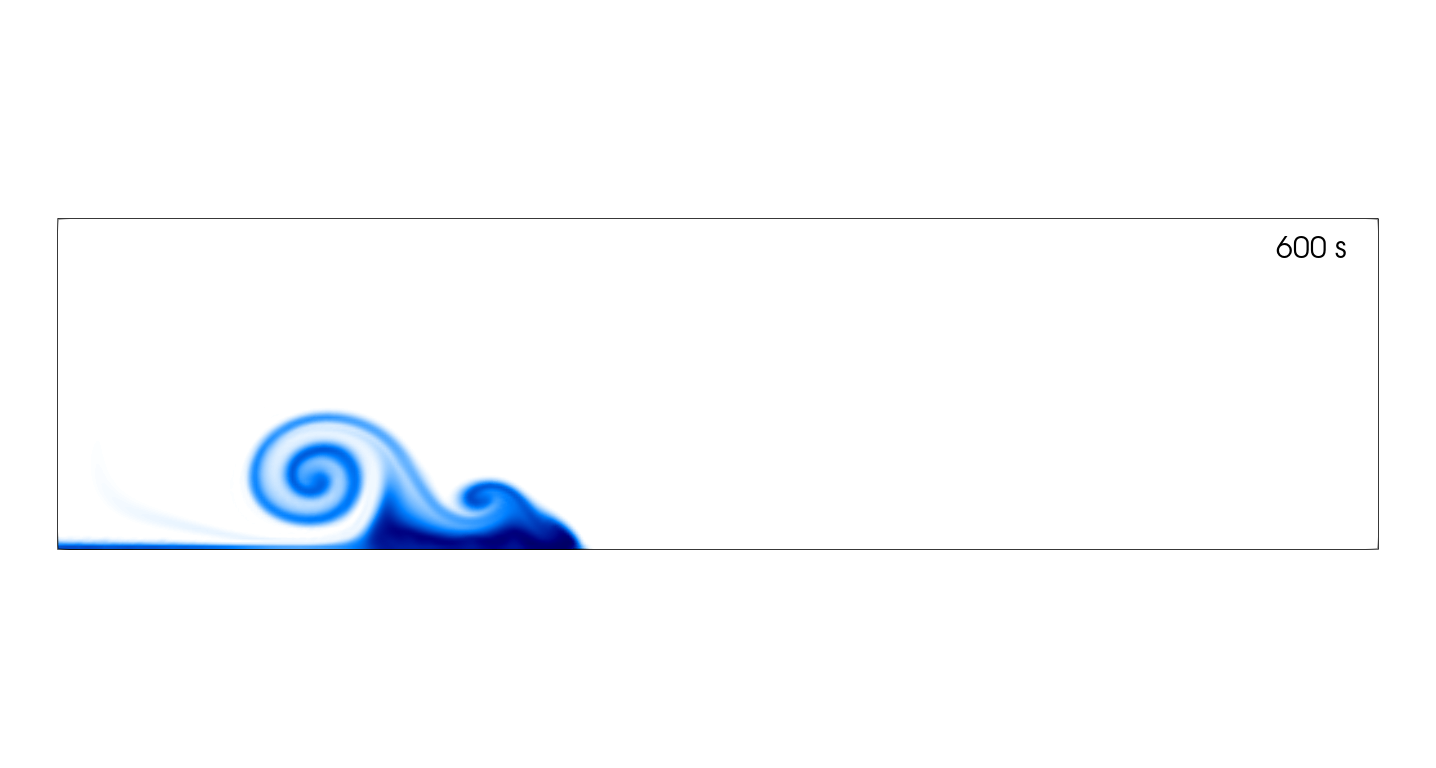}
\hfill
\includegraphics[width=0.49\linewidth,trim={125bp 200bp 50bp 250bp}]{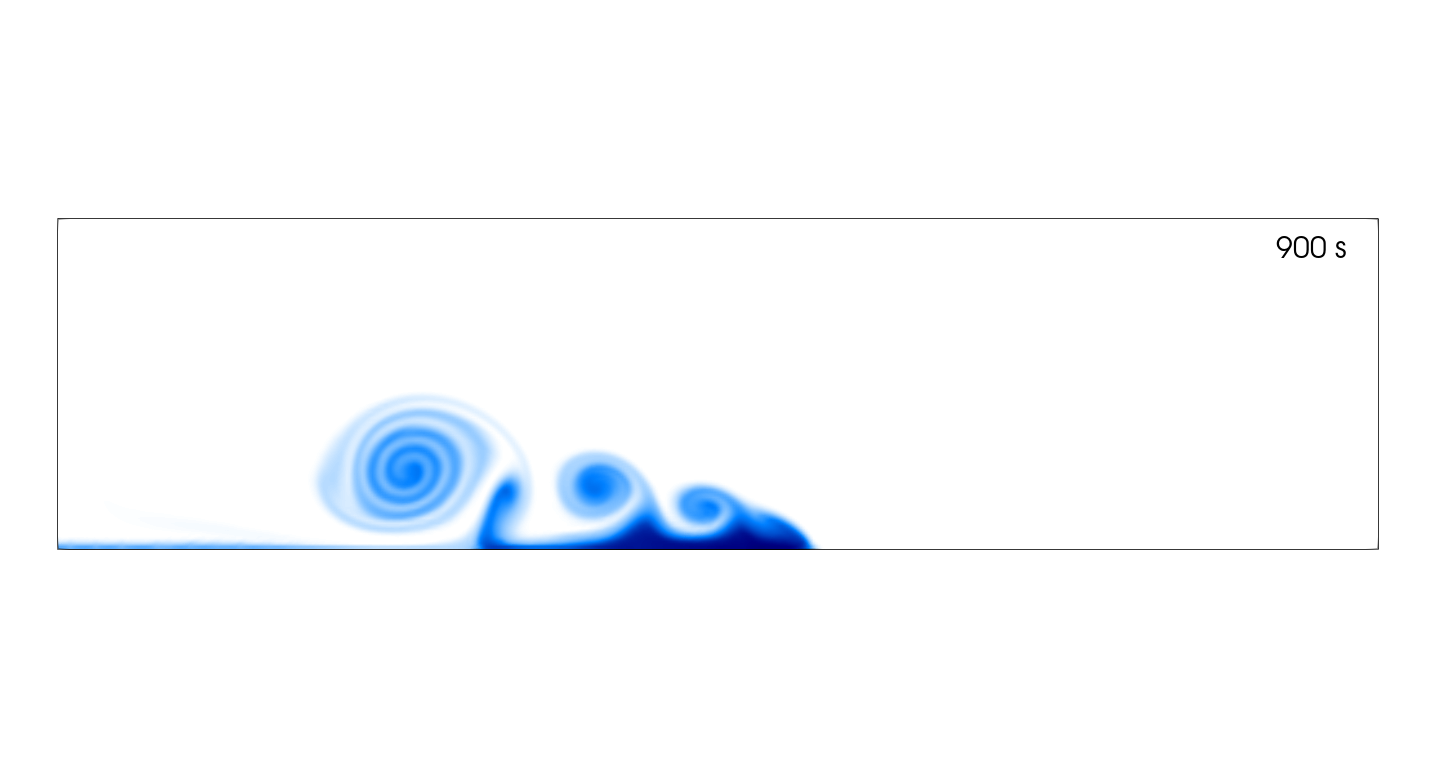}

\caption{$h=50$ m}
\label{fig:density_50}
\end{subfigure}

\vspace{0.8cm}

\begin{subfigure}{\textwidth}
\centering
\includegraphics[width=0.49\linewidth,trim={225bp 250bp 150bp 250bp}]{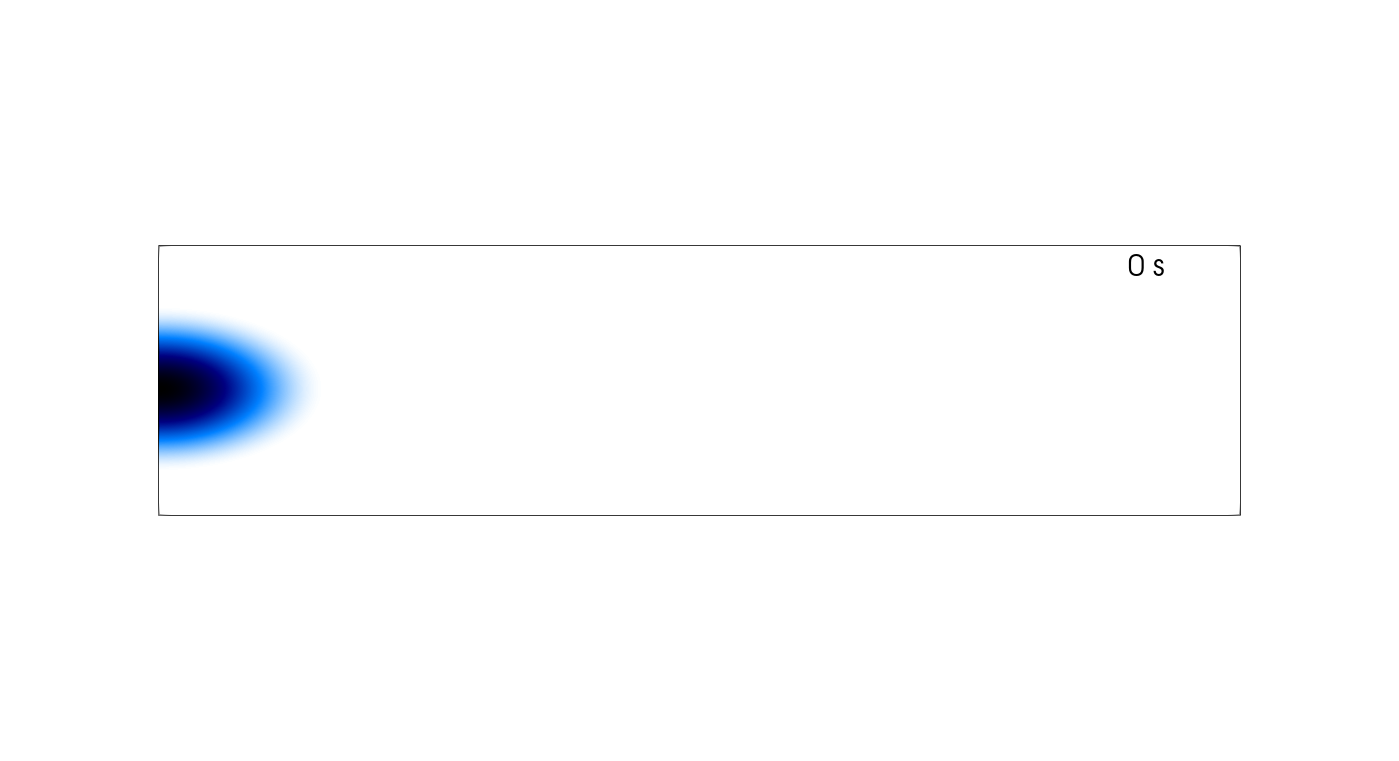}
\hfill
\includegraphics[width=0.49\linewidth,trim={225bp 250bp 150bp 250bp}]{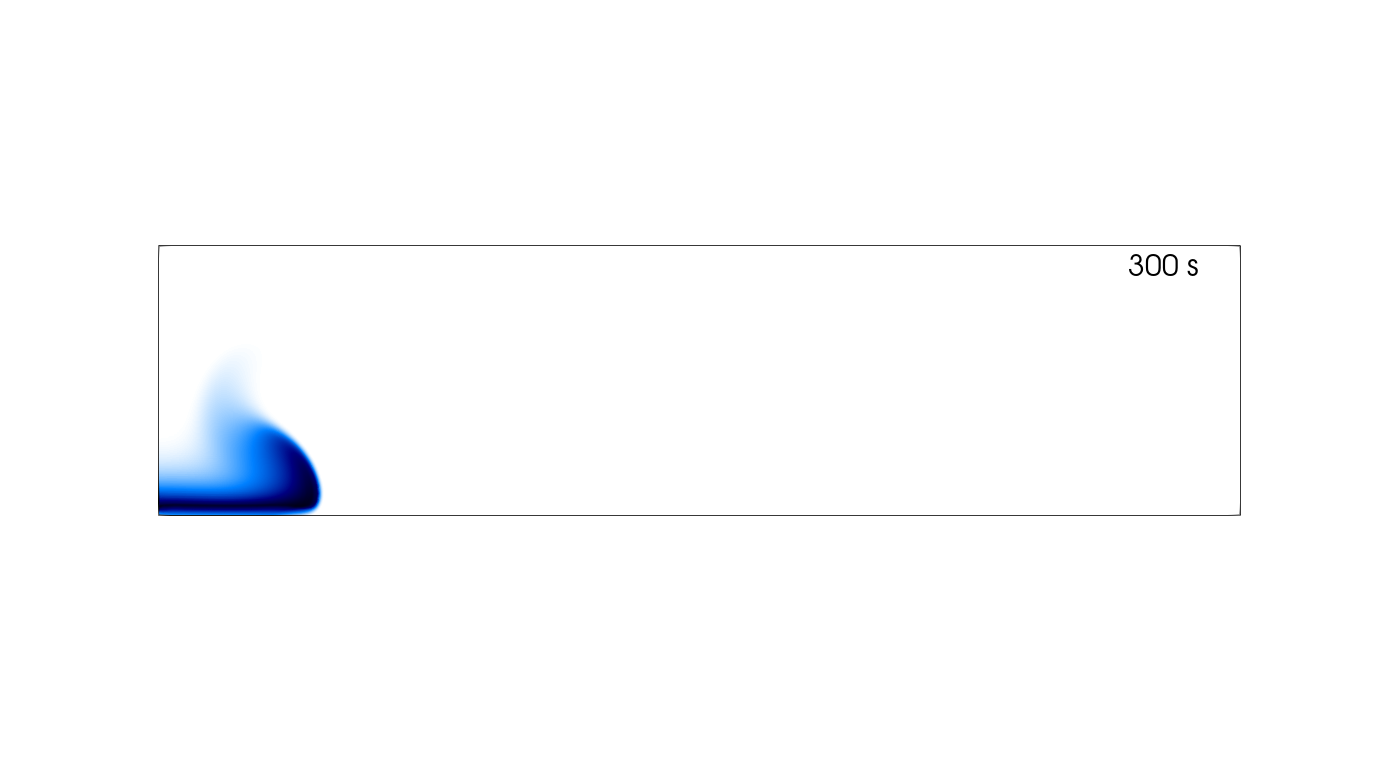}
\newline
\includegraphics[width=0.49\linewidth,trim={225bp 250bp 150bp 250bp}]{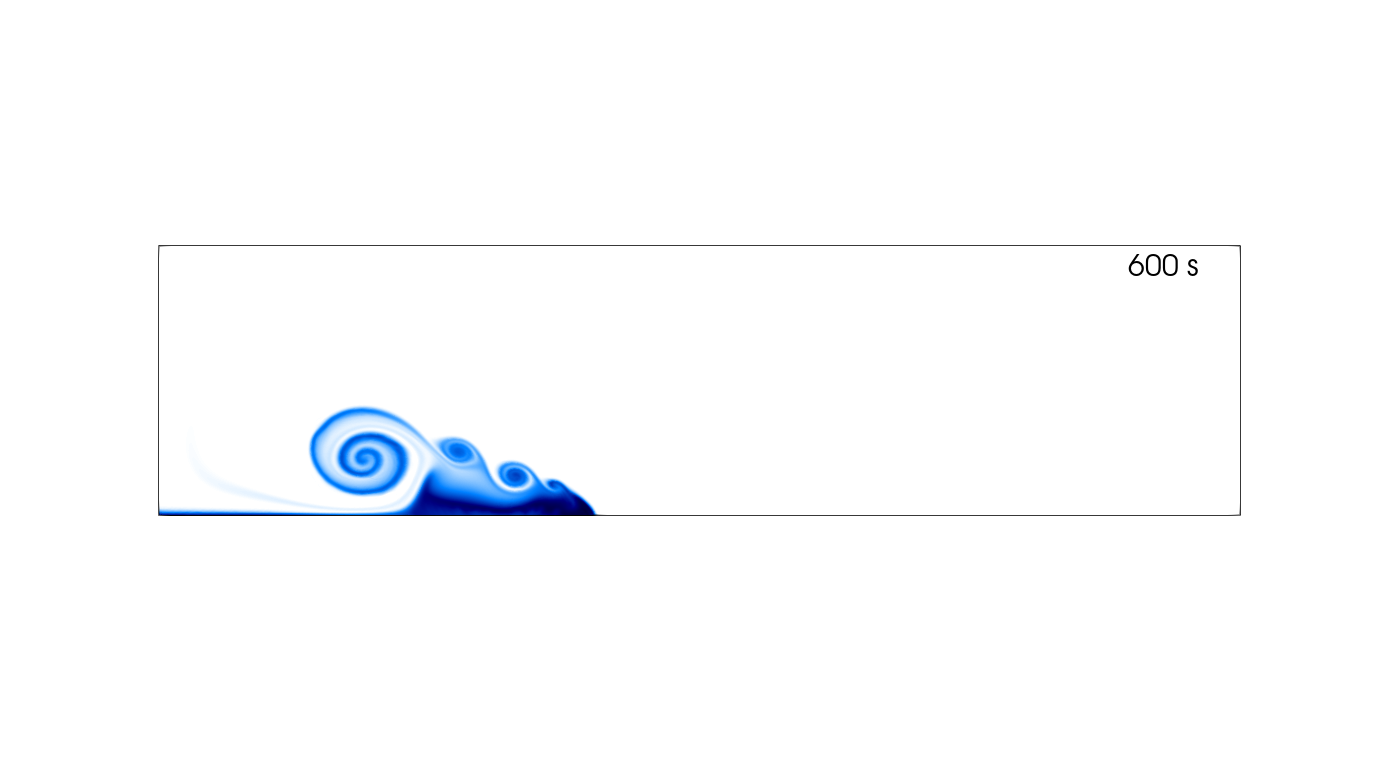}
\hfill
\includegraphics[width=0.49\linewidth,trim={225bp 250bp 150bp 250bp}]{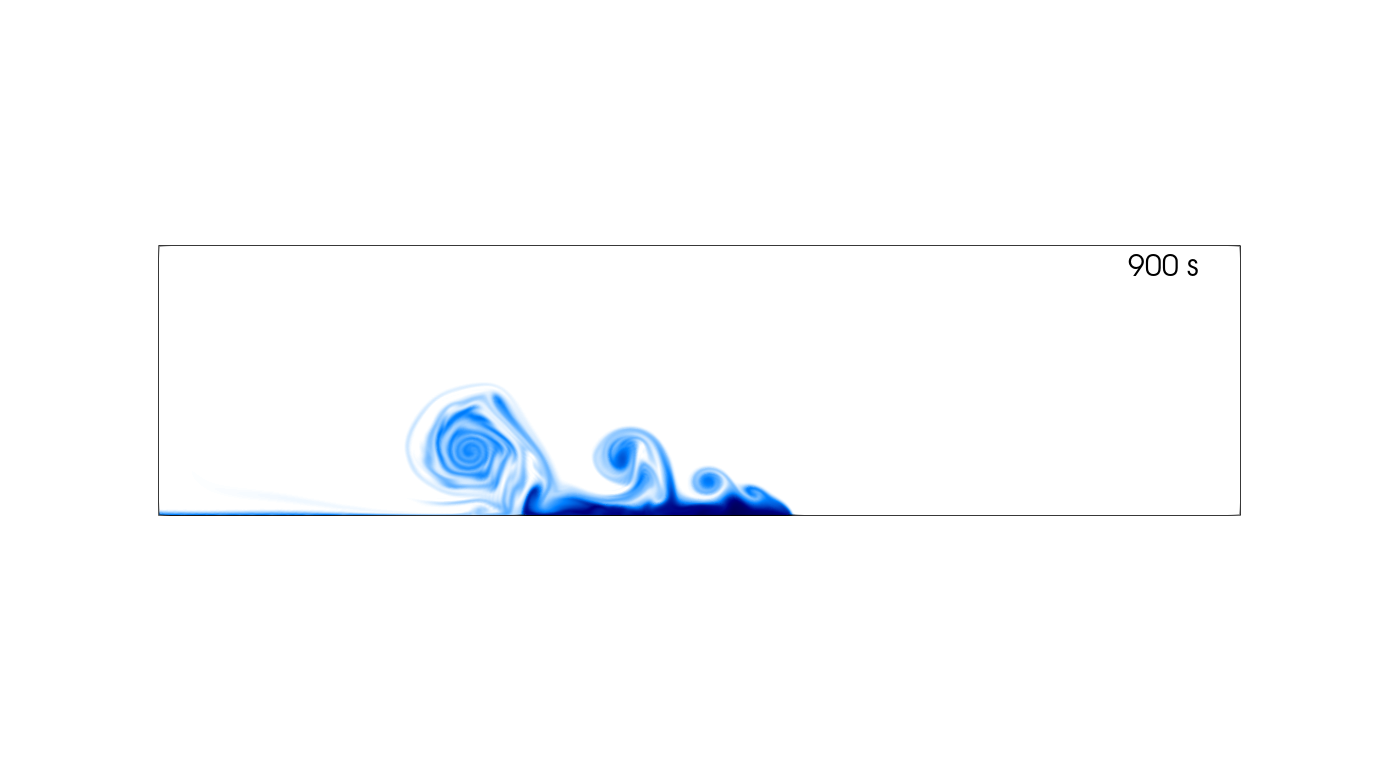}

\caption{$h=25$ m}
\label{fig:density_25}
\end{subfigure}

\vspace{0.8cm}

\begin{subfigure}{\textwidth}
\centering
\includegraphics[width=0.49\linewidth,trim={200bp 260bp 125bp 250bp}]{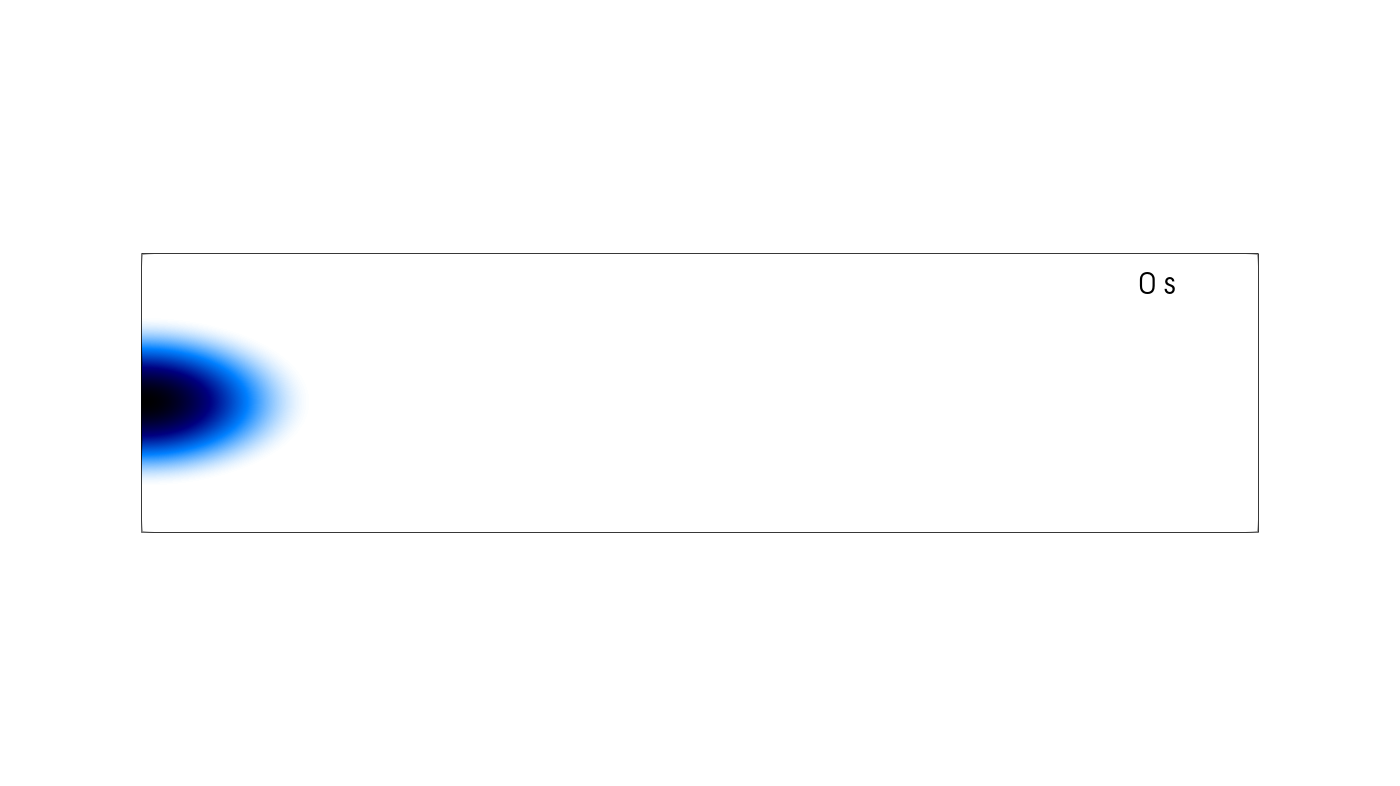}
\hfill
\includegraphics[width=0.49\linewidth,trim={200bp 260bp 125bp 250bp}]{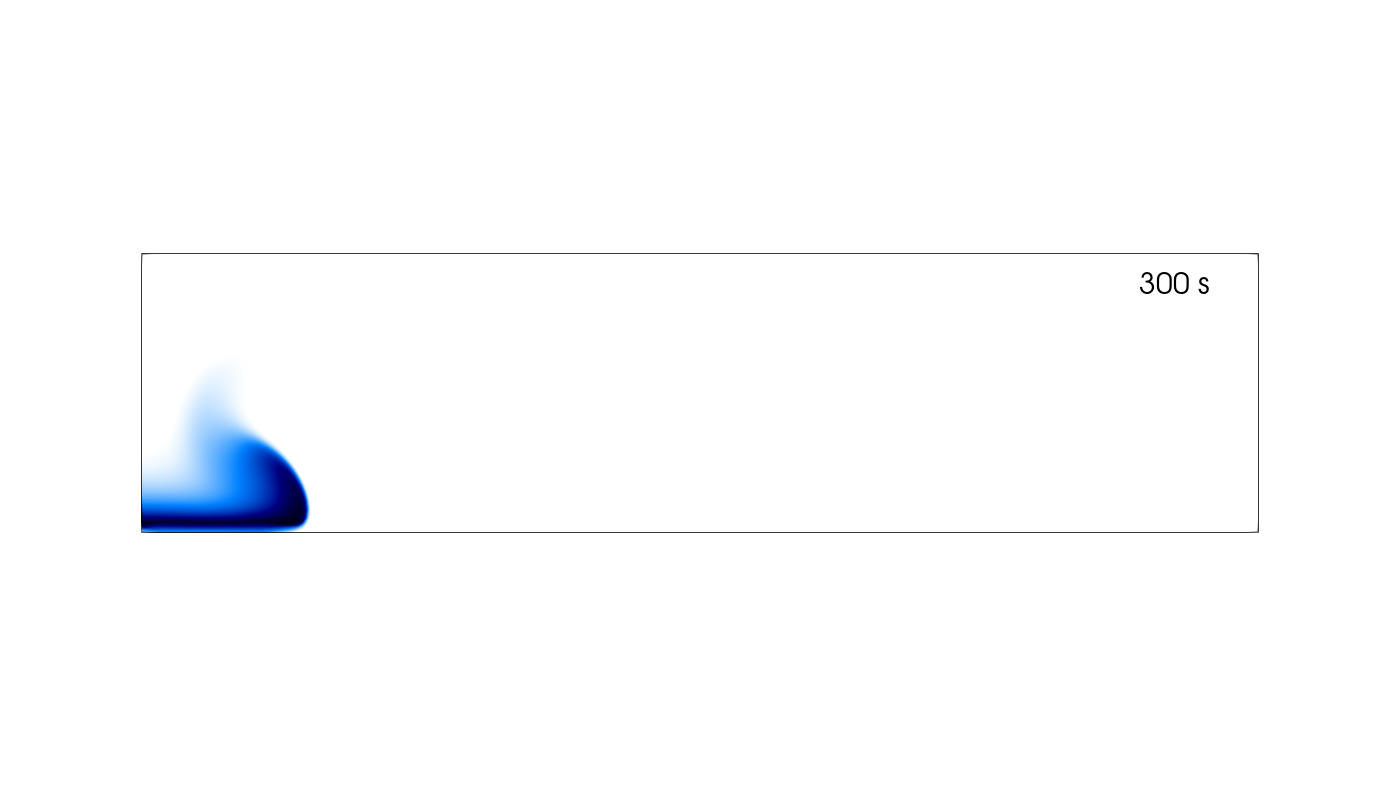}
\newline
\includegraphics[width=0.49\linewidth,trim={200bp 260bp 125bp 250bp}]{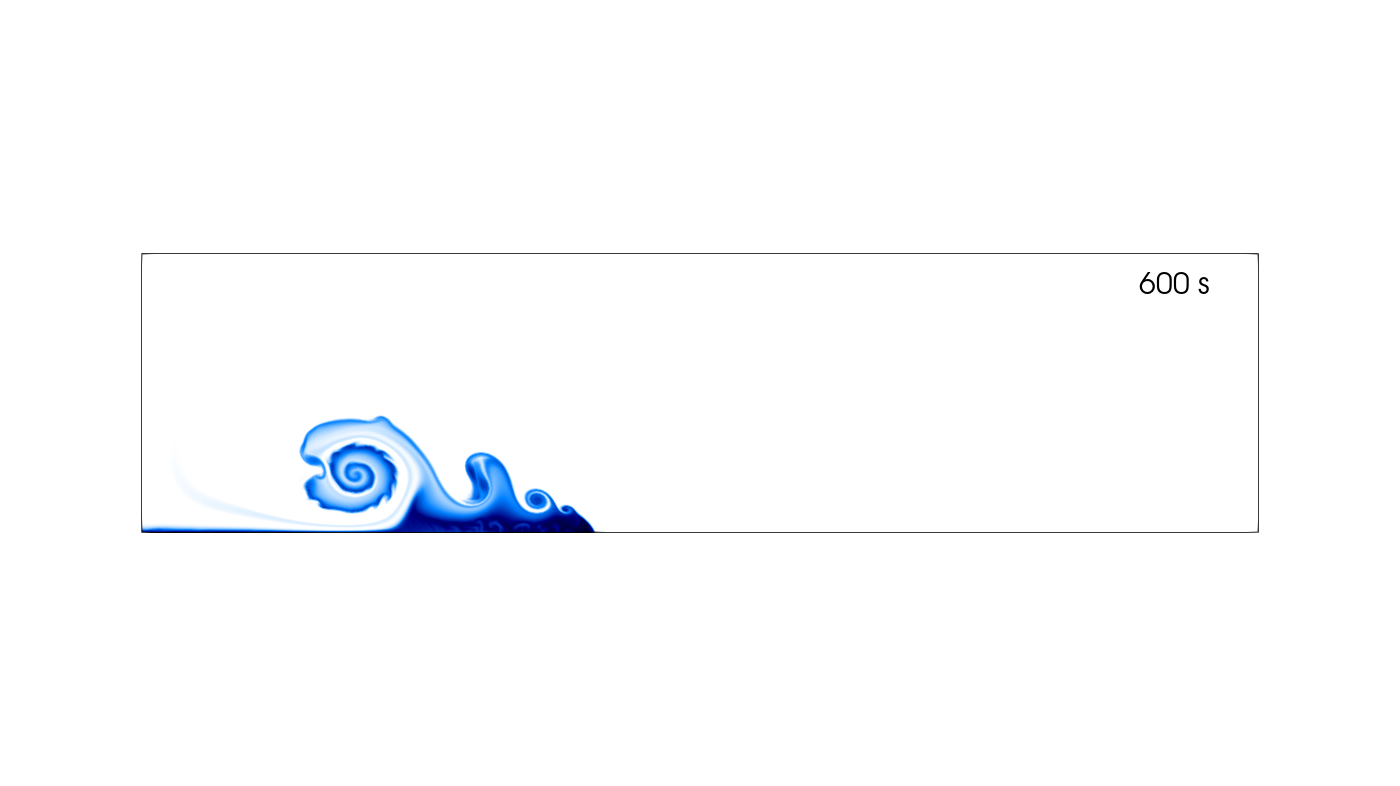}
\hfill
\includegraphics[width=0.49\linewidth,trim={200bp 260bp 125bp 250bp}]{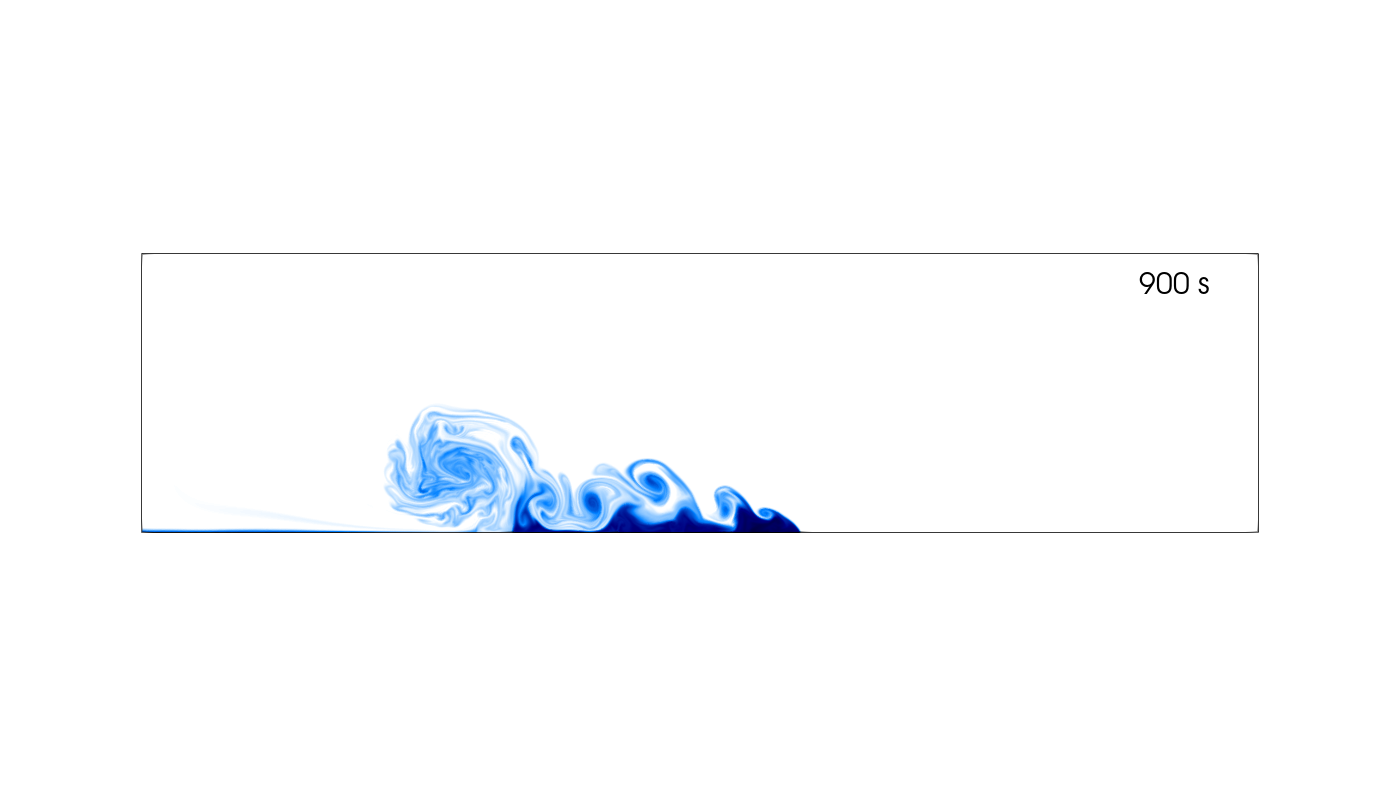}

\caption{$h=12.5$ m}
\label{fig:density_12pt5}
\end{subfigure}

\vspace{0.8cm}

\begin{subfigure}{\textwidth}
\centering
\includegraphics[width=0.49\linewidth,trim={215bp 260bp 125bp 250bp}]{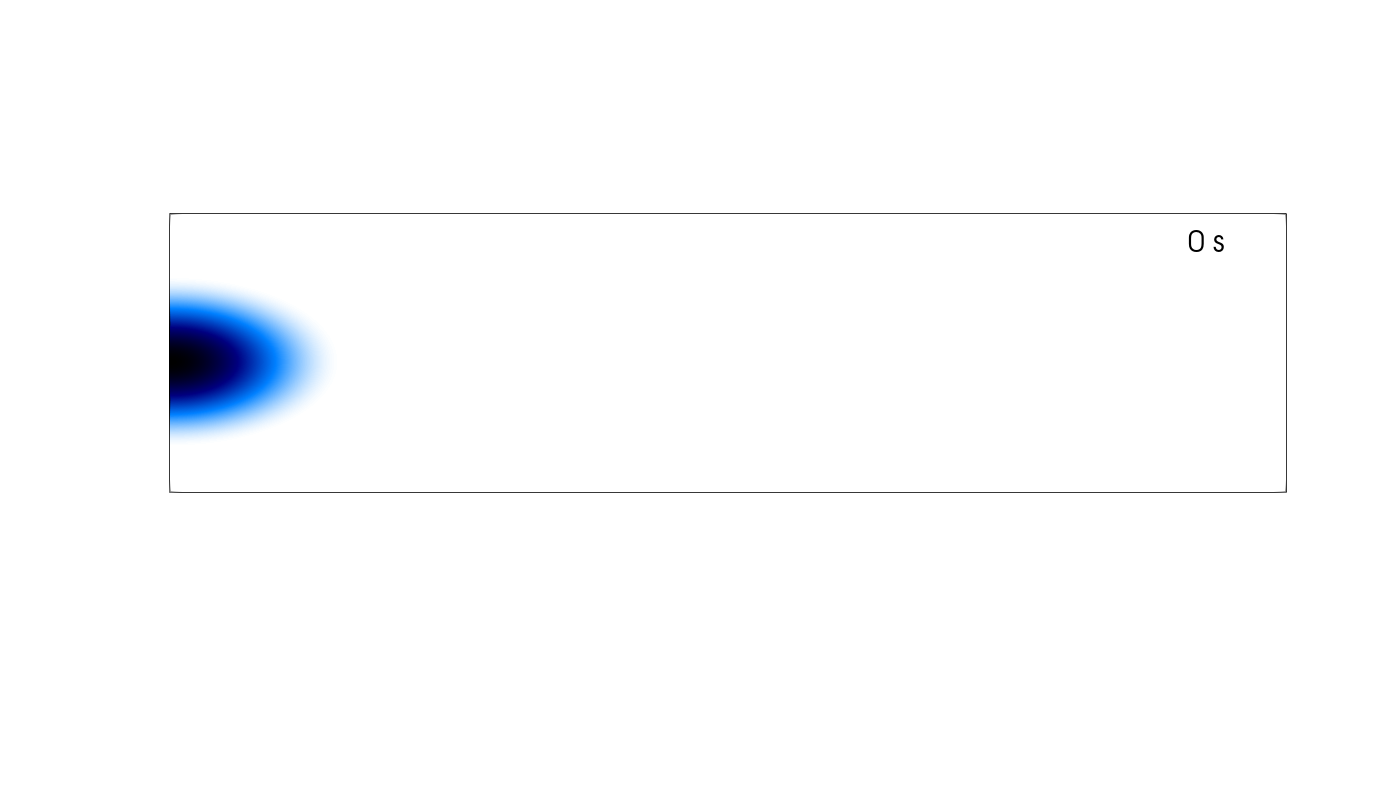}
\hfill
\includegraphics[width=0.49\linewidth,trim={215bp 260bp 125bp 250bp}]{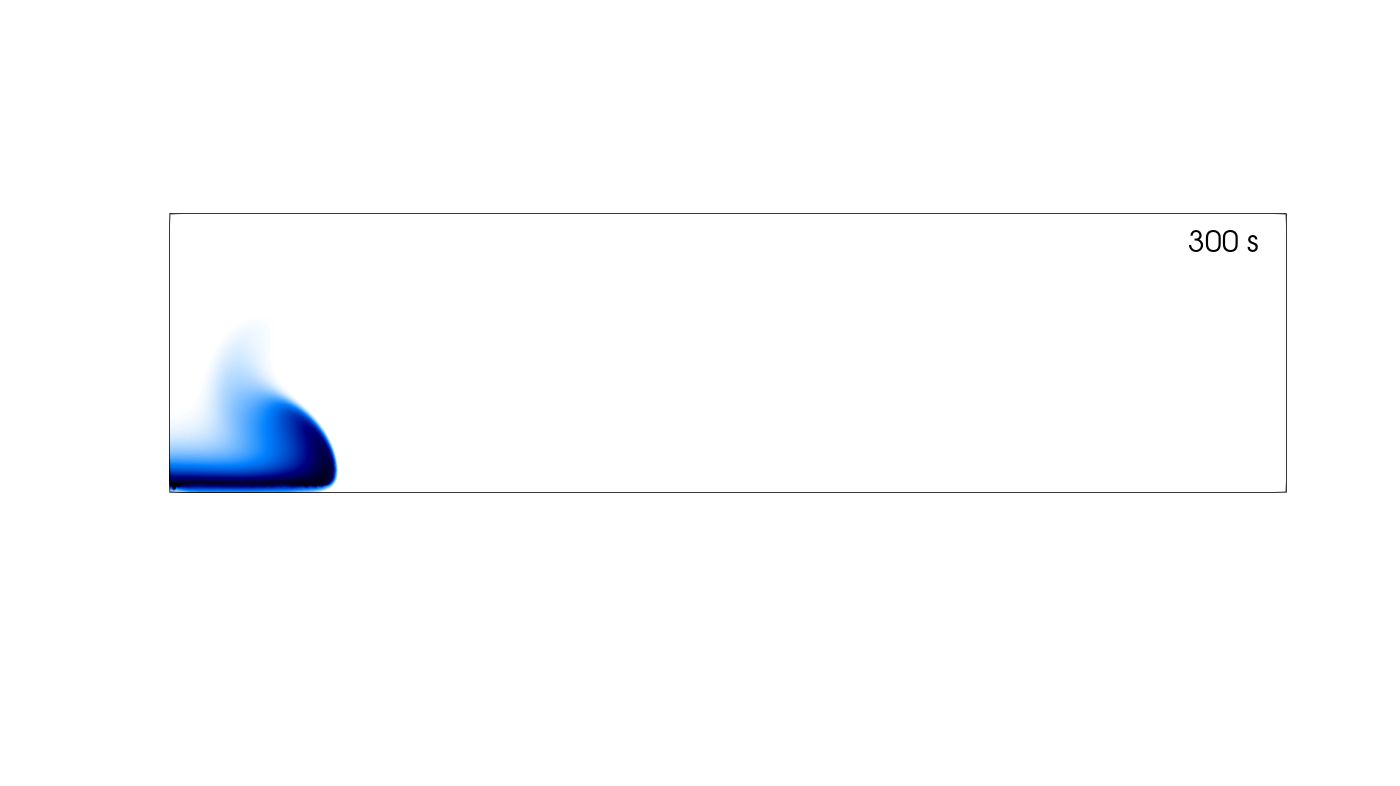}
\newline
\includegraphics[width=0.49\linewidth,trim={215bp 260bp 125bp 250bp}]{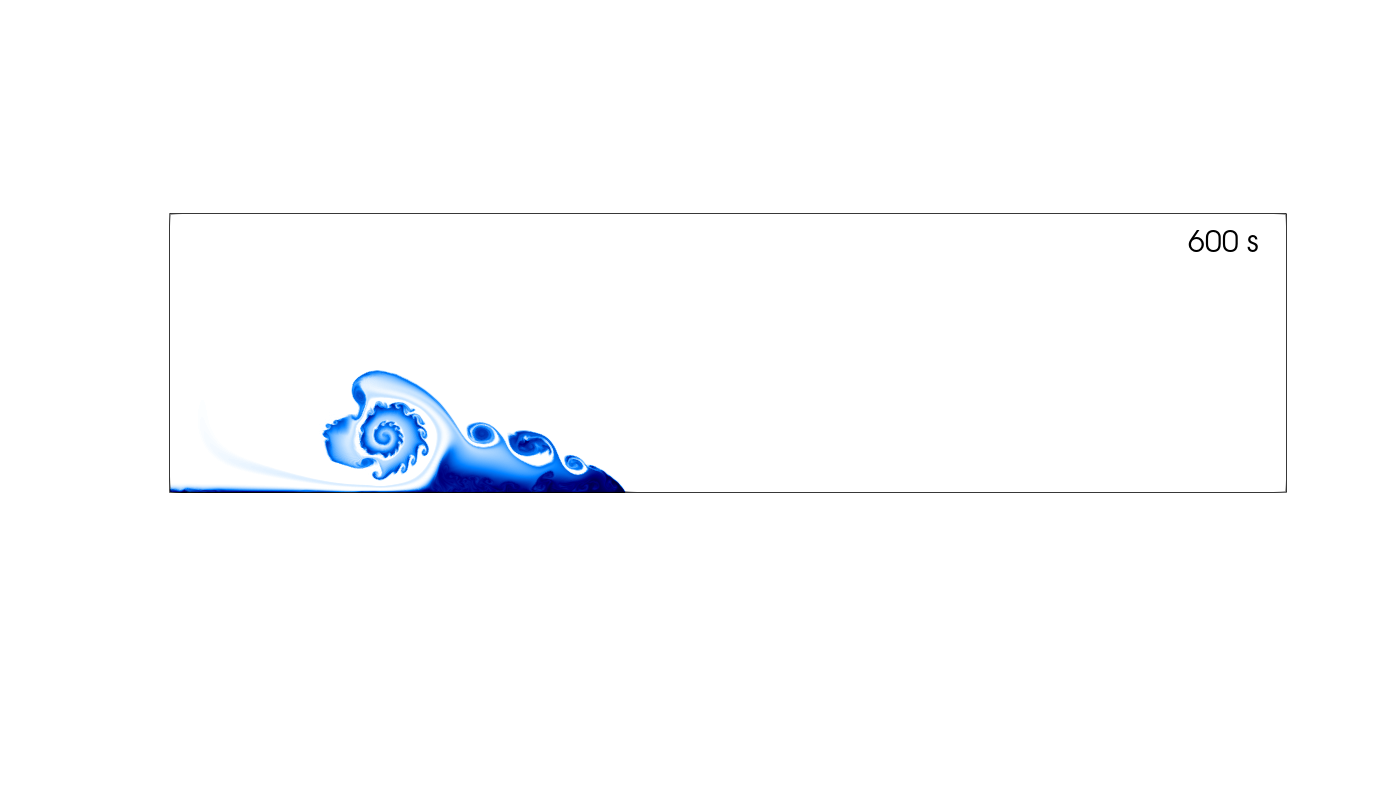}
\hfill
\includegraphics[width=0.49\linewidth,trim={215bp 260bp 125bp 250bp}]{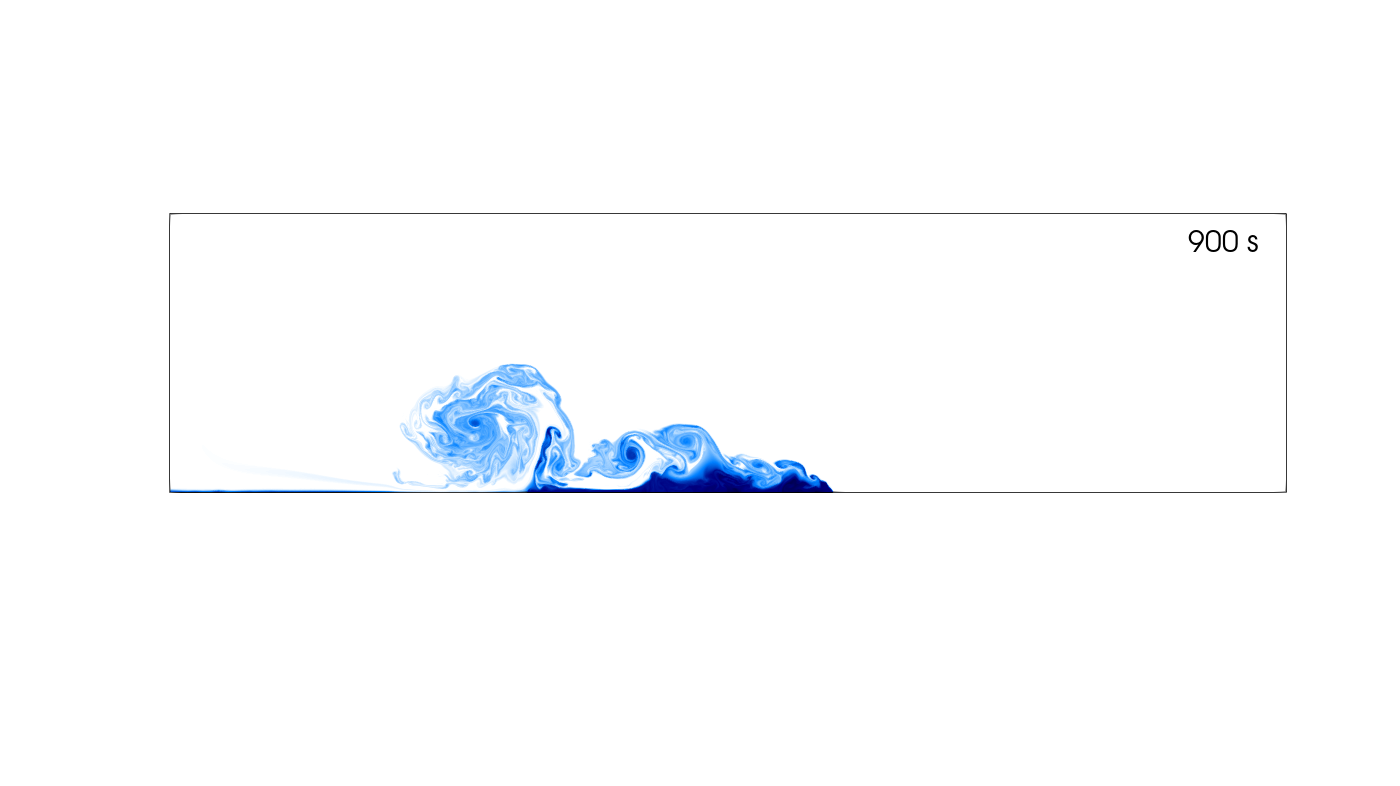}

\caption{$h=6.25$ m}
\label{fig:density_6pt25}
\end{subfigure}

\caption{2D density current over time from $t_0 = 0$ to $t_f=900$ s for different mesh sizes:  (a) $h=50$ m, (b) $h=25$ m, (c) $h = 12.5$ m, and (d) $h=6.25$ m. As the mesh is refined, the current travels slightly faster, causing an increase in the front location.}
\label{fig:density_comparison}
\end{figure}

\subsubsection{2D Sch\"ar Mountain}\label{schaer}
\begingroup
We consider the stratified flow over the multi-peak mountain of Sch\"ar et
al.~\cite{schaer2002terrain} in the setup discussed in~\cite{artiano2025structurepreserving},
\begin{equation*}
  h(x) = h_c \exp\big(-(x/a_c)^2\big)\cos^2(\pi x/\lambda_c),
\end{equation*}
with $h_c = 250$ m, $a_c = 5000$ m, $\lambda_c = 4000$ m, on $D =
\{(x,z) : |x| \le 25\,000,\ h(x) \le z \le 21\,000\}$ m$^2$. The background
state is a uniform flow $u = 10$ m/s with constant Brunt--V\"ais\"al\"a
frequency $N = 0.01$ s$^{-1}$,
\begin{gather*}
  \theta(z) = \theta_0\exp\Big(\frac{N^2}{g}z\Big),\qquad
  \pi(z) = 1 + \frac{g^2}{c_p\theta_0N^2}\Big(\exp\Big(-\frac{N^2}{g}z\Big)-1\Big),
  \\
  p = p_0\pi^{c_p/R},\qquad \varrho = \frac{p}{R\,\theta\,\pi},
\end{gather*}
with $\theta_0 = 280$ K, $p_0 = 10^5$ Pa, $c_p = 1004$, $c_v = 717$, $R = c_p
- c_v$, and $g = 9.81$.
The scheme preserves this background state at rest to machine precision. We
use a uniform mesh with $h = 125$ m ($400 \times 168$ cells) aligned with the
terrain, where cells cut by $h(x)$ are split into triangles. We impose slip
boundary conditions at the bottom and prescribe the background state on the
remaining boundaries; reflections are absorbed by Rayleigh damping
layers~\cite[(172)--(173)]{artiano2025structurepreserving} in the upper $8$ km
and the outer $5$ km on both sides ($\alpha = 0.03$ s$^{-1}$).

Figure~\ref{fig:schaer} shows the velocity perturbations at $t = 5$ h, when
the flow is essentially steady. Compared with the steady linear solution, the
vertical flux of horizontal momentum $M(z) = \int_{|x|\le 20\,\text{km}}
\varrho\,u'w'\,\mathrm{d}x$ attains $0.99$ of its linear value at $z = 1$ km
and $0.91$ at $z = 2$ km, and decreases to $0.64$ at $z = 8$ km due to
numerical dissipation; with $h = 500$ m these values are $0.27$, $0.34$, and
$0.03$.

\begin{figure}
  \centering
  \includegraphics[width=\linewidth]{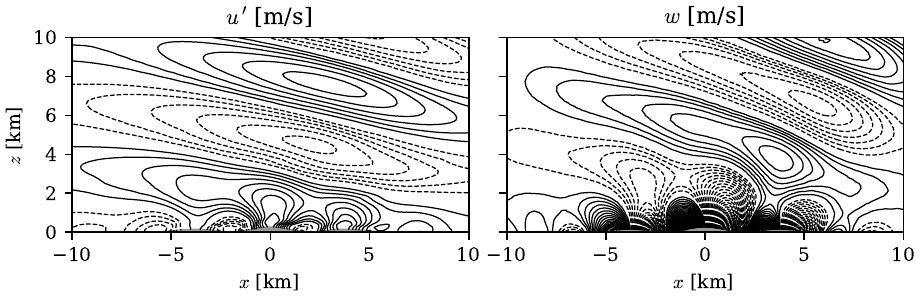}
  \caption{Sch\"ar mountain with $h = 125$ m at $t = 5$ h:
  horizontal velocity perturbation $u' = u - 10$ m/s (left, contour
  interval $0.2$ m/s) and vertical velocity $w$ (right, contour
  interval $0.05$ m/s); negative values dashed.}
  \label{fig:schaer}
\end{figure}
\endgroup
\subsubsection{Blast Over a Mountain}\label{sec:blast}
Our final test models a 2-dimensional ideal blast over a mountain with gravity. Instead of a small perturbation of potential temperature, the inner radius temperature is now much larger causing the blast. The computational domain in km is $D=[-20,20]\times[0,20]$ with a cosine-shaped mountain $y=0.58\,(1+\cos(\pi x/3.4))$ for $|x|\le3.4$.
The mesh was constructed using the \texttt{GMSH} software~\cite{geuzaine2009gmsh} and is composed of triangular elements of size $h=0.1$ km, highlighting the robustness of the method on unstructured grids.
Let the radius be $r=0.75$ km. We set $\mathbf{v}_0(\bx) = 0$, and let $\rho_0(\bx)$ be given by the isentropic profile~\eqref{adiabatic_eq} with $\theta=300$ in the whole domain, so that the initial state is in hydrostatic equilibrium outside the disc only. We impose slip boundary conditions on all sides. Set the potential temperature to be
\begin{equation*}
    \theta_0(\bx) =
    \begin{cases}
        3000 & \text{if } x^2+(y-4)^2 \leq r^2 \\
        300 & \text{if } x^2+(y-4)^2>r^2,
    \end{cases}
\end{equation*}
with thermodynamic constants
\begin{align*}
    p_0 = 0.1, c_v = 717\times10^{-6}, \gamma =1.4, g=0.00981.
\end{align*}
The result for density at time $t_f=30$ s is seen in Figure \ref{fig:mountain_blast}.

\begin{figure}
    \centering
    \includegraphics[width=0.9\linewidth]{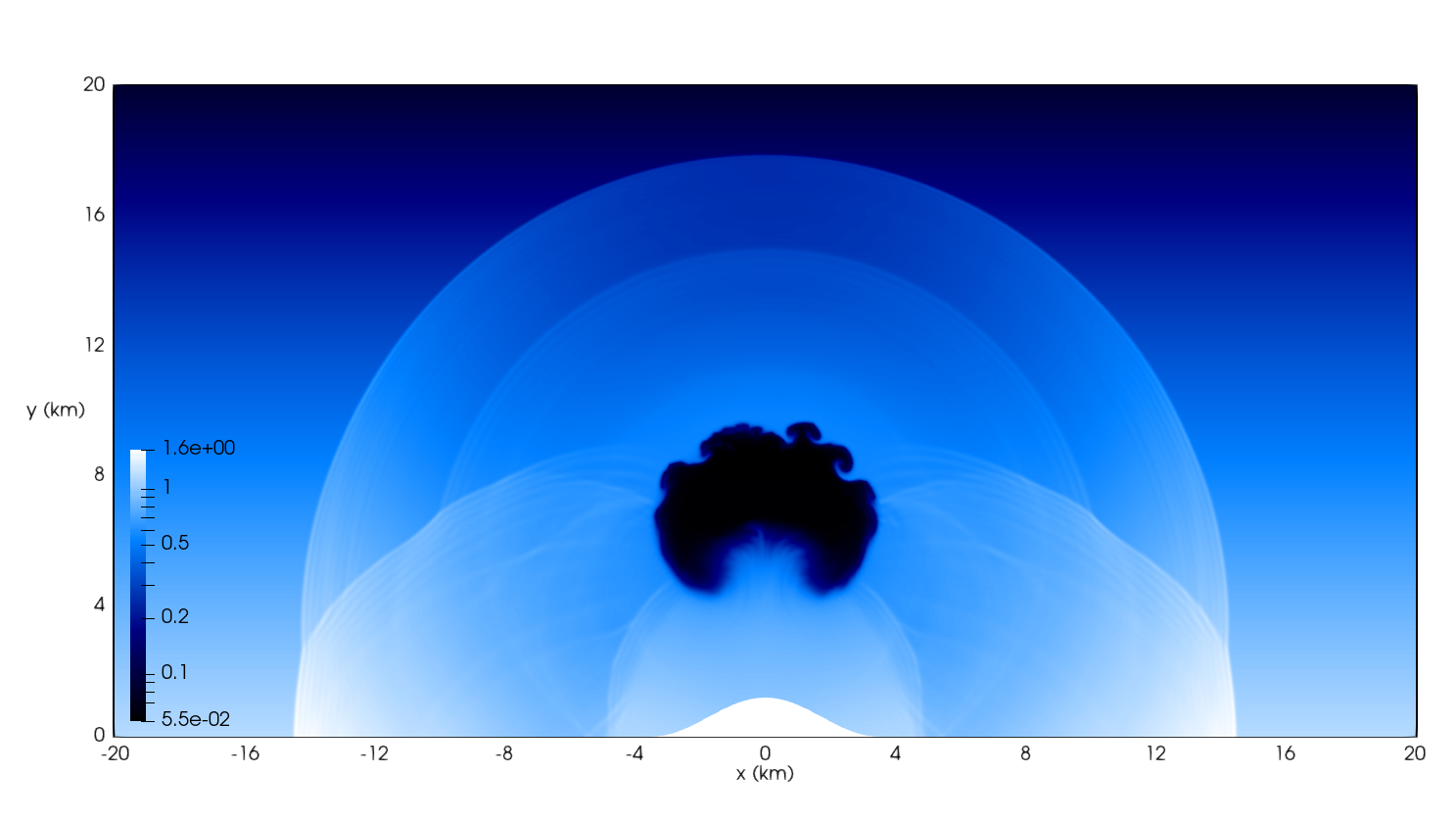}
    \caption{2D blast over a mountain}
    \label{fig:mountain_blast}
\end{figure}

\section{Conclusion}
This work presents a second-order invariant domain preserving and well-balanced scheme for the compressible atmospheric Euler equations. The method preserves positivity of density, potential temperature, and therefore pressure. The method also preserves stratified equilibrium states given by the hydrostatic balance, including isentropic and isothermal ones. We constructed a robust first order approximation using auxiliary states to maintain the well-balancing property. The well-balanced method was then extended to second order using convex limiting to preserve the invariant domain. The numerical illustrations confirm the validity of the method and verify its convergence.

\section*{Reproducibility}
The method described in this paper is implemented in the finite element code
\texttt{ryujin}~\cite{ryujin-2021-1, ryujin-2021-3} and is part of the public
release \textcolor{magenta}{ryujin commit a1cc6d6}, available at
\url{https://github.com/conservation-laws/ryujin}. All parameter files needed
to reproduce the numerical results of Section~\ref{sec:verification}
and Section~\ref{sec:validation} with accompanying documentation are
available in the main repository under the directory
\texttt{prm/reproducibility/euler\_atmospheric/}.

\appendix
 \section{Elementary Wave Structure}\label{sec:wave_structure}
We will now give an overview of the elementary wave structure for this problem. One can show the Jacobian matrix of the one-dimensional system is:
\begin{equation}
\mathbb{A}(\mathbf{u}) =
     \left[{\begin{array}{ccc}
     0 & 1 & 0\\
     -v^2 & 2v& \gamma p(\rho\theta)^{-1} \\
     -v\theta & \theta & v
\end{array}}\right].
\end{equation}
The eigenvalues of this matrix are given by:
\[
\lambda_1(\mathbf{u}) = v - \sqrt{\frac{\gamma p}{\rho}},
\,
\lambda_2(\mathbf{u}) = v,
\,
\lambda_3(\mathbf{u}) = v + \sqrt{\frac{\gamma p}{\rho}},
\]
and the corresponding eigenvectors are given by:
\[
\br_1 = \left({\begin{array}{ccc}
     \frac{1}{\theta} \\
     \frac{\rho v - \sqrt{\gamma\rho p}}{\rho\theta}\\
    1
\end{array}}\right),
\,
\br_2 = \left({\begin{array}{ccc}
     1 \\
     v \\
     0
\end{array}}\right),
\,
\br_3 = \left({\begin{array}{ccc}
     \frac{1}{\theta} \\
     \frac{\rho v + \sqrt{\gamma\rho p}}{\rho\theta}\\
     1
\end{array}}\right).
\]
\begin{lemma}[Wave structure]
    The 1-wave and 3-wave are genuinely nonlinear, and the 2-wave is linearly degenerate.
\end{lemma}

\begin{proof}
    The gradient of $\lambda_1$ and $\lambda_3$ with respect to $\mathbf{u}$ are
    \[
    D_\mathbf{u}\lambda_1(\mathbf{u}) = \left({\begin{array}{c}
    -\frac{v}{\rho} + \frac{1}{2\rho^2}\sqrt{\gamma\rho p}\\
    \frac{1}{\rho}\\
    -\frac{\gamma}{2\rho^2\theta}\sqrt{\gamma\rho p}
\end{array}}\right) \text{ and }
D_\mathbf{u}\lambda_3(\mathbf{u}) = \left({\begin{array}{c}
    -\frac{v}{\rho} - \frac{1}{2\rho^2}\sqrt{\gamma\rho p}\\
    \frac{1}{\rho}\\
    \frac{\gamma}{2\rho^2\theta}\sqrt{\gamma\rho p}
\end{array}}\right).
    \]
    Hence, $D_\mathbf{u}\lambda_1(\mathbf{u})\cdot \mathbf{r_1}
=  -\frac{(\gamma+1)\sqrt{\gamma\rho p}}{2\rho^2\theta} < 0$, and similarly
     $D_\mathbf{u}\lambda_3(\mathbf{u})\cdot \mathbf{r_3}
=  \frac{(\gamma+1)\sqrt{\gamma\rho p}}{2\rho^2\theta} > 0$.
The gradient of $\lambda_2$ with respect to $\mathbf{u}$ is $D_\mathbf{u}\lambda_2(\mathbf{u}) = \left({\begin{array}{c}
    -\frac{v}{\rho}\\
    \frac{1}{\rho}\\
    0
\end{array}}\right)$, and so
$D_\mathbf{u}\lambda_2(\mathbf{u}) \cdot \mathbf{r_2}
= -\frac{v}{\rho} + \frac{v}{\rho} = 0$.
\end{proof}
\begin{lemma}[Potential temperature across shocks/expansions] \label{lem:pt}
   $\theta$ remains constant across shock waves and expansion waves.
\end{lemma}
\begin{proof}
 Suppose first that the 1-wave is a shock wave. We have the following Rankine-Hugoniot condition:
\begin{equation*}
\mathbf{f}(\mathbf{u}_L)\cdot \mathbf{n} - \mathbf{f}(\mathbf{u}_{L^\ast})\cdot \mathbf{n}
= S_L( \mathbf{u}_L - \mathbf{u}_{L^\ast}).
\end{equation*}
Thus, $\rho_{L^\ast}\theta_{L^\ast}(S_L - v_{L^\ast}) = \rho_L\theta_L(S_L - v_L)$. If we let $\hat v_{L^\ast} = S_L - v_{L^\ast}$ and $\hat v_L = S_L - v_L$, by conservation of mass we have that $\rho_{L^\ast} \hat v_{L^\ast}= \rho_L \hat v_L$. Plugging this back into our Rankine-Hugoniot condition for potential temperature we obtain that $\theta_{L^\ast} = \theta_L$.

Suppose now the 1-wave is an expansion wave. Expanding the potential temperature conservation equation and applying conservation of mass we have:
\begin{equation*}
    \rho\partial_t \theta + (\rho v)\cdot\nabla\theta = 0.
\end{equation*}
Thus, $\frac{D\theta}{Dt} := \partial_t \theta + v\cdot\nabla\theta = 0$, and so $\theta$ remains constant across the expansion wave.\\
Similarly, it can be shown that if the 3-wave is an expansion or shock, $\theta$ remains constant across the wave.
\end{proof}

 \section{Solution to the Riemann problem}\label{sec:riemann_problem}

The methodology used for constructing the result in this section is given in Toro \cite{toro2013riemann}, as well as Lax~\cite{Lax_1957} and Godlewski \cite{godlewski2013numerical}.
\begin{proposition}
    The solution to the Riemann problem for pressure in the star region $p_\ast$ is given by the root of the algebraic equation
    \begin{equation}
        \phi(p) = f_L + f_R + v_R - v_L
    \end{equation}
    where $f_Z$, $Z\in\{L,R\}$ is
    \begin{equation}
    f_Z =
    \begin{cases}
    (p - p_Z)\sqrt{\frac{1}{\rho_Z(p-p_Z)}\left(1 - \left(\frac{p_Z}{p}\right)^\frac{1}{\gamma}\right)} &\text{if $p>p_Z$ (shock)} \\
        \frac{2a_Z}{\gamma -1}\bigg[\bigg(\frac{p}{p_Z}\bigg)^\frac{\gamma -1}{2\gamma} - 1\bigg] & \text{if $p \leq p_Z$ (expansion)}.
    \end{cases}
\end{equation}
\end{proposition}

\begin{proof}
    Suppose first that $f_L$ is a shock. Let $\hat{v}_L = v_L - S_L$ and $\hat{v}_\ast = v_\ast - S_L$. We have the following Rankine-Hugoniot conditions:
    \begin{align}
        \rho_L\hat{v}_L &= \rho_{\ast L}\hat{v}_\ast \\
        \rho_L\hat{v}_L^2 + p_L &= \rho_{\ast L}\hat{v}_\ast^2 + p_\ast \label{rh_mom_L}\\
        \rho_L\theta_L\hat{v}_L &= \rho_{\ast L}\theta_{\ast L}\hat{v}_\ast \label{rh_theta}
    \end{align}
    We introduce the mass flux
    \begin{equation}
        Q_L \equiv \rho_L\hat{v}_L = \rho_{\ast L}\hat{v}_\ast, \label{m_flux_L}
    \end{equation}
    and plugging $Q_L$ into \eqref{rh_mom_L} we get
    \begin{align*}
        Q_L = - \frac{p_\ast - p_L}{\hat{v}_\ast - \hat{v}_L}=- \frac{p_\ast - p_L}{v_\ast - v_L},
    \end{align*}
    where we used that $\hat{v}_\ast - \hat{v}_L = v_\ast - S_L - v_L + S_L = v_\ast - v_L$. This in turn gives us the formula for the star-state velocity:
    \begin{equation}
        v_\ast = v_L - \frac{p_\ast - p_L}{Q_L}. \label{v_ast_left}
    \end{equation}
    Also, by \eqref{m_flux_L}
    \begin{equation}
        \hat{v}_L = \frac{Q_L}{\rho_L} \text{ and } \hat{v}_\ast = \frac{Q_L}{\rho_{\ast L}}
    \end{equation}
    Thus,
       \[
        Q_L^2 = - \frac{p_\ast - p_L}{\frac{1}{\rho_{\ast L}} - \frac{1}{\rho_L}}.
    \]
    Note by Lemma~\ref{lem:pt} (equivalently by \eqref{rh_theta}), $\theta_{\ast L} = \theta_L$. Relating $\rho_{\ast L}$ to the density behind the shock $p_\ast$ gives us
    \begin{align*}
        \rho_{\ast L} &= \frac{1}{C_{eos}^\frac{1}{\gamma}\theta_L}p_\ast^\frac{1}{\gamma},
    \end{align*}
    Thus,
    \begin{equation*}
        Q_L^2 = - \frac{p_\ast - p_L}{\frac{C_{eos}^\frac{1}{\gamma}\theta_L}{p_\ast^\frac{1}{\gamma}} - \frac{1}{\rho_L}} = - \frac{p_\ast - p_L}{\frac{C_{eos}^\frac{1}{\gamma}\theta_L\rho_L - p_\ast^\frac{1}{\gamma}}{p_\ast^\frac{1}{\gamma}\rho_L}}.
    \end{equation*}
    Simplifying gives us
    \begin{equation}
        Q_L = \bigg[\frac{(p_\ast - p_L)p_\ast^\frac{1}{\gamma}\rho_L}{p_\ast^\frac{1}{\gamma} - C_{eos}^\frac{1}{\gamma}\theta_L\rho_L}\bigg]^\frac{1}{2}
    \end{equation}
    Plugging into \eqref{v_ast_left} we obtain our desired relation. We can similarly show the relation for $f_R$ in the case of a shock.

Suppose now $f_L$ is an expansion wave. From the equation of state we have
\begin{equation}
    p_L = C_{eos}\theta_L^\gamma\rho_L^\gamma \text{ and } p_\ast = C_{eos}\theta_{\ast L}^\gamma\rho_{\ast L}^\gamma.
\end{equation}
   Since $\theta$ is constant across expansion waves, i.e., $\theta_{\ast L} = \theta_L$, we obtain the relation
    \begin{equation}
        \rho_{\ast L} = \rho_L\bigg(\frac{p_\ast}{p_L}\bigg)^{\frac{1}{\gamma}} \label{exp_relation}
    \end{equation}
    As for isentropic gas dynamics~\cite{toro2013riemann}, the Riemann invariant $v+\frac{2a}{\gamma-1}$ for the left state is constant. Evaluating the constant gives us
    \begin{equation}
        v_L + \frac{2a_L}{\gamma -1} = v_\ast + \frac{2a_{\ast L}}{\gamma-1}
    \end{equation}
Rewriting sound speed in terms of \eqref{exp_relation} gives us
\begin{equation}
    a_{\ast L} = \bigg(\frac{\gamma p_\ast}{\rho_{\ast L}}\bigg)^\frac{1}{2} = a_L\bigg(\frac{{p_\ast}}{p_L}\bigg)^\frac{\gamma - 1}{2\gamma}
\end{equation}
Solving for $v_\ast$ we obtain
\begin{align}
    v_\ast = v_L - \frac{2a_L}{\gamma -1}\bigg[\bigg(\frac{p_\ast}{p_L}\bigg)^\frac{\gamma -1}{2\gamma} - 1\bigg].
\end{align}
Hence,
\begin{equation}
    f_L = \frac{2a_L}{\gamma -1}\bigg[\bigg(\frac{p_\ast}{p_L}\bigg)^\frac{\gamma -1}{2\gamma} - 1\bigg].
\end{equation}
Again, we can similarly obtain the relation for $f_R$ in the case of an expansion.
\end{proof}

\begin{remark}[Pressure Function Behavior]
Let $Z\in \{L,R\}$. A straightforward computation of $f_Z'(p)$ shows that $f_Z'(p)>0$ and hence $\phi'(p) >0$. So, the function is monotone increasing. Similarly, we can see that $f_Z''(p) < 0$, and so $\phi''(p) < 0$. Thus, our pressure function is strictly concave.
\end{remark}
\subsection{Entropy Solution}
Let $\xi:=\frac{x}{t}$ be a self-similar parameter, and let $\mathbf{c}(x,t)$ be the solution to the Riemann problem in primitive variables. The entropy solution to the Riemann problem is
\[\mathbf{c}(x,t)=
\begin{cases}
    c_L & \xi<\lambda_L^-(p_\ast)\\
    c_{LL} & \lambda_L^-(p_\ast)<\xi < \lambda_L^+(p_\ast)\\
    c_L^\ast & \lambda_L^+(p_\ast) < \xi < v_\ast\\
    c_R^\ast & v_\ast < \xi < \lambda_R^-(p_\ast)\\
    c_{RR} & \lambda_R^-(p_\ast) < \xi < \lambda_R^+(p_\ast)\\
    c_R & \xi > \lambda_R^+(p_\ast)
\end{cases}
\]
where
\begin{subequations}
\begin{equation}
    c_{LL}(\xi) = \bigg(\rho_L(\frac{2}{\gamma + 1} + \frac{(\gamma -1)(v_L-\xi)}{(\gamma + 1)a_L})^\frac{2}{\gamma-1},v_L-f_L(p_L(\xi)), p_L(\xi)\bigg)^T
\end{equation}
\begin{equation}
    c_L^\ast = (\rho_{\ast L}, v_\ast, p_\ast)^T,
    \quad \rho_{\ast L} = \rho_L(p_\ast/p_L)^{1/\gamma}
\end{equation}
\begin{equation}
    c_R^\ast = (\rho_{\ast R}, v_\ast, p_\ast)^T,
    \quad \rho_{\ast R} = \rho_R(p_\ast/p_R)^{1/\gamma}
\end{equation}
\begin{equation}
    c_{RR}(\xi) = \bigg(\rho_R(\frac{2}{\gamma + 1} - \frac{(\gamma -1)(v_R-\xi)}{(\gamma + 1)a_R})^\frac{2}{\gamma-1},v_R+f_R(p_R(\xi)), p_R(\xi)\bigg)^T.
\end{equation}
\end{subequations}

\begin{proof}[proof of $c_{LL}$ and $c_{RR}$]
    We first find velocity $v_L(\xi)$. We have the following set of equations:
    \begin{align}
        \xi &= v(\xi) - a(\xi)  \label{eq 1}\\
        v_L + \frac{2a_L}{\gamma-1} &= v_L(\xi) + \frac{2a_L(\xi)}{\gamma-1}. \label{eq 2}
    \end{align}
    Solving for $a$ in \eqref{eq 1} and substituting into \eqref{eq 2} we get
    \begin{equation}
        v_L(\xi) + \frac{2(v_L(\xi)-\xi)}{\gamma-1} = v_L + \frac{2a_L}{\gamma-1}.
    \end{equation}
    Rearranging to solve for $v$ gives us
    \begin{equation}
        v_L(\xi) = \frac{2}{\gamma+1}\left[\frac{\gamma-1}{2}v_L + a_L + \xi\right].
    \end{equation}
    Now, the sound speed in the fan is given by
    \begin{equation}
        a_L(\xi) = \left(\frac{\gamma p_L(\xi)}{\rho_L(\xi)}\right)^\frac{1}{2}
    \end{equation}
    using the definition of sound speed for $a_L$ and the relation \eqref{exp_relation}, we can rewrite the fan sound speed as
    \begin{equation}
        a_L(\xi) = a_L\left(\frac{p_L(\xi)}{p_L}\right)^\frac{\gamma-1}{2\gamma}.
    \end{equation}
    Using \eqref{eq 1} and rearranging we obtain
    \begin{equation}
        p_L(\xi) = p_L\left[\frac{2}{\gamma+1} + \frac{\gamma-1}{(\gamma+1)a_L}(v_L-\xi) \right]^\frac{2\gamma}{\gamma-1}.
    \end{equation}
The fan density follows from plugging in the above pressure to \eqref{exp_relation}:
\begin{equation}
    \rho_L(\xi) = \rho_L\left[\frac{2}{\gamma+1} + \frac{\gamma-1}{(\gamma+1)a_L}(v_L-\xi) \right]^\frac{2}{\gamma-1}.
\end{equation}
\end{proof}

\subsection{Wave-Speeds}
The waves speeds are given by:
\begin{equation}
    \lambda_L^-(p_\ast) = v_L - a_L\sqrt{1 + \frac{1}{\gamma}\max\bigg\{\frac{(p_\ast-(1+\gamma)p_L)p_\ast^\frac{1}{\gamma}+ \gamma p_Lp_L^\frac{1}{\gamma}}{p_L(p_\ast^\frac{1}{\gamma}-p_L^\frac{1}{\gamma})}, 0}\bigg\}
\end{equation}
\begin{equation}
    \lambda_L^+(p_\ast) =
    \begin{cases}
        v_L - f_L(p_\ast) - a_L(\frac{p_\ast}{p_L})^\frac{\gamma-1}{2\gamma} & \text{if $p_\ast < p_L$} \\
        \lambda_L^-(p_\ast) & \text{ if $p_\ast \geq p_L$}
    \end{cases}
\end{equation}
\begin{equation}
    \lambda_R^-(p_\ast) =
    \begin{cases}
        v_R + f_R(p_\ast) + a_R(\frac{p_\ast}{p_R})^\frac{\gamma-1}{2\gamma} & \text{if $p_\ast < p_R$} \\
        \lambda_R^+(p_\ast) & \text{ if $p_\ast \geq p_R$}
    \end{cases}
\end{equation}
\begin{equation}
\lambda_R^+(p_\ast) = v_R + a_R\sqrt{1 + \frac{1}{\gamma}\max\bigg\{\frac{(p_\ast-(1+\gamma)p_R)p_\ast^\frac{1}{\gamma}+ \gamma p_Rp_R^\frac{1}{\gamma}}{p_R(p_\ast^\frac{1}{\gamma}-p_R^\frac{1}{\gamma})}, 0}\bigg\}
\end{equation}
\begin{proof}
    We have
    \begin{equation}
        (v_L - S_L) = \frac{Q_L}{\rho_L}
    \end{equation}
    Solving for $S_L$ we obtain:
    \begin{align*}
        S_L &= v_L - \frac{Q_L}{\rho_L} \\
        &= v_L - a_L\left[1 + \frac{(p_\ast-p_L)p_\ast^\frac{1}{\gamma}}{\gamma p_L(p_\ast^\frac{1}{\gamma} - p_L^\frac{1}{\gamma})} - \frac{\gamma p_L(p_\ast^\frac{1}{\gamma} - p_L^\frac{1}{\gamma})}{\gamma p_L(p_\ast^\frac{1}{\gamma} - p_L^\frac{1}{\gamma})}\right]^\frac{1}{2} .
    \end{align*}
    Combining fractions and simplifying gives the result
    \begin{equation}
        S_L = v_L-a_L\sqrt{1 + \frac{1}{\gamma}\frac{\left[p_\ast^\frac{\gamma+1}{\gamma}-(1+\gamma)p_Lp_\ast^\frac{1}{\gamma}+\gamma p_L^\frac{\gamma+1}{\gamma}\right]}{p_L\big(p_\ast^\frac{1}{\gamma}-p_L^\frac{1}{\gamma}\big)}}.
    \end{equation}
\end{proof}

\bibliographystyle{abbrvnat}
\bibliography{ref}

\end{document}